\documentclass[12pt,a4paper]{amsart}
\usepackage{amssymb,amscd,amsthm,amsmath}
\usepackage{amsfonts}
\usepackage[top=35mm, bottom=35mm, left=30mm, right=30mm]{geometry}
\usepackage[colorlinks=true,citecolor=blue]{hyperref}
\usepackage{mathptmx}
\usepackage{eucal}
\usepackage{graphicx}
\usepackage{mathrsfs}
\usepackage{xcolor}
\usepackage[all]{xy}
\usepackage{bm}
\usepackage[pagewise]{lineno}\nolinenumbers

\newtheorem{thm}{Theorem}[section]
\newtheorem{cor}[thm]{Corollary}
\newtheorem{lem}[thm]{Lemma}
\newtheorem{prop}[thm]{Proposition}

\theoremstyle{definition}
\newtheorem{df}[thm]{Definition}
\newtheorem{ex}[thm]{Example}
\newtheorem{rem}[thm]{Remark}
\newtheorem*{thma}{Theorem A}
\numberwithin{equation}{section}
\def\UU{\mathcal{U}}
\def\VV{\mathcal{V}}
\def\diam{\text{\rm diam}}
\def\Leb{\text{\rm Leb}}

\def\FF{\mathcal{F}}

\def\UU{\mathcal{U}}
\def\VV{\mathcal{V}}
\def\diam{\text{\rm diam}}
\def\Leb{\text{\rm Leb}} 
\def\logf{\log\frac{1}{\epsilon}}

\makeatletter

\newcommand{\Rmnum}[1]{\expandafter\@slowromancap\romannumeral #1@}
\makeatother

\begin{document}

\title{Metric mean dimension and the variational principle for actions of amenable groups}

\author{Rui Yang$^{*,1,2}$  and Xiaoyao Zhou$^{3}$
}
\address
{1.College of Mathematics and Statistics, Chongqing University, Chongqing 401331, P.R.China}
\address{2. Key Laboratory of Nonlinear Analysis and its Applications (Chongqing University), Ministry of Education}
\address
{3.School of Mathematical Sciences and Institute of Mathematics, Ministry of Education Key Laboratory of NSLSCS, Nanjing Normal University, Nanjing, 210023, Jiangsu, P.R.China}

\email{zkyangrui2015@163.com}
\email{zhouxiaoyaodeyouxian@126.com}

\renewcommand{\thefootnote}{}

\footnotetext{*Corresponding author}

 
\renewcommand{\thefootnote}{\arabic{footnote}}
\setcounter{footnote}{0}

\begin{abstract}

Metric mean dimension is a  dynamical counterpart of the box dimension in fractal geometry to  characterize the topological complexity of infinite entropy systems. The classical variational principle states that topological entropy equals the supremum of measure-theoretic entropy over the set of invariant measures. Lindenstrauss and Tsukamoto proved that this variational principle fails for metric mean dimension in terms of rate-distortion dimensions. For the actions of amenable groups, we define a new  measure-theoretic metric mean dimension for invariant measures and establish a classical-type variational principle for metric mean dimension. In particular, we extend the Lindenstrauss-Tsukamoto variational principles to the classical variational principle by defining modified rate-distortion dimensions.

As applications, for systems with zero metric mean dimension, we introduce infinite entropy dimensions in both topological and measure-theoretic settings, and relate them via a variational principle. For systems with positive metric mean dimension, we introduce local metric mean dimension from a local perspective, and relate it to the metric mean dimension of the whole phase space via variational principles.
\end{abstract}


\maketitle
\pagestyle{plain}  
\tableofcontents 
\section{Introduction}

 Measure-theoretic entropy \cite{kol58, s59} and topological entropy \cite{akm65, bow71} are two crucial invariants for characterizing the complexity of dynamical systems. They connect information theory, dimension theory, fractal geometry, and other mathematical concepts. For $\mathbb{Z}$-actions, there is a classical variational principle stating that topological entropy equals the supremum of measure-theoretic entropy over the set of invariant measures (cf. \cite{mis75, w82}). It turns out that this variational principle permits the application of powerful ergodic theorems and related entropy formulas to the study of dynamical systems, particularly to dimension theory and multifractal analysis of dynamical systems \cite{p97}, and reveals chaotic phenomena in the stable and unstable sets of positive-entropy systems \cite{h08}.  Since then, a central question in entropy theory has been whether the classical variational principle exists for entropy-like quantities of general group actions. Beyond $\mathbb{Z}$-actions, it is well-known that the amenability of a group guarantees the existence of invariant probability measures. Ornstein and Weiss’s pioneering work \cite{ow87} laid a solid foundation for the entropy theory of actions of amenable groups. To advance the investigation of the dynamics of amenable group actions, the corresponding variational principle between amenable topological entropy and measure-theoretic entropy was established in \cite{st80, op82}. Moreover, Lindenstrauss \cite{lin01} extended several powerful theorems in ergodic theory, such as the Pointwise Ergodic Theorem and the Shannon–McMillan–Breiman Theorem, from $\mathbb{Z}$-actions to amenable groups. See also \cite{hyz11, ly12, dz15} for local variational principles of amenable topological entropy-like quantities and \cite{kl11} for variational principles of actions of sofic groups.

Notice that topological entropy no longer provides additional information about the dynamics of systems with infinite entropy. To characterize the complexity of such infinite entropy systems, Gromov \cite{gromov} first introduced a new topological invariant called mean dimension. In 2000, inspired by the definition of box dimension in fractal geometry, Lindenstrauss and Weiss \cite{lw00} introduced metric mean dimension to quantify the divergent rate at which the $\epsilon$-topological entropy diverges to $\infty$ as $\epsilon \to 0$. Mean dimensions have found applications in solving embedding problems of dynamical systems \cite{l99, lw00, g15, glt16, g17, gt20} and in the analog compression framework of information theory \cite{gs,y25}. For progress on the mean dimension theory of the actions of amenable groups and sofic groups, one can refer to \cite{l13, ll19, cdz22, gg25}.

The aforementioned work motivates us to ask the following question:

\emph{Question: For the actions of amenable groups,  can we define a  measure-theoretic counterpart to metric mean dimension such that the   classical-type variational principle holds for the metric mean dimension in terms of this measure-theoretic metric mean dimension?}

Even for $\mathbb{Z}$-actions, the above question remains highly challenging and unsolved, primarily due to the lack of a measure-theoretic metric mean dimension for invariant measures. In 2018, for $\mathbb{Z}$-actions Lindenstrauss and Tsukamoto \cite{lt18} introduced $L^p$ ($1 \leq p < \infty$) and $L^{\infty}$-rate-distortion functions from information theory \cite{ct06, gra 11}, and first established the \emph{analogous} variational principles for upper metric mean dimension:
\begin{thma}
Let $T:X\rightarrow X$ be a homeomorphism map on a compact metric space $(X,d)$. Then 
$${\overline{\rm  mdim}}_M(T,X,d)=\limsup_{\epsilon \to 0}\frac{1}{\logf}\sup_{\mu \in M(X,T)}R_{\mu, L^{\infty}}(\epsilon).$$
Additionally, if $(X,d)$ has  the tame growth of  covering numbers\footnote[1]{That is, $\lim_{\epsilon \to 0}\epsilon^{\theta}\log r(X,d,\epsilon)=0$ for each $\theta>0$, where $r(X,d,\epsilon)$ denotes the smallest cardinality of  $(d,\epsilon)$-spanning sets of $X$. This definition is purely geometric and is independent of the dynamics. It is a mild condition  since every compact metrizable space admits a metric satisfying the tame growth of covering numbers \cite[Lemma 4]{lt18}.}, then for any $1\leq p<\infty$, 
$${\overline{\rm  mdim}}_M(T,X,d)=\limsup_{\epsilon \to 0}\frac{1}{\logf}\sup_{\mu \in M(X,T)}R_{\mu, L^{p}}(\epsilon),$$
where  ${ \overline{\rm  mdim}}_M(T,X,d)$ denotes the upper metric mean dimension of $X$,  and $R_{\mu, L^{\infty}}(\epsilon), R_{\mu, L^{p}}(\epsilon)$ are respectively the $L^{\infty}$, $L^p$ rate-distortion functions of $\mu$, and $M(X,T)$ denotes the set of $T$-invariant Borel probability measures on $X$.
\end{thma}

Comparing Lindenstrauss-Tsukamoto’s variational principles with the classical ones, exchanging the order of $\sup_{\mu \in \mathcal{M}(X,T)}$ and $\limsup_{\epsilon \to 0}$ is doomed to \emph{fail}, since the counterexample in \cite[Section VIII]{lt18} shows that the supremum of rate-distortion dimensions over the set of invariant measures is strictly less than the metric mean dimension of $X$. Later, Chen, Dou, and Zheng \cite{cdz22} extended Lindenstrauss-Tsukamoto’s variational principles to actions of amenable groups. Furthermore, by replacing the rate-distortion functions, the authors in \cite{vv17, gs20, shi} proved that Lindenstrauss-Tsukamoto’s variational principles still hold for Katok’s $\epsilon$-entropy, Kolmogorov-Sinai $\epsilon$-entropy, Brin-Katok $\epsilon$-entropy, and other measure-theoretic $\epsilon$-entropies. Readers can refer to \cite{ycz25} for a systematic study of these variational principles and a comparison of the divergent rates of these measure-theoretic $\epsilon$-entropies.


The goal of this paper is to give a positive answer to the aforementioned question within the framework of amenable groups. We first collect several candidates for  measure-theoretic $\epsilon$-entropies and calculate their divergent rates. Partially inspired by the ideas in \cite{ycz22b, cpv24, ycz25}, we define an upper measure-theoretic metric mean dimension for invariant measures, and establish the following classical-type variational principle for  upper metric mean dimension:

\begin{thm}\label{thm 1.1}
Let $(X, d, G)$ be a $G$-system, and   let $\{F_n\}$ be a tempered  F\o lner sequence of $G$.  Then 
\begin{align*}
{\overline{\rm  mdim}}_M(G,X,d)&=\max_{\mu \in M(X,G)}\overline{\rm {mdim}}_M(\mu,\{F_n\},d),
\end{align*}
where  ${\overline{\rm  mdim}}_M(G,X,d)$ denotes the   upper  metric mean dimension of $X$, $\overline{\rm {mdim}}_M(\mu,\{F_n\},d)$ denotes the upper  measure-theoretic  metric mean dimension of $\mu$ (see Definition \ref{df 3.26}), and $ M(X,G)$ is the set of  $G$-invariant Borel probability measures on $X$.
\end{thm}

Based on the idea of the formulation of  measure-theoretic metric mean dimension, we define the modified $L^p$ ($1\leq p<\infty$) and $L^{\infty}$ rate-distortion dimensions for invariant measures, and establish  an alternative Lindenstrauss-Tsukamoto variational principle for  upper metric mean dimension:
\begin{thm}\label{thm 1.2}
Let $(X, d, G)$ be a $G$-system,  and let  $\{F_n\}$ be a tempered F\o lner sequence of $G$.  Then
\begin{align*}
{\overline{\rm  mdim}}_M(G,X,d)&=\max_{\mu \in M(X,G)}\overline{\rm {mdim}}_{M,L^{\infty}}(\mu,\{F_n\},d).
\end{align*}
If  $(X,d)$ has the tame growth of covering numbers, then for any $1\leq p<\infty$,
$${\overline{\rm  mdim}}_M(G,X,d)=\max_{\mu \in M(X,G)}\overline{\rm {mdim}}_{M,L^p}(\mu,\{F_n\},d),$$
where $\overline{\rm {mdim}}_{M,L^{p}}(\mu,\{F_n\},d)$ and $ \overline{\rm {mdim}}_{M,L^{\infty}}(\mu,\{F_n\},d)$  respectively denote the modified $L^p$ and $L^{\infty}$ rate-distortion  dimensions of $\mu$ (see Definition \ref{df 3.27}). 
\end{thm}
We explore two applications based on Theorem \ref{thm 1.1}. The first one concerns how to characterize the topological complexity  of $G$-systems with zero metric mean dimension.  Two vital quantities, which have been introduced to characterize the complexity of zero entropy systems, are the topological and measure-theoretic entropy dimensions \cite{c97}. However, the classical variational principle fails for topological entropy dimension in terms of measure-theoretic entropy dimension \cite{adp10}. Analogous to Carvalho’s work \cite{c97} on zero entropy systems,  for infinite-entropy $G$-systems with zero metric mean dimension, we introduce topological and measure-theoretic versions of infinite entropy dimensions, and derive several elementary properties of infinite entropy dimensions. In particular, for any $s \in (0,1)$, we construct an infinite-entropy $G$-system with  zero metric mean dimension, while its infinite entropy dimension equals $s$. Besides, a variational principle between topological and measure-theoretic infinite entropy dimensions is established in terms of \emph{invariant measures}.

\begin{thm}\label{thm 1.3}
Let $(X, d, G)$ be a $G$-system, and  $\{F_n\}$ be a tempered  F\o lner sequence  of $G$. Then
$$\overline{D}_M(G,X,d)=\max_{\mu \in M(X,G)}\overline{D}_M(\mu,\{F_n\},d),$$
where $\overline{D}_M(G,X,d)$   and $\overline{D}_M(\mu,\{F_n\},d)$ denote the infinite upper entropy dimensions of $X$ and $\mu$, respectively. See Subsection \ref{subsec 5.1} for their precise definitions.
\end{thm}

The second one concerns the local structure of $G$-systems with positive metric mean dimension. Inspired by the concept of entropy point \cite{yz07, cpv24}, from a local perspective we define  local metric mean dimension, and establish variational principles for  upper metric mean dimension in terms of local metric mean dimension.  Then we use the variational principles to  investigate the local structures of $G$-systems with positive metric mean dimension by introducing  the (full) metric mean dimension  point.


\begin{thm}\label{thm 1.4}
Let $(X, d, G)$ be a $G$-system.  Then for every tempered  two-sided F\o lner sequence  $\{F_n\}$ of $G$,
\begin{align*}
\overline{\rm mdim}_M(G,X,d)
&=\max_{x\in X}  \overline{\rm mdim}_M(x,\{F_n\},d),\\
&=\max_{\mu \in M(X,G)}\int \overline{\rm mdim}_M(x,\{F_n\},d) d\mu,
\end{align*}
where $\overline{\rm mdim}_M(x,\{F_n\},d)$ denotes the  local upper  metric mean  dimension of $x$. See Subsection \ref{subsec 5.2} for the precise definition.
\end{thm}


The rest of this paper is organized as follows: In Section \ref{sec 2}, we review the basic setting and relevant concepts of $G$-systems. In Section \ref{sec 3}, we introduce metric mean dimension on subsets, define measure-theoretic metric mean dimension for invariant measures, and derive several fundamental properties for them. In Section \ref{sec 4}, we prove the main Theorems \ref{thm 1.1} and \ref{thm 1.2}. In Section \ref{sec 5}, we apply the variational principles to study infinite-entropy systems with zero and positive metric mean dimension, and prove Theorems \ref{thm 1.3} and \ref{thm 1.4}.

\section{Preliminaries}\label{sec 2}


\subsection{The setup of $G$-systems}
We first recall several fundamental facts about amenable groups. Let $G$ be a group. The cardinality of a subset $F$ of $G$ is denoted by $|F|$. Denote the symmetric difference  of two subsets  $E, F$ of $G$ by $E\Delta F:=(E\backslash F)\cup (F\backslash E)$.  We write  $\FF(G)$  to denote  the set of all  non-empty  finite subsets of $G$.  A countable discrete  group $G$ is  \emph{amenable} if there exists a sequence  $\{F_n\}$ of non-empty finite subsets in $ \FF(G)$  such that for all $g\in G$,
$$\lim_{n\to\infty}\frac{|gF_n\triangle F_n|}{|F_n|}=0.$$

Such a sequence $\{F_n\}$  is called a   \emph{F\o lner sequence} of $G$. Additionally, the  F\o lner sequence $\{F_n\}$  is said to be a two-sided F\o lner sequence if  $\lim_{n\to\infty}\frac{|F_ng\triangle F_n|}{|F_n|}=0$ for all $g\in G$. Examples of amenable groups  include  all finite groups, abelian groups, and  all finitely generated groups of sub-exponential growth. The free group of rank 2 is a non-amenable group.   A  F\o lner sequence  $\{F_n\}$ of $G$  is  \emph{tempered} if there  exists a constant $C>0$ (independent of $n$)  such that for all $n\geq 1$, 
$$|\bigcup_{j<n}F_j^{-1}F_n|\leq C\cdot |F_n|.$$

This condition is \emph{mild} since  every  F\o lner sequence  $\{F_n\}$ of $G$ admits a tempered  subsequence  \cite[Proposition 1.4]{lin01}. 

Throughout this paper, we always assume that  $G$ is a countably infinite discrete  amenable group.

 A \emph{$G$-system} is a pair $(X,G)$, where $X$ is a compact metrizable topological space  and $\Gamma: G\times X \rightarrow X$ is a continuous map satisfying: 
\begin{enumerate}
\item$\Gamma(e_G,x)=x$ for every $x\in X$, where $e_G$ is the identity element of  $G$; 
\item  $\Gamma(g_1,\Gamma(g_2,x))=\Gamma(g_1g_2,x)$ for every $g_1,g_2\in G$ and $x\in X$.
\end{enumerate}
For notational simplicity, we  write $gx:=\Gamma(g,x)$(or $T_gx:=\Gamma(g,x)$). Furthermore, to  emphasize the metric $d$ on $X$,  we sometimes use $(X,d,G)$ to denote a $G$-system.

\subsection{Topological entropy and its variational principle}

Next  we review the definitions of topological entropy and  measure-theoretic entropy, and present the variational principle for topological entropy within the framework of  $G$-systems.

Let $(X,d)$ be a compact metric space. Denote by $\mathcal{C}_X, \mathcal{P}_{X},  \mathcal{C}_{X}^o$ the collections of finite Borel   covers of $X$,  finite Borel  partitions of $X$ and  finite open covers of $X$, respectively.  Given $\alpha \in \mathcal{C}_{X}$, the \emph{diameter} of $\alpha$, denoted by $\diam (\alpha)$,  is the maximum of the diameters of the elements  of $\alpha$  with respect to (w.r.t.) $d$. The \emph{Lebesgue number} of  $\mathcal{U}\in \mathcal{C}_{X}^o$, denoted by $\Leb (\UU)$, is the largest positive number $\delta>0$ such that each  $d$-open ball $B_d(x,\delta)$ of $X$ is contained in some element of $\UU$. A fundamental and useful fact is that for every $\epsilon >0$,  if we consider the family $\UU=\{B_d(x,\frac{\epsilon}{2}):x\in E\}$, where $E$ is a finite $\frac{\epsilon}{4}$-net of $X$, then $\UU$ is an open cover of $X$ with $\diam(\UU)\leq \epsilon$ and $\Leb(\UU)\geq \frac{\epsilon}{4}$ (cf. \cite[Lemma 3.4]{gs20}). For $F\in \FF(G)$,  the join of $\alpha \in  \mathcal{C}_X$  w.r.t. $F$ is the cover   $\alpha_F:=\vee_{g\in F}g^{-1}\alpha$, where $g^{-1}\alpha=\{g^{-1}A: A\in \alpha\}$. 

Let $\UU \in \mathcal{C}_X^o$. We define $N(\UU)$ as  the minimal cardinality of a subcover of $\UU$ covering $X$. Recall that
a set function $h:\FF(G)\rightarrow \mathbb{R}$  is said  to be \\
(1)  monotone if $h(E)\le h(F)$ for  $\forall E,F\in\FF(G)$ with $E\subset F$;\\
(2) $G$-invariant if $h(Eg)= h(E)$  for $\forall g\in G$, $\forall E\in\FF(G)$; \\
(3) sub-additive if $h(E\cup F)\le h(E)+h(F)$   for any disjoint $E,F\in\FF(G)$.

The following lemma is the well-known Ornstein-Weiss Theorem, which  allows  us to determine whether amenable entropy-like quantities  depend on the  F\o lner sequences  of amenable groups. 
\begin{lem}\cite{ow87,gromov} (See also \cite[Theorem 6.1]{lw00})\label{lem 2.1}
Let $G$ be an amenable group. Let  $h:\FF(G)\rightarrow \mathbb{R}$  be a monotone $G$-invariant sub-additive function. Then there exists $\lambda\in [-\infty,\infty)$ depending on $G$ and $h$  such  that
$$\lim_{n\to\infty}\frac{h(F_n)}{|F_n|}=\lambda$$
for all F\o lner sequences $\{F_n\}_{n\ge1}$  of $G$.
\end{lem}
One can check  that the function
$F\in \FF(G)\mapsto  \log N(\UU_F)$ is a monotone, $G$-invariant and sub-additive  non-negative function, and hence satisfies   the conditions of  Ornstein-Weiss Theorem. Hence, by  Lemma \ref{lem 2.1}  the limit
\begin{align}\label{equu 3.1}
h_{top}(G,X,\UU):=\lim_{n \to \infty}\frac{\log N(\UU_{F_n})}{|F_n|}
\end{align}
exists, and is independent of the choice  of   F\o lner sequence  $\{F_n\}$  of $G$. 

The \emph{topological entropy of $X$} is defined by 
$$h_{top}(G,X)=\sup_{\UU \in \mathcal{C}_X^o} h_{top}(G,X,\UU).$$

By $M(X)$ we denote the  set of  all Borel probability measures on $X$.  The  $G$-invariant and  $G$-ergodic Borel probability measures on $X$  are  respectively denoted by  $M(X,G)$ and $E(X,G)$. Endow $M(X)$ with the weak*-topology. A standard fact is that both  the  sets $M(X)$ and $M(X,G)$  are  compact convex sets in the weak$^{*}$-topology. Furthermore, the amenability of $G$ guarantees that $M(X,G)$ is non-empty.

Let $\mu \in M(X,G)$ and  $\alpha  \in  \mathcal{P}_X$.
The \emph{measure-theoretic  entropy of $\mu$ w.r.t. $\alpha$}  is given by
$$h_\mu(G,\alpha)=\lim_{n\to \infty}\frac{1}{|F_n|}H_\mu(\alpha_{F_n}),$$
where $H_{\mu}(\alpha)$ denotes the usual partition entropy of $\alpha$. The limit exists and is independent of the choice of F\o lner sequence of $G$ since the function
$F\in \mathcal{F}(G)\mapsto H_\mu(\alpha_{F})$ satisfies the conditions of  the Ornstein-Weiss Theorem.  Moreover, by  \cite[Lemma 3.1, (4)]{hyz11}, the limit can be replaced by the infimum, i.e., $h_\mu(G,\alpha)=\inf_{F\in \mathcal{F}(G)} \frac{1}{|F|}H_\mu(\alpha_{F})$.

We define the \emph{measure-theoretic entropy of $\mu$} as
$$h_{\mu}(G,X)=\sup_{\alpha  \in  \mathcal{P}_X}h_\mu(G,\alpha).$$

The classical variational principle \cite{st80, op82} (see also \cite[Theorem 9.48, p.226]{kl16}) relates the topological entropy with measure-theoretic entropy of $G$-systems:
\begin{align*}
h_{top}(G,X)=\sup_{\mu \in M(X,G)}h_{\mu}(G,X).
\end{align*}

\section{Metric mean dimension of   infinite  entropy systems}\label{sec 3}


\subsection{Metric mean dimension on subsets}

In this subsection, we introduce the metric mean dimension on subsets using spanning sets, separated sets and open covers, and derive several fundamental properties for metric mean dimension.

\subsubsection{Definitions of metric mean dimension on subsets} 

Let $(X,d, G)$ be a $G$-system. Given a non-empty subset $Z \subset X$ and $\epsilon>0$, a set $E \subset X$ is a  \emph{$(d,\epsilon)$-spanning set} of $Z$ if for any $x\in Z$, there exists $y \in E$ such that $d(x,y)\leq\epsilon$; a subset $S \subset Z$  is a  \emph{$(d,\epsilon)$-separated set} of $Z$ if for any distinct $x,y \in S$, one has $d(x,y)> \epsilon$. Denote  the  smallest cardinality of $(d,\epsilon)$-spanning sets of $Z$ by  $r(Z,d,\epsilon)$, and   the largest cardinality of $(d,\epsilon)$-separated sets of $Z$ by $s(Z,d,\epsilon)$.  The compactness of $X$ ensures  $r(Z,d,\epsilon)\leq s(Z,d,\epsilon)\leq r(Z,d,\frac{\epsilon}{2})<\infty$ for every $\epsilon >0$.

 The  metric mean dimension of  $X$  was introduced  in \cite{cdz22} using spanning sets and separated sets. We  extend this concept to  any  non-empty subset.  Let $\epsilon>0$ and $\{F_n\}$ be a F\o lner sequence of $G$.  We define the   $\epsilon$-topological entropy of $Z$ w.r.t. $\{F_n\}$ as 
 $$h_{top}(G,Z,d, \{F_n\}, \epsilon):=\limsup_{n \to \infty}\frac{\log s(Z,d_{F_n},\epsilon)}{|F_n|},$$
 where  the Bowen metric  $d_F$  on $X$ w.r.t.  $F\in \FF(G)$ is  given by $$d_F(x,y)=\max_{g\in F}d(gx,gy).$$ 
 
\begin{df}\label{df 3.1}
The   \emph{upper and lower metric mean dimensions of $Z$} are  respectively defined   by 
\begin{align*}
{\overline{\rm  mdim}}_M(G,Z,\{F_n\},d)&:=\limsup_{\epsilon \to 0}\frac{h_{top}(G,Z,d, \{F_n\}, \epsilon)}{\logf},\\
{\underline{\rm  mdim}}_M(G,Z,\{F_n\},d)&:=\liminf_{\epsilon \to 0}\frac{h_{top}(G,Z,d, \{F_n\}, \epsilon)}{\logf}.
\end{align*}
\end{df} 
 
If $Z=\emptyset$, we let ${\overline{\rm  mdim}}_M(G,Z,\{F_n\},d)={\underline{\rm  mdim}}_M(G,Z,\{F_n\},d):=0$.  Definition \ref{df 3.1} can be equivalently defined using spanning sets.
\begin{prop}\label{prop 3.1}
Let $(X, d, G)$ be a $G$-system, and let $Z$ be  a non-empty subset of  $X$. Then for every  F\o lner sequence  $\{F_n\}$ of $G$,
\begin{align*}
{\overline{\rm  mdim}}_M(G,Z,\{F_n\},d)&=\limsup_{\epsilon \to 0}\frac{\hat{h}_{top}(G,Z,d, \{F_n\}, \epsilon)}{\logf},\\
{\underline{\rm  mdim}}_M(G,Z,\{F_n\},d)&=\liminf_{\epsilon \to 0}\frac{\hat{h}_{top}(G,Z,d, \{F_n\}, \epsilon)}{\logf},
\end{align*}
where  $\hat{h}_{top}(G,Z,d,\{F_n\}, \epsilon):=\limsup_{n \to \infty}\frac{\log r(Z, d_{F_n}, \epsilon)}{|F_n|}$.
\end{prop}
 
To clarify the dependence of   metric mean dimension of $Z$ on F\o lner sequences of $G$, we define metric mean dimension on subsets using  open covers. Let $\UU \in \mathcal{C}_X^o$ and $Z$ be a non-empty subset of $X$. We put
$$N(\UU,Z):=\min\{\#\mathcal{V}: \mathcal{V}\subset \mathcal{U}~\text{and}~Z\subset \cup_{V\in \mathcal{V}}V\}.$$ 
If $Z$ is a $G$-invariant subset of $X$ (i.e., $gZ=Z$ for all $g\in G$), then  the function
$F\in \FF(G)\mapsto  \log N(\UU_F,Z)$ is a monotone, $G$-invariant and sub-additive  non-negative function.  Then the Ornstein-Weiss Theorem  guarantees  that  the limit
\begin{align}\label{equu 3.1}
h_{top}(G,Z,\UU):=\lim_{n \to \infty}\frac{\log N(\UU_{F_n},Z)}{|F_n|}
\end{align}
exists, and is independent of the choice of  F\o lner sequence  $\{F_n\}$  of $G$.

 \begin{prop} \label{prop 3.4}
Let $(X, d, G)$ be a $G$-system, and let $Z \subset X$ be a non-empty $G$-invariant subset. Then for every F\o lner  sequence $\{F_n\}$ of $G$,
\begin{align*}
{\overline{\rm  mdim}}_M(G,Z,\{F_n\},d)&=\limsup_{\epsilon \to 0}\frac{1}{\logf}\inf_{\diam (\UU) \leq \epsilon}h_{top}(G,Z,\UU),\\
{\underline{\rm  mdim}}_M(G,Z,\{F_n\},d)&=\liminf_{\epsilon \to 0}\frac{1}{\logf}\inf_{\diam (\UU) \leq \epsilon}h_{top}(G,Z,\UU).
\end{align*}
Moreover,  $\inf_{\diam (\UU) \leq \epsilon}\limits$ can be replaced by $\sup_{\diam (\UU) \leq \epsilon, \Leb (\UU)\geq \frac{\epsilon}{4}}\limits$ and $\sup_{\Leb (\UU)\geq \frac{\epsilon}{4}}\limits$.
 \end{prop}
 
 \begin{proof}
 Fix a  F\o lner sequence $\{F_n\}$ of $G$, and let $\epsilon >0$ and $F \in \FF(G)$.
 
  Choose  $\UU \in \mathcal{C}_X^o$  with $\diam (\UU) \leq \epsilon$ and $\Leb (\UU)\geq \frac{\epsilon}{4}$. Assume that $E$ is a $(d_F,\frac{\epsilon}{4})$-spanning  set of $Z$ with the smallest cardinality. Then for each $x \in E$, the Bowen ball $B_F(x,\frac{\epsilon}{4}):=\{y\in X:d_F(x,y)<\frac{\epsilon}{4}\}$ is contained an element  $U_x$ of $\mathcal{U}_F$. This yields that $N(\UU_{F_n},Z)\leq  r(Z,d_{F_n},\frac{\epsilon}{4})$ for every $n\geq 1$, and hence
\begin{align}\label{equu 3.2}
\inf_{\diam (\UU) \leq \epsilon}h_{top}(G,Z,\UU)\leq  \hat{h}_{top}(G,Z,d, \{F_n\}, \frac{\epsilon}{4}).
\end{align}

On the other hand,  now we fix $\UU \in \mathcal{C}_X^o$ with $\diam (\UU) \leq \epsilon$. Let $\mathcal{V}$ be  a subcover of $\UU_F$ covering $Z$ with the smallest cardinality. Take a point $x_V$ for each $V\in  \mathcal{V}$, and denote the set of these points  by $E$. Then $E$ is a $(d_F, 2\epsilon)$-spanning set of $Z$. This implies that $r(Z,d_{F_n},2\epsilon)\leq N(\UU_{F_n},Z)$ for every $n\geq 1$, and hence
\begin{align}\label{equu 3.3}
\begin{split}
  \hat{h}_{top}(G,Z,d, \{F_n\}, 2\epsilon)&\leq\inf_{\diam (\UU) \leq \epsilon}h_{top}(G,Z,\UU) \leq \sup_{\diam (\UU) \leq \epsilon, \Leb (\UU)\geq \frac{\epsilon}{4}}\limits h_{top}(G,Z,\UU)\\
&\leq \sup_{\Leb (\UU)\geq \frac{\epsilon}{4}}\limits h_{top}(G,Z,\UU) \leq  \inf_{\diam (\UU) \leq \frac{\epsilon}{8}}h_{top}(G,Z,\UU).
\end{split}
 \end{align}
By (\ref{equu 3.2}) and (\ref{equu 3.3}), taking the corresponding  limits  we complete the proof.
 \end{proof}

 Proposition \ref{prop 3.4} allows us to  show that  $\limsup_{n \to \infty}$ in  the metric mean dimension, defined by spanning sets and separated sets, can be  replaced by $\liminf_{n \to \infty}$.  This fact shall be useful for  deriving the product formula of  metric mean dimension.

\begin{prop}\label{prop 3.13}
 Let $(X, d, G)$ be a $G$-system, and let $Z \subset X$ be a non-empty $G$-invariant subset. Then for every F\o lner  sequence $\{F_n\}$ of $G$,
 \begin{align*}
{\overline{\rm  mdim}}_M(G,Z,\{F_n\},d)&=\limsup_{\epsilon \to 0}\frac{1}{\logf}\liminf_{n \to \infty} \frac{1}{|F_n|}\log r(Z,d_{F_n},\epsilon),\\
 &=\limsup_{\epsilon \to 0}\frac{1}{\logf}\liminf_{n \to \infty} \frac{1}{|F_n|}\log  s(Z,d_{F_n},\epsilon).
 \end{align*}
 The results are valid for ${\underline{\rm  mdim}}_M(G,Z,\{F_n\},d)$ by changing $\limsup_{\epsilon \to 0}$ into $\liminf_{\epsilon \to 0}$.
 \end{prop}
 
 \begin{proof}
 Similar to (\ref{equu 3.2}), we have
 \begin{align*}
 \inf_{\diam (\UU) \leq \epsilon}h_{top}(G,Z,\UU)\leq \liminf_{n \to \infty} \frac{1}{|F_n|}\log r(Z,d_{F_n},\frac{\epsilon}{4})\leq& \liminf_{n \to \infty} \frac{1}{|F_n|}\log s(Z,d_{F_n},\frac{\epsilon}{4})\\
  \leq& h_{top}(G,Z,d, \{F_n\}, \epsilon).
 \end{align*}
The proof is completed by Proposition  \ref{prop 3.4}. 
 \end{proof}
 
\begin{rem}
$(i)$. To sum up, if $Z \subset X$  is a non-empty $G$-invariant subset, then,  by Proposition \ref{prop 3.4} the upper and lower metric mean dimensions of $Z$ are independent of the choice  of F\o lner sequence of $G$. For this  case, we  drop the dependence on  $\{F_n\}$ for  ${\overline{\rm  mdim}}_M(G,Z,\{F_n\},d)$ and ${\underline{\rm  mdim}}_M(G,Z,\{F_n\},d)$. Additionally, if  ${\overline{\rm  mdim}}_M(G,Z,d)={\underline{\rm  mdim}}_M(G,Z,d)$,   we denote the common value by  ${\rm {mdim}}_M(G,Z,d)$,  and call ${\rm {mdim}}_M(G,Z,d)$ the metric mean dimension of $Z$.

$(ii)$.  Unlike topological entropy,  metric mean dimensions of $Z$  are   metric-dependent quantities and measure  how fast  $\epsilon$-topological entropy $h_{top}(G,Z,d,\{F_n\},\epsilon)$ diverges to $\infty$ in the   sense that if $\epsilon>0$ is sufficiently small,  then one may approximately think of 
$${h}_{top}(G,Z,d,\{F_n\},\epsilon)\approx {\overline{\rm  mdim}}_M(G,Z,\{F_n\},d)\cdot \logf.$$
Clearly,  every $G$-system with finite topological entropy has zero metric mean dimension for every subset $Z \subset X$. This clarifies  that metric mean dimensions  are  useful to capture the  dynamics of $G$-systems with infinite topological entropy.
\end{rem}

We  calculate the metric mean dimension of  the full shift over   compact metric space $(X,d)$. Recall that the upper and lower box dimensions of $X$ w.r.t. $d$ are respectively given  by
\begin{align*}
\overline{\rm dim}_B(X,d)&=\limsup_{\epsilon \to 0}\frac{\log r(X,d,\epsilon)}{\logf}=\limsup_{\epsilon \to 0}\frac{\log s(X,d,\epsilon)}{\logf},\\
\underline{\rm dim}_B(X,d)&=\liminf_{\epsilon \to 0}\frac{\log r(X,d,\epsilon)}{\logf}=\liminf_{\epsilon \to 0}\frac{\log s(X,d,\epsilon)}{\logf}.
\end{align*}

\begin{ex}\label{ex 3.6}
Let $(X,d)$ be a compact metric space.    Endow  the product space $X^G$  with the product topology, which is  metrizable by the following  metric:
$$D((x_g)_g, (y_g)_g)=\sum_{g\in G}\alpha_g d(x_g,y_g),$$
where $(\alpha_g)_g$ is a sequence of positive real  numbers such that $\sum_{g\in G}\alpha_g <\infty$ and $\alpha_{e_G}=1$.

Fix $h\in G$.  The right shift map $\sigma_h: X^G \rightarrow X^G$   is given  by $\sigma_h(x)=(x_{gh})_g$ for every $x=(x_g)_g$.  Then we have
\begin{align*}
{\overline{\rm  mdim}}_M(G, X^G, D)&=\overline{\rm dim}_B(X,d),\\
{\underline{\rm  mdim}}_M(G, X^G, D)&=\underline{\rm dim}_B(X,d).
\end{align*}
\end{ex}
\begin{proof}
Let $\{F_n\}$ be a F\o lner  sequence of $G$.  Fix $\epsilon >0$ and put $M:=\sum_{g\in G}\alpha_g<\infty$. Then there is a finite subset $S\subset G$ containing $e_G$ such that $\sum_{g\in G\backslash S}\alpha_g \cdot \diam(X,d)<\frac{\epsilon}{2}$.  Given a  $(d,\frac{\epsilon}{2M})$-spanning set $E$ of $X$ and $y \in X$, denote by $E_n$ the set of the points in $X^G$ consisting of the sequence $(x_g)_g$ with $x_g \in E$ for  the coordinates $g\in SF_n$ and with $x_g=y$  for the others. Then $E_n$ is a $(D_{F_n},\epsilon)$-spanning set of $X^G$. Indeed,  for any $x=(x_g)_g \in X^G$, choose a point $y_g$ in $E$ such that $d(x_g,y_g)<\frac{\epsilon}{2M}$ for all $g\in SF_n$. Hence, for any $h \in F_n$, we have 
$$D(hx,hy)\leq\sum_{g\in S} \alpha_g d(x_{gh},y_{gh})+\frac{\epsilon}{2}<\epsilon.$$ So $D_{F_n}(x,y)<\epsilon$. This shows that  $r(X^G, D_{F_n},\epsilon)\leq r(X,d,\frac{\epsilon}{2M})^{|SF_n|}$, and hence
\begin{align}\label{equu 3.4}
\hat{h}_{top}(G,X^G,D, \{F_n\},\epsilon)
\leq  \log r(X,d,\frac{\epsilon}{2M}),
\end{align}  
where we used the fact $\lim_{n \to \infty}\frac{|SF_n|}{|F_n|}=1$ since $\{F_n\}$ is a F\o lner sequence. 

Let $R$ be a $(d,\epsilon)$-separated set of $X$ with the largest cardinality and $y\in X$, and let $R_n$  be the set of the points in $X^G$ consisting of the sequence $(x_g)_g$ with $x_g \in R$ for  the coordinates $g \in F_n$ and with $x_g=y$ for the others. Then $R_n$ is a $(D_{F_n},\epsilon)$-separated set of $X^G$. Therefore,  we have $s(X^G, D_{F_n},\epsilon)\geq s(X,d,\epsilon)^{|F_n|}$. This means that
\begin{align}\label{equu 3.5}
h_{top}(G,X^G,D, \{F_n\},\epsilon)
\geq \log s(X,d,\epsilon).
\end{align}
The desired results follow from (\ref {equu 3.4}) and (\ref {equu 3.5}).
\end{proof}

In the following, we present several fundamental properties for metric mean dimension in the rest of this subsection.

\subsubsection{Lipschitz factor map}

One may easily construct a compact metrizable topological space $X$ with  two compatible metrics $d_1,d_2$ on $X$ such that $\overline{\rm dim}_B(X,d_1)\not= \overline{\rm dim}_B(X,d_2)$. Take  $X$, $D_1$ and $D_2$ (corresponding $d_1$ and $d_2$) as in Example \ref{ex 3.6}. The identity map $id: X^G \rightarrow X^G$ is a  conjugate map between the two full shifts $(X^G,D_1,G)$ and  $(X^G,D_2,G)$, while ${\overline{\rm  mdim}}_M(G, X^G, D_1)\not={\overline{\rm  mdim}}_M(G, X^G, D_2)$. Now we introduce the  concept of bi-Lipschitz conjugate map:

Let $(X,d_X, G)$ and $(Y,d_Y, G)$ be two $G$-systems. A map $\pi: X\rightarrow Y$ is a \emph{(bi-)Lipschitz $G$-equivariant} if $\pi$ is   a (bi-)Lipschitz  map satisfying  $\pi\circ T_g(x)=S_g\circ\pi(x)$ for all $g\in G$ and $x\in X$. In addition, $\pi: X\rightarrow Y$ is called a \emph{Lipschitz factor map}  if $\pi$ is  a surjective Lipschitz  $G$-equivariant map.  $\pi$ is called a \emph{bi-Lipschitz conjugate map} if $\pi$ is a bijection bi-Lipschitz  
$G$-equivariant map; in this case,  $(X,d_X,G)$ and $(Y,d_Y, G)$ are said to be  bi-Lipschitz  conjugate.

The metric mean dimensions  are preserved  for bi-Lipschitz conjugate maps between two $G$-systems.

\begin{prop}\label{prop 3.3}
Let $(X,d_X,G)$  and $(Y,d_Y,G)$ be two  $G$-systems. Suppose that
$\pi:X\rightarrow Y$ is a Lipschitz factor map.   Then
\begin{align*}
{\overline{\rm  mdim}}_M(G,X,d_X)&\geq {\overline{\rm  mdim}}_M(G,Y,d_Y),\\ {\underline{\rm  mdim}}_M(G,X,d_X)&\geq {\underline{\rm  mdim}}_M(G,Y,d_Y).
\end{align*}
Additionally, the  equalities  hold if  $\pi$ is a bi-Lipschitz conjugate map.  
\end{prop}

\begin{ex}
Let $([0,1]^G, G)$ and $([a,b]^G, G)$ be two  $G$-systems,  where $a,b\in \mathbb{R}$, and  the metric $D((x_g)_g, (y_g)_g)=\sum_{g\in G}\alpha_g |x_g-y_g|$  is given as in Example \ref{ex 3.6}. Then
$${{\rm  mdim}}_M(G,[a,b]^G,D)=1={{\rm  mdim}}_M(G,[0,1]^G,D).$$
\end{ex}

\subsubsection{Power rule}
We clarify the relation of the metric mean dimensions of a $G$-system and the action of its subgroup. Let  $H$ be a subgroup of $G$. A subset $R \subset G$ is called a \emph{complete  set of representatives} of left-cosets $\{gH:g\in G\}$ of $H$ if each left-coset of $H$  exactly contains a  unique element belonging to $R$. This means the cardinality of $R$  coincides with the index of $H$ in $G$, i.e.,  $|R|=[G:H]$.

An  amenable group $G$ continuously acting on $X$   naturally  induces an action $H\times X \rightarrow X$ by the homeomorphisms $\{T_g: g\in H\}$ of $X$. 

\begin{prop}\label{prop 3.9}
Let $(X, d, G)$ be a $G$-system, and let $H$ be a subgroup of $G$ with finite index in $G$.
   Then 
$${\overline{\rm  mdim}}_M(H,X,d)\leq {\overline{\rm  mdim}}_M(G,X, d)\cdot [G:H].$$
Additionally, if  $e_G \in R \subset G$ is a finite complete  set of  representatives  of the left-cosets of $H$ and  $T|_{R\setminus \{e_G\}}$ are  Lipschitz maps, then $${\overline{\rm  mdim}}_M(H,X,d)= {\overline{\rm  mdim}}_M(G,X, d)\cdot [G:H].$$
The  results are also valid  for the corresponding  lower metric mean dimensions. 
\end{prop}

\begin{proof}
Choose   a finite complete  set   $R$  of  representatives  of the left-cosets of $H$ satisfying $e_G \in R$ and  $|R|=[G:H]$. Let $\{L_n\}$ be a F\o lner sequence  of $H$.
By \cite[Lemma 10.3.7]{coo15},  $\{F_n\}:=\{RL_n\}$  is a F\o lner sequence of $G$.  Then for every $\epsilon >0$,  we have $s(X,d_{L_n},\epsilon) \leq s(X,d_{F_n},\epsilon)$.  This yields that
\begin{align}\label{inequ 3.4}
h_{top}(H,X,d, \{L_n\}, \epsilon)\leq h_{top}(G,X,d, \{F_n\}, \epsilon)\cdot[G:H].
\end{align} 

On the other hand, since $T|_{R\setminus \{e_G\}}: X \rightarrow X$ are  Lipschitz maps, one can choose a  constant $L\geq 1$ such that  for any $x,y \in X$,
$$d(T_gx,T_gy)\leq L \cdot d(x,y)$$   for  all $g\in R$. Using this fact, if $E $  is a $(d_{F_n},\epsilon)$-separated set of $X$, then $E$ is also a $(d_{L_n},\frac{\epsilon}{L})$-separated set of $X$. Similarly, we obtain
\begin{align}\label{inequ 2.10}
h_{top}(H,X,d, \{L_n\}, \frac{\epsilon}{L})\geq h_{top}(G,X,d, \{F_n\}, \epsilon)\cdot[G:H].
\end{align}
By (\ref{inequ 3.4}) and (\ref{inequ 2.10}), we get    the  desired results by taking the corresponding limits.
\end{proof}

\begin{ex}
Consider  $G=\mathbb{Z}$ and $H=m\mathbb{Z}$, where $m\geq 2$, and take $\alpha_n=\frac{1}{2^{|n|}}$  as in Example \ref{ex 3.6}. In this case,  the $G$-systems  $([0,1]^{\mathbb{Z}},D,\mathbb{Z})$ and $([0,1]^{\mathbb{Z}},D,m\mathbb{Z})$ are generated by  the homeomorphisms $T:=T_1$, which is given by $T((x_n)_{n\in \mathbb{Z}})=(x_{n+1})_{n\in \mathbb{Z}}$, and  the iteration $T^m$,   respectively.

Moreover, $L=2^m$ is a uniformly Lipschitz constant such that 
$D(T^jx,T^jy)\leq L \cdot D(x,y)$ for all $1\leq j\leq m-1$. Then, by Proposition  \ref{prop 3.9}, we have
$${\rm mdim}_M(T^m, [0,1]^{\mathbb{Z}},D)=m.$$
\end{ex}

Now applying  Proposition \ref{prop 3.9} to a continuous action of a finite group $G$ and  its trivial subgroup $H=\{e_G\}$,  we get the following corollary.

\begin{cor}
Let  $G$ be a finite group. 
Assume that  $(X,G)$ is  a  $G$-system with a metric $d$ such that $T|_{G\setminus \{e_G\}}$ are  Lipschitz maps. Then
\begin{align*}
{\overline{\rm  mdim}}_M(G,X,d)=\frac{1}{|G|}\overline{\rm dim}_B(X,d),\\
{\underline{\rm  mdim}}_M(G,X,d)=\frac{1}{|G|}\underline{\rm dim}_B(X,d).
\end{align*}
\end{cor}

\subsubsection{Product formula}

Given two  $G$-systems $(X,d_X,G)$ and $(Y,d_Y,G)$, we consider a  $G$-action on $X \times Y$ by the homeomorphism $g(x,y)=(gx,gy)$, where $g\in G$,  and equip $X\times Y$  with a compatible product metric $$d_X\times d_Y\big((x_1,y_1),(x_2,y_2)\big)=\max\{d_X(x_1,x_2), d_Y(y_1,y_2)\}.$$

\begin{prop}\label{prop 3.14}
Let $(X,d_X,G)$ and $(Y,d_Y,G)$ be two  $G$-systems. Then
\begin{align*}
{\underline{\rm  mdim}}_M(G,X\times Y, d_X\times d_Y)
\geq  {\underline{\rm  mdim}}_M(G,X,d_X)+{\underline{\rm  mdim}}_M(G,Y,d_Y),
\end{align*}
and 
\begin{align*}
{\overline{\rm  mdim}}_M(G,X\times Y, d_X\times d_Y)
\leq  {\overline{\rm  mdim}}_M(G,X,d_X)+{\overline{\rm  mdim}}_M(G,Y,d_Y).
\end{align*}
Additionally, if  ${\overline{\rm  mdim}}_M(G,X,d_X)={\underline{\rm  mdim}}_M(G,X,d_X)$, then the  above two equalities hold.
\end{prop}

\begin{proof}
Fix  $F\in  \FF(G)$. If  $E_1$ and $E_2$ are  $(d_F,\epsilon)$-separated sets of $X$ and $Y$  respectively, then $E_1\times E_2$ is a $(d_F,\epsilon)$-separated set of $X\times Y$. Then $s(X\times Y, (d_X\times d_Y)_{F_n},\epsilon)\geq  s(X, (d_X)_{F_n},\epsilon)\cdot s(Y, (d_Y)_{F_n},\epsilon)$.
This implies that
\begin{align}\label{equ 3.3}
\begin{split}
&\liminf_{n \to \infty} \frac{1}{|F_n|} \log s(X, (d_X)_{F_n},\epsilon) + \liminf_{n \to \infty} \frac{1}{|F_n|} \log s(Y, (d_Y)_{F_n},\epsilon)\\
\leq & \liminf_{n \to \infty} \frac{1}{|F_n|} \log s(X\times Y, (d_X\times d_Y)_{F_n},\epsilon).
\end{split}
\end{align}
We obtain the first inequality by Proposition \ref{prop 3.13}. 

Similarly, using spanning sets we get the following  inequality:
\begin{align}\label{equ 3.4}
\begin{split}
\hat{h}_{top}(G,X\times Y,d_X\times d_Y,\{F_n\}, \epsilon)\leq \hat{h}_{top}(G,X,d_X,\{F_n\}, \epsilon)+\hat{h}_{top}(G,Y,d_Y,\{F_n\}, \epsilon).
\end{split}
\end{align} 
This implies the second inequality.

If ${\overline{\rm  mdim}}_M(G,X,d_X)={\underline{\rm  mdim}}_M(G,X,d_X)$, the remaining statement holds by the inequalities (\ref{equ 3.3}) and  (\ref{equ 3.4}).
\end{proof}

\begin{rem}
There are abundant examples  in fractal geometry  (cf. \cite[Example 3.7]{rs12}, \cite[Theorem 3]{www16})  showing  that for some compact metric spaces $(X,d_X)$ and $(Y,d_Y)$, it holds that
\begin{align*}
&\overline{\rm dim}_B(X\times Y,d_X\times d_Y)< \overline{\rm dim}_B(X,d_X)+\overline{\rm dim}_B(Y,d_Y)\\
&\underline{\rm dim}_B(X,d_X)+\underline{\rm dim}_B(Y,d_Y)<\underline{\rm dim}_B(X\times Y,d_X\times d_Y).
\end{align*}
Using this fact and  considering the full shifts over  $X$, $Y$ and the product $X\times Y$ as in Example \ref{ex 3.6}, we conclude that  the  strict inequalities in Proposition \ref{prop 3.14} can happen  for two \emph{different}  $G$-systems\footnote[2]{A similar phenomenon for mean dimension has been revealed  in \cite{t19,jq24}.}.   However, if two  $G$-systems are identical,  by \eqref{equ 3.3} and \eqref{equ 3.4} again, the product formula  of metric mean dimension always holds for  a single $G$-system.
\end{rem}

\begin{prop}
Let $(X, d, G)$ be a $G$-system. Then 
\begin{align*}
{\overline{\rm  mdim}}_M(G,X\times X,d\times d)&=2\cdot{\overline{\rm  mdim}}_M(G,X,d),\\
{\underline{\rm  mdim}}_M(G,X\times X,d\times d)&=2\cdot{\underline{\rm  mdim}}_M(G,X,d).
\end{align*}
\end{prop}




\subsection{Divergent  rate of  measure-theoretic $\epsilon$-entropy}

In this subsection, we collect several types of  measure-theoretic $\epsilon$-entropies of invariant measures, and prove that their divergent rates coincide in terms of ergodic measures (i.e., Theorem \ref{thm 2.15}). 



\subsubsection{Local measure-theoretic entropy of open covers}

There are two types of  local measure-theoretic entropies for finite open covers of $X$ \cite{hyz11}.

Given $\UU \in \mathcal{C}_X^o$ and $\mu \in M(X,G)$, the \emph{local entropy of $\UU$ w.r.t. $\mu$} is defined by
$$h_{\mu}(G,\UU):=\inf_{\alpha \succ \UU}h_{\mu}(G,\alpha),$$
where the infimum ranges over  the set of all finite Borel  partitions of $X$ that refine $\UU$.

It was shown  in \cite[Theorem 3.5]{hyz11} that,  for every $\mu \in M(X,G)$, one has $$h_{\mu}(G,X)=\sup_{\UU \in \mathcal{C}_X^o}h_{\mu}(G,\UU).$$ 

The other  is  Shapira's   local  entropy. Let $0<\delta <1$ and $\mu \in M(X,G)$. We define
$$N_{\mu}(\UU,\delta):=\min\{\#\mathcal{V}: \mathcal{V}\subset \UU~\text{and}~\mu(\cup_{V\in \mathcal{V}}V)\geq 1-\delta\}.$$

The  \emph{Shapira's   local upper and lower entropies of $\UU$ w.r.t. $\mu$} are  respectively defined by
\begin{align*}
\overline{h}_{\mu}^S(G,\UU,\{F_n\})&=\lim_{\delta \to 0}\liminf_{n \to \infty}\frac{1}{|F_n|} \log N_{\mu}(\UU_{F_n},\delta),\\
\underline{h}_{\mu}^S(G,\UU,\{F_n\})&=\lim_{\delta \to 0}\limsup_{n \to \infty}\frac{1}{|F_n|} \log N_{\mu}(\UU_{F_n},\delta).
\end{align*}

The  local  measure-theoretic entropy  and Shapira's   local entropy  of  $\UU$ are  related by the following lemma.

\begin{lem}\label{lem 3.17}\cite[Theorem 4.19]{hyz11}
Let $(X,G)$ be a $G$-system with an open cover  $\UU \in \mathcal{C}_X^o$, and let $\mu \in E(X,G)$. Then 
\begin{align*}
h_{\mu}(G,\UU)=\overline{h}_{\mu}^S(G,\UU)=\underline{h}_{\mu}^S(G,\UU).
\end{align*}
\end{lem}

\begin{rem}
We drop the dependence  of  Shapira's   local upper and lower entropies  on  the F\o lner sequence $\{F_n\}$ if $\mu$ is a ergodic measure.
\end{rem}
\subsubsection{L$^{P}$ and L$^{\infty}$ rate-distortion functions}
For $\mathbb{Z}$-actions, under  certain rate-distortion conditions,  $L^p$ and $L^{\infty}$ rate-distortion functions   were defined by   mutual information of random variables  \cite{lt18}.   Here, we state the amenable version of $L^p$ and $L^{\infty}$ rate-distortion functions introduced in \cite{cdz22}. Readers can refer to the  monographs \cite{ct06,gra 11} for a comprehensive treatment of rate-distortion theory.

Let $(\Omega,\mathbb{P})$ be a  probability space. Suppose that  $\mathcal{X}$, $\mathcal{Y}$ are  two measurable spaces and   $\xi:\Omega \rightarrow \mathcal{X}$ and $\eta: \Omega \rightarrow \mathcal{Y}$ are two measurable maps. The $\emph{mutual information}$ $I(\xi;\eta)$  of $\xi$ and $\eta$  is defined  as 
\begin{align*}
\sup_{\mathcal{P},\mathcal{Q}}\sum_{P \in \mathcal{P}, Q \in \mathcal{Q}}\mathbb{P}((\xi, \eta) \in P \times  Q)\log \frac{\mathbb{P}((\xi, \eta) \in P \times  Q)}{\mathbb{P}(\xi \in P)\mathbb{P}(\eta \in Q)},
\end{align*}
where $\mathcal{P}$ and $\mathcal{Q}$  range over  all  finite   partitions of $\mathcal{X}$  and $\mathcal{Y}$, respectively. Here, we  obey  the convention that $0\log \frac{0}{a}:=0$ for all $a\geq0$. 

Clearly, $I(\xi;\eta)=I(\eta;\xi)\geq 0$.  In particular, $I(\xi;\eta)=0$ if $\xi,\eta$ are two independent random variables. Hence,  mutual information  gives the  amount of  the total information  shared  by both $\xi$ and $\eta$.

Let  $(X,G)$ be a $G$-system with a metric $d$, and let $\mu \in M(X,G)$. For $\epsilon >0$, $1\leq p<\infty$ and $F\in \FF(G)$, we define  $R_{\mu,L^p}(F,\epsilon)$  as the infimum of
$I(\xi;\eta),$
where $\xi$ and $\eta=(\eta_g)_{g\in F}$  are random variables  defined on some probability space $(\Omega, \mathbb{P})$ such that
\begin{enumerate}
\item $\xi$ takes values in $X$, and its law is given by $\mu$.
\item For every $g\in F$, $\eta_g$ takes  values in $X$ and 
\begin{align*}
\mathbb{E}\left(\frac{1}{|F|}\sum_{g\in F}d(g\xi,\eta_g)^p\right)<\epsilon^p,
\end{align*}
\end{enumerate} 
where $\mathbb{E}(\cdot)$  is  the  usual expectation w.r.t.  $\mathbb{P}$.

We define the  $\emph{$L^{p}$ rate-distortion function}$  of $\mu$ w.r.t. a F\o lner sequence $\{F_n\}$ as   $$R_{\mu,L^p}(\{F_n\},\epsilon)=\liminf_{n\to \infty}\frac{R_{\mu,L^p}(F_n,\epsilon)}{|F_n|}.$$

Let $s>0$. We define  $R_{\mu,L^\infty}(F,\epsilon,s)$  as the infimum of
$I(\xi;\eta),$
where  $\xi$ and $\eta=(\eta_g)_{g\in F}$  are random variables  defined on some probability space $(\Omega, \mathbb{P})$ such that
\begin{enumerate}
\item $\xi$ takes values in $X$, and its law is given by $\mu$.
\item For every $g\in F$, $\eta_g$ takes  values in $X$ and 
\begin{align*}
\mathbb{E}\left( \frac{1}{|F|}\#\{g \in F: d(g\xi,\eta_g)\geq \epsilon\}\right)<s.
\end{align*}
\end{enumerate} 
We set $R_{\mu,L^{\infty}}(\{F_n\},\epsilon,s)=\liminf_{n\to \infty}\frac{R_{\mu,L^\infty}(F_n,\epsilon,s)}{|F_n|}$. The $\emph{$L^{\infty}$ rate-distortion function}$  of $\mu$ w.r.t. a F\o lner sequence $\{F_n\}$ is defined as   $$R_{\mu,L^{\infty}}(\{F_n\},\epsilon)=\lim_{s \to 0} R_{\mu,L^{\infty}}(\{F_n\},\epsilon,s).$$ 

\subsubsection{Katok's $\epsilon$-entropies}
Measure-theoretic entropy of ergodic measures admits a  classical Katok's  entropy formula \cite{k80}.

Let $\delta \in (0,1)$, $\mu \in M(X)$ and $\{F_n\}$ be a F\o lner sequence of $G$. For every $F\in \FF(G)$ and $\epsilon >0$, we put
$$R_{\mu}(F,\delta):=\min\{m: \mu(\bigcup_{j=1}^mB_F(x_j,\epsilon)) \geq 1-\delta, x_j \in X, 1\leq j \leq m\}.$$ 
We  define \emph{the upper and lower Katok's  $\epsilon$-entropies of $\mu$ w.r.t. $\{F_n\}$} as
\begin{align*}
\overline{h}_{\mu}^K(\{F_n\},\epsilon):&=\lim_{\delta \to 0} \limsup_{n \to \infty}\frac{1}{|F_n|}\log R_{\mu}(F_n,\delta),\\
\underline{h}_{\mu}^K(\{F_n\},\epsilon):&=\lim_{\delta \to 0} \liminf_{n \to \infty}\frac{1}{|F_n|}\log R_{\mu}(F_n,\delta),
\end{align*}
respectively.

 It was shown in \cite[Theorem 3.1]{zcy16} that  if $\{F_n\}$ is a tempered F\o lner sequence of $G$ with $\lim_{n\to \infty} \frac{|F_n|}{\log n}=\infty$, then for  every $\mu \in E(X,G)$,  $$\lim_{\epsilon \to 0} \overline{h}_{\mu}^K(\{F_n\},\epsilon)=\lim_{\epsilon \to 0} \underline{h}_{\mu}^K(\{F_n\},\epsilon)=h_{\mu}(G,X).$$

Given $F\in \FF(G)$, we define the average metric on $X$ as $$\overline{d}_F(x,y)=\frac{1}{|F|}\sum_{g\in F}d(gx,gy).$$ 
Denote by $\overline{R}_{\mu}(F,\delta)$  the smallest integer $m$ such that  there exists $x_j \in X$, $1\leq j\leq m$, the  $\mu$-measure of the union  of the average $\overline{d}_F$-Bowen balls $B_{\overline{d}_F}(x_j,\epsilon)$ is  no less than $1-\delta$. 

 We  define \emph{the upper and lower average Katok's  $\epsilon$-entropies of $\mu$ w.r.t. $\{F_n\}$} as
 \begin{align*}
 \overline{h}_{\mu,L^1}^K(\{F_n\},\epsilon):&=\lim_{\delta \to 0} \limsup_{n \to \infty}\frac{1}{|F_n|}\log \overline{R}_{\mu}(F_n,\delta),\\
 \underline{h}_{\mu,L^1}^K(\{F_n\},\epsilon):&=\lim_{\delta \to 0} \liminf_{n \to \infty}\frac{1}{|F_n|}\log \overline{R}_{\mu}(F_n,\delta),
 \end{align*}
 respectively.

\subsubsection{Pfister and Sullivan's  $\epsilon$-entropy}
Measure-theoretic entropy can be equivalently defined using separated sets.

Given $F\in \FF(G)$ and $x\in X$, we define a  Borel probability measure  $\mathcal{E}_F(x)=\frac{1}{|F|}\sum_{g\in F}\delta_{gx}$ on $X$,
where $\delta_x$ is the Dirac measure   at $x$.

Let $\mu\in M(X)$. The \emph{Pfister and Sullivan's  $\epsilon$-entropy of $\mu$} is defined by
$$PS_\mu(\{F_n\},\epsilon)=\inf_{O\ni \mu}\limsup_{n \to \infty}\frac{1}{|F_n|}\log s(X_{F_n, O}, d_{F_n},\epsilon),$$
where the infimum  is taken over  all neighborhoods  $O$  of $\mu$ in $M(X)$, and 
$X_{F, O}:=\{x\in X:  \mathcal{E}_F(x) \in O\}.$

In \cite{rtz23}, the authors showed that  for every $\mu \in E(X,G)$,  $$\lim_{\epsilon \to 0}\limits PS_\mu(\{F_n\},\epsilon)=h_{\mu}(G,X).$$

Letting $\epsilon \to 0$, different  measure-theoretic 
$\epsilon$-entropies of ergodic measures may lead to infinite measure-theoretic entropy.
In  the following,  we demonstrate that these  measure-theoretic $\epsilon$-entropies, chosen from the candidate set
$$\mathcal{E}:=\left\{\inf_{
\diam (\alpha)\leq \epsilon}\limits h_{\mu}(G,\alpha), PS_\mu(\{F_n\},\epsilon), \underline{h}_{\mu}^K(\{F_n\}, \epsilon), \overline{h}_{\mu}^K(\{F_n\},\epsilon),\atop  \inf_{\diam (\UU) \leq \epsilon}\limits h_{\mu}(G,\UU), \sup_{\diam (\UU) \leq \epsilon, \atop \Leb (\UU) \geq \frac{\epsilon}{4} } \limits h_{\mu}(G,\UU),\sup_{\Leb (\UU) \geq \epsilon } \limits h_{\mu}(G,\UU),R_{\mu,L^{\infty}}(\{F_n\},\epsilon) \right\},$$
have the same divergent rate as $\epsilon \to 0$.




\begin{thm}\label{thm 2.15}
Let  $(X,d,G)$ be a $G$-system and $\mu \in E(X,G)$, and let $\{F_n\}$ be a tempered F\o lner sequence of $G$. Then for every $h_{\mu}(G,\{F_n\},\epsilon) \in \mathcal{E}$,
\begin{align*}
\limsup_{\epsilon \to 0}\frac{h_{\mu}(G,\{F_n\},\epsilon)}{\logf}&=\limsup_{\epsilon \to 0}\frac{1}{\logf} \lim_{\delta \to 0}\inf\{h_{top}(G, Z,d,\{F_n\},\epsilon):\mu(Z)\geq 1-\delta\},\\
\liminf_{\epsilon \to 0}\frac{h_{\mu}(G,\{F_n\},\epsilon)}{\logf}&=\liminf_{\epsilon \to 0}\frac{1}{\logf} \lim_{\delta  \to 0}\inf\{h_{top}(G, Z,d,\{F_n\},\epsilon):\mu(Z)\geq 1-\delta\}.
\end{align*}
Additionally, if $(X,d)$ has the tame growth  of covering numbers, then   $R_{\mu,L^p}(\{F_n\},\epsilon)$ ($1\leq p <\infty$) can be included in $\mathcal{E}$.
\end{thm}

By Theorem \ref{thm 2.15},  modulo a sufficiently small constant $\delta>0$,  one may approximately  regard  the precise divergent rates of these measure-theoretic $\epsilon$-entropies  as the  metric mean dimension of a   
Borel set with $\mu$-measure no less than $1-\delta$.


Several auxiliary lemmas are presented to pave the way for the proof of Theorem \ref{thm 2.15}.
\begin{lem}\label{lem 3.20}
Let  $(X,d,G)$ be a $G$-system and $\mu \in E(X,G)$, and let $\{F_n\}$ be a tempered F\o lner sequence of $G$. Then for every  $\epsilon >0$,
$$\lim_{\delta \to 0}\inf\{h_{top}(G, Z,d,\{F_n\},\epsilon):\mu(Z)\geq 1-\delta\}\leq PS_\mu(\{F_n\},\epsilon).$$
\end{lem}

\begin{proof}
Let  $s> PS_\mu(\{F_n\},\epsilon)$.  Then there exists an open neighborhood $O$ of $\mu$ such that $$\limsup_{n \to \infty}\frac{1}{|F_n|}\log s(X_{F_n, O}, d_{F_n},\epsilon)<s.$$
 We define the  set of generic points of $\mu$ along  the tempered F\o lner sequence $\{F_{n}\}$ as
$$G_{\mu, \{F_{n}\}}:=\{x\in X:\lim_{n \to \infty}\mathcal{E}_{F_{n}}(x)=\mu ~\text{in the weak}^{*} \text{-topology}\}.$$
By  Lindenstrauss's pointwise ergodic theorem \cite[Theorem  1.2]{lin01}, for  any $f\in C(X,\mathbb{R})$, one has 
$\lim_{n\to \infty}\frac{1}{|F_n|}\sum_{g\in F_n}f(gx)=\int fd\mu$
for $\mu$-a.e.$x\in X$. This implies that $\mu (G_{\mu, \{F_{n}\}})=1$. 
Let $N\in \mathbb{N}$ and put
$$G_{\mu, \{F_{n}\}}^N=\{x\in  G_{\mu, \{F_{n}\}}: \mathcal{E}_{F_{n}}(x)\in O ~~\forall n\geq N\}.$$
Since the map  $x\in X \mapsto  \mathcal{E}_{F_{n}}(x)$ is continuous for each fixed $n\geq 1$, $G_{\mu, \{F_{n}\}}^N$ is Borel measurable for all $N \geq 1$.  Fix $\delta \in (0,1)$. Notice that $G_{\mu, \{F_{n}\}}^N$ is non-decreasing in $N$ and the union  $\bigcup_{N\geq 1}G_{\mu, \{F_{n}\}}^N$ has $\mu$-full measure. There exists  a positive integer $N_0$ such that $\mu(G_{\mu, \{F_{n}\}}^{N_0}) \geq 1-\delta$.  Since $G_{\mu, \{F_{n}\}}^{N_0}\subset X_{F_{n},O}$ for every $n\geq N_0$, we have
$$\inf\{h_{top}(G, Z,d,\{F_n\},\epsilon):\mu(Z)\geq 1-\delta\}\leq h_{top}(G,G_{\mu, \{F_n\}}^{N_0},d, \{F_n\}, \epsilon)<s.$$
Letting $\delta \to 0$ and $s \to PS_\mu(\{F_n\},\epsilon)$, we get the desired inequality.
\end{proof}

We  need the  following  three  propositions to  link the Pfister and Sullivan's  $\epsilon$-entropy and  Kolmogorov-Sinai $\epsilon$-entropy of ergodic measures.

Let $(X,d)$ be a compact metric space. Recall that  the \emph{boundary} of a finite Borel partition $\alpha$ of $X$ is the union of the boundaries of the atoms of  $\alpha$, i.e., $\partial \alpha =\bigcup_{A\in \alpha}\partial A$.

\begin{prop}\label{lem 3.21}
Let $\mu \in M(X,G)$.
For any $\gamma >0$, $F\in \FF(G)$ and $\UU \in  \mathcal{C}_X^o$, if $\alpha  \in \mathcal{P}_X$ refines $\UU$,  then there exists $\beta \in \mathcal{P}_X$  with $\mu(\partial \beta)=0$ refining  $ \UU$  such that
$$H_{\mu}(\beta_F|\alpha_F)<\frac{\gamma}{2},$$
where $H_{\mu}(\beta |\alpha)$ denotes the usual conditional entropy of $\beta$ given $\alpha$. 
\end{prop}

\begin{proof}
This is a slight modification of the proof in \cite[Lemma 3.6]{hyz11}.
\end{proof}

Proposition \ref{lem 3.21} allows us   to formulate an equivalent definition for the  local measure-theoretic entropy of  open covers.

\begin{prop}\label{prop 3.22}
Let  $(X,d,G)$ be a $G$-system and $\mu \in M(X,G)$. Then for every $\UU\in \mathcal{C}_X^o$, we have
$$h_{\mu}(G,\UU)=\inf_{\alpha \succ \UU, \mu(\partial \alpha)=0}h_{\mu}(G,\alpha).$$
\end{prop}

\begin{proof}
By  comparing the definitions, it suffices to show  $$h_{\mu}(G,\UU)\geq \inf_{\alpha \succ \UU, \mu(\partial \alpha)=0}h_{\mu}(G,\alpha).$$
Let  $\alpha \succ \UU$. Given  $\gamma >0$, we choose $F_0\in \FF(G)$ such that
$\frac{H_{\mu}(\alpha_{F_0})}{|F_0|}<h_{\mu}(G,\alpha)+\frac{\gamma}{2}.$ By Proposition \ref{lem 3.21},  there exists a  finite Borel partition  $\beta $ of $X$ with $\mu(\partial \beta)=0$ refining $ \UU$  such that  $H_{\mu}(\beta_{F_0}|\alpha_{F_0})<\frac{\gamma}{2}$. Therefore, one has
\begin{align*}
\frac{H_{\mu}(\beta_{F_0})}{|F_0|}\leq  \frac{H_{\mu}(\alpha_{F_0})}{|F_0|}+H_{\mu}(\beta_{F_0}|\alpha_{F_0})
<h_{\mu}(G,\alpha)+\gamma.
\end{align*}
This yields that
\begin{align*}
\inf_{\mathcal{P} \succ \UU, \mu(\partial \mathcal{P})=0}h_{\mu}(G,\mathcal{P})&=\inf_{\mathcal{P} \succ \UU, \mu(\partial \mathcal{P})=0}~\inf_{F\in \FF(G)} \frac{H_{\mu}(\mathcal{P}_{F})}{|F|}\\
&=\inf_{F\in \FF(G)}  ~\inf_{\mathcal{P} \succ \UU, \mu(\partial \mathcal{P})=0}\frac{H_{\mu}(\mathcal{P}_{F})}{|F|}\\
&\leq\inf_{\mathcal{P}\succ \UU, \mu(\partial \mathcal{P})=0}\frac{H_{\mu}(\mathcal{P}_{F_0})}{|F_0|}\leq \frac{H_{\mu}(\beta_{F_0})}{|F_0|}<h_{\mu}(G,\alpha)+\gamma.
\end{align*}
Since $\alpha$ is  arbitrary and $\gamma>0$ can be chosen sufficiently small, this implies the desired  inequality.
\end{proof}

\begin{prop}\label{lem 3.23}
Let $(X,d,G)$ be a $G$-system  and $\{F_n\}$ be a F\o lner sequence of $G$.  Let $\epsilon >0$ and  $\{X_n\}$ be a  sequence of non-empty subsets of $X$. Suppose that  $E_n$ is a  $(d_{F_n},\epsilon)$-separated set of $X_n$ with the largest cardinality. Define
$$\mu_n=\frac{1}{\#E_n}\sum_{x\in E_n}\mathcal{E}_{F_n}(x).$$
Assume that $\mu_n \rightarrow \mu$ in the weak$^{*}$-topology. Then $\mu \in M(X,G)$ and 
\begin{align*}
\limsup_{n \to \infty}\frac{1}{|F_n|} \log s(X_n,d_{F_n},\epsilon)
\leq \inf_{ 
\diam (\alpha)\leq \epsilon,   \mu(\partial \alpha)=0}h_{\mu}(G,\alpha).
\end{align*}
\end{prop}

\begin{proof}
The statement $\mu \in M(X,G)$ is due to the fact $\lim_{n\to\infty}\frac{|gF_n\triangle F_n|}{|F_n|}=0$ for all $g\in G$.  The remaining  inequality can be proved using the arguments in \cite[Theorem 9.48, p.227-228]{kl16}.  
\end{proof}
\begin{lem}\label{lem 3.24}
Let  $(X,d,G)$ be a $G$-system and $\mu \in E(X,G)$, and let $\{F_n\}$ be a F\o lner sequence of $G$. Then for every $\epsilon >0$, $$ PS_\mu(\{F_n\},\epsilon)\leq  \inf_{ 
\diam (\alpha)\leq \epsilon,   \mu(\partial \alpha)=0}h_{\mu}(G,\alpha).$$
\end{lem}

\begin{proof}
Suppose that  the inequality is false.  Then  there exists $\delta >0$ such that $$\inf_{ 
\diam (\alpha)\leq \epsilon,  \mu(\partial \alpha)=0} \limits h_{\mu}(G,\alpha)  +\delta< PS_\mu(\{F_n\},\epsilon).$$ 
Choose a   dense set $\{f_n\}_{n=1}^{\infty}$   of $C(X,\mathbb{R})$. The following metric  on $M(X)$ $$D(\mu,m):=\sum_{n\geq 1}\frac{|\int f_n d\mu- \int f_n dm|}{2^n(||f||+1)}$$
is  compatible  with the weak$^{*}$-topology of $M(X)$ (cf. \cite[Theorem 6.4, p.148]{w82}).
We let $L_{k}:=\overline{B_{D}(\mu,\frac{1}{k})}$. Then $L_k$ is a   closed  convex set of $M(X)$ and $\cap_{k\geq 1}L_k=\{\mu\}$. We choose a strictly increasing sequence $\{m_k\}$ of positive integers such that
$$\inf_{ 
	\diam (\alpha)\leq \epsilon,  \mu(\partial \alpha)=0} \limits h_{\mu}(G,\alpha)  +\delta<\frac{\log s(X_{F_{m_k},L_k}, d_{F_{m_k}},\epsilon)}{|F_{m_k}|}.$$	
For every $n\in [m_k, m_{k+1})$, we define $O_n:=L_k$. Then $\{O_n\}$  is a  decreasing  sequence of  closed  convex sets  in $M(X)$ such that $\cap_{n\geq 1}O_n=\{\mu\}$ and 
$$\inf_{ 
	\diam (\alpha)\leq \epsilon,  \mu(\partial \alpha)=0} \limits h_{\mu}(G,\alpha)  +\delta<\frac{\log s(X_{F_{m_k},O_{m_k}}, d_{F_{m_k}},\epsilon)}{|F_{m_k}|}.$$	
This yields that
$\inf_{
\diam (\alpha)\leq \epsilon,  \mu(\partial \alpha)=0}h_{\mu}(G,\alpha) +\delta\leq \limsup_{n\to \infty}\frac{1}{|F_n|}\log s(X_{F_n, O_n},d_{F_n},\epsilon).$
Let $E_n $ be a $(d_{F_n},\epsilon)$-separated set of $ X_{F_n, O_n}$ with the largest cardinality. We  define
$$\mu_n=\frac{1}{\#E_n}\sum_{x\in E_n}\mathcal{E}_{F_n}(x).$$
Then $\mu_n \to \mu$ as $n \to \infty$ by the choice of $O_n$. By  Proposition \ref{lem 3.23}, we have 
$$PS_\mu(\{F_n\},\epsilon)\leq \limsup_{n\to \infty}\frac{1}{|F_n|}\log s(X_{F_n, O_n},d_{F_n},\epsilon)
\leq \inf_{
\diam (\alpha)\leq \epsilon,  \mu(\partial \alpha)=0}h_{\mu}(G,\alpha),$$
a contradiction with the assumption!
\end{proof}

Finally, we clarify the relations between the Kolmogorov-Sinai $\epsilon$-entropy and Katok's $\epsilon$-entropy  of ergodic measures.

\begin{lem}\label{lem 3.25}
Let  $(X,d,G)$ be a $G$-system, and let $\{F_n\}$ be a  F\o lner sequence of $G$. Then for every $\epsilon >0$,
\begin{itemize}
\item [(1)]  if $\mu \in M(X,G)$, then $\inf_{\diam (\alpha)\leq \epsilon} \limits h_{\mu}(G,\alpha)\leq \inf_{
\diam (\alpha)\leq \epsilon,  \mu(\partial \alpha)=0} \limits h_{\mu}(G,\alpha) \leq \inf_{\diam (\UU) \leq \epsilon} \limits  h_{\mu}(G,\UU)
\leq  \sup_{\diam (\UU) \leq \epsilon,  \Leb (\UU) \geq \frac{\epsilon}{4} }\limits h_{\mu}(G,\UU) \leq  \sup_{\Leb (\UU) \geq \frac{\epsilon}{4} } \limits h_{\mu}(G,\UU)
\leq \inf_{
\diam (\alpha)\leq \frac{\epsilon}{8}}\limits h_{\mu}(G,\alpha);$
\item [(2)]if $\mu \in E(X,G)$, then   $\inf_{
\diam (\alpha)\leq \epsilon}\limits h_{\mu}(G,\alpha)\leq  \underline{h}_{\mu}^K(\{F_n\},\frac{\epsilon}{4})\leq \overline{h}_{\mu}^K(\{F_n\},\frac{\epsilon}{4})$;
\item [(3)] if $\mu \in M(X,G)$, then  $ \overline{h}_{\mu}^K(\{F_n\},\epsilon)\leq  \lim_{\delta \to 0} \inf\{h_{top}(G, Z,d,\{F_n\},\epsilon):\mu(Z)\geq 1-\delta\}.$
\end{itemize}
\end{lem}

\begin{proof}
(1). Let $\UU \in \mathcal{C}_X^o$  with  $\diam (\UU)\leq  \epsilon$. Notice that the diameter of each $\alpha \in \mathcal{P}_X$ refining  $\UU \in \mathcal{C}_X^o$  is at most $\epsilon$. Then, by Proposition \ref{prop 3.22} we  have 
$$\inf_{
\diam (\alpha)\leq \epsilon,   \mu(\partial \alpha)=0}h_{\mu}(G,\alpha)\leq \inf_{\diam (\UU) \leq \epsilon} h_{\mu}(G,\UU).$$
Choose $\VV \in \mathcal{C}_X^o$ with  $\diam (\VV) \leq 
\epsilon$ and $\Leb (\VV) \geq \frac{\epsilon}{4}$. Then
\begin{align*}
\inf_{\diam (\UU) \leq \epsilon} h_{\mu}(G,\UU)\leq h_{\mu}(G,\VV) \leq&  \sup_{\diam (\UU) \leq \epsilon,  \Leb (\UU) \geq \frac{\epsilon}{4} } h_{\mu}(G,\UU)\\
 \leq&  \sup_{\Leb (\UU) \geq \frac{\epsilon}{4} } h_{\mu}(G,\UU) 
\leq  \inf_{ 
\diam (\alpha)\leq \frac{\epsilon}{8}}h_{\mu}(G,\alpha),
\end{align*}
where  we used  the fact that each $\alpha \in \mathcal{P}_X$  with  $\diam  (\alpha) \leq 
\frac{\epsilon}{8}$ refines $\UU$ for the last inequality.

(2). Choose $\UU \in \mathcal{C}_X^o$ with  $\diam (\UU) \leq 
\epsilon$ and $\Leb (\UU) \geq \frac{\epsilon}{4}$ again. Then, by Lemma \ref{lem 3.17} we have
\begin{align*}
\inf_{ 
\diam (\alpha)\leq \epsilon}h_{\mu}(G,\alpha)\leq  h_{\mu}(G,\UU)\leq \underline{h}_{\mu}^K(\{F_n\},\frac{\epsilon}{4})\leq \overline{h}_{\mu}^K(\{F_n\},\frac{\epsilon}{4}). 
\end{align*}

(3).  It is clear by comparing their definitions. 
\end{proof}

With the help of these lemmas, we  are ready to prove Theorem \ref{thm 2.15}.
\begin{proof}[Proof of Theorem \ref{thm 2.15}]
By Lemmas \ref{lem 3.20}, \ref{lem 3.24} and \ref{lem 3.25},  for every $\epsilon >0$, we have the following inequality:
\begin{align}\label{equ 2.10}
\begin{split}
&\lim_{\delta \to 0}\inf\{h_{top}(G, Z,d,\{F_n\},\epsilon):\mu(Z)\geq 1-\delta\}\\
\leq &
PS_\mu(\{F_n\},\epsilon)
\leq \inf_{\diam (\UU) \leq \epsilon} h_{\mu}(G,\UU)
\leq  \sup_{\diam (\UU) \leq \epsilon,  \Leb (\UU) \geq \frac{\epsilon}{4} } h_{\mu}(G,\UU)\\
\leq & \sup_{\Leb (\UU) \geq \frac{\epsilon}{4} } h_{\mu}(G,\UU)
\leq 
\inf_{\diam (\alpha)\leq \frac{\epsilon}{8}}h_{\mu}(G,\alpha)
\leq \underline{h}_{\mu}^K(\{F_n\},\frac{\epsilon}{32})\\
\leq& \overline{h}_{\mu}^K(\{F_n\},\frac{\epsilon}{32})
\leq  \lim_{\delta \to 0} \inf\{h_{top}(G, Z,d,\{F_n\},\frac{\epsilon}{32}):\mu(Z)\geq 1-\delta\}.
\end{split}
\end{align}

As for Katok's $\epsilon$-entropies and rate-distortion functions,  the authors \cite[Propositions 4.1 and 4.2]{ljzz25} have proved the following two inequalities: for every $\epsilon >0$,
 \begin{align}\label{equu 3.11}
R_{\mu,L^{\infty}}(\{F_n\},\epsilon) \leq \overline{h}_{\mu}^K(\{F_n\},\epsilon) \leq R_{\mu,L^{\infty}}(\{F_n\},\frac{\epsilon}{8}),
\end{align}
and  for  $L>2$,
\begin{align}\label{equu 3.12}
R_{\mu,L^{1}}(\{F_n\},2\epsilon) \leq \overline{h}_{\mu,L^1}^K(\{F_n\},\epsilon) \leq \frac{L}{L-1}R_{\mu,L^{1}}(\{F_n\},\frac{\epsilon}{8L+2}).
\end{align}

Then, by (\ref{equ 2.10}) and (\ref{equu 3.11}) the divergent rates of the  measure-theoretic $\epsilon$-entropies in $\mathcal{E}$ are as desired.

Now assume that $(X,d)$ has the tame growth of covering numbers. By the proof of \cite[Theorem 1.3, p.14]{ljzz25}, for any $\theta \in (0,1)$ we have 
\begin{align}\label{equu 3.13}
\begin{split}
 \overline{h}_{\mu}^K(\{F_n\},2\epsilon^{1-\theta}) 
\leq& \log 2 + \epsilon^{\theta}\log r(X,d,\epsilon) +\overline{h}_{\mu,L^1}^K(\{F_n\},\epsilon). 
\end{split}
\end{align}
Then we have \begin{align*}
(1-\theta)\cdot \limsup_{\epsilon \to 0}\frac{ \overline{h}_{\mu}^K(\{F_n\},2\epsilon^{1-\theta}) }{\log \frac{1}{2\epsilon^{1-\theta}}}\leq& \limsup_{\epsilon \to 0}\frac{\log 2 + \epsilon^{\theta}\log r(X,d,\epsilon) +\overline{h}_{\mu,L^1}^K(\{F_n\},\epsilon)}{\logf}\\
=&\limsup_{\epsilon \to 0}\frac{\overline{h}_{\mu,L^1}^K(\{F_n\},\epsilon)}{\logf}.
\end{align*}
Letting $\theta \to 1$ and  noticing that $\overline{h}_{\mu,L^1}^K(\{F_n\},\epsilon) \leq \overline{h}_{\mu}^K(\{F_n\},\epsilon)$ for every $\epsilon>0$,  we conclude that  $\overline{h}_{\mu,L^1}^K(\{F_n\},\epsilon)$ and $ \overline{h}_{\mu}^K(\{F_n\},\epsilon)$   have the same divergent rates by  taking $\limsup_{\epsilon \to 0}$(or $\liminf_{\epsilon \to 0}$).  Together with  (\ref{equu 3.12}), we know that  $R_{\mu,L^{1}}(\{F_n\},\epsilon)$ can be included in $\mathcal{E}$. For the case $1<p<\infty$,  we have
\begin{align}\label{equu 3.14}
R_{\mu,L^{1}}(\{F_n\},2\epsilon)\leq R_{\mu,L^{p}}(\{F_n\},2\epsilon)\leq R_{\mu,L^{\infty}}(\{F_n\},\epsilon)
\end{align}
by H\"{o}lder's inequality and comparing their definitions. Thus, we include $R_{\mu,L^{p}}(\{F_n\},\epsilon)$ into $\mathcal{E}$. This  completes the proof.
\end{proof}

Since $\lim_{\epsilon \to 0}\inf_{\diam (\UU) \leq \epsilon} h_{\mu}(G,\UU)=h_{\mu}(G,X)$ for every $\mu \in E(X,G)$, as an additional result of  (\ref{equ 2.10}) and (\ref{equu 3.11}), we summarize the  known  measure-theoretic entropy formulae as follows:

\begin{thm}
Let  $(X,G)$ be a $G$-system and $\mu \in E(X,G)$, and let $\{F_n\}$ be a tempered F\o lner sequence of $G$. Then for every $h_{\mu}(G,\{F_n\},\epsilon) \in \mathcal{E}$,
$$\lim_{\epsilon \to  0} h_{\mu}(G,\{F_n\},\epsilon)=h_{\mu}(G,X).$$
\end{thm}

\begin{rem}
We remark that the assumption $\lim_{n\to \infty} \frac{|F_n|}{\log n}=\infty$ imposed on the F\o lner sequence for the  Katok's entropy formula \cite[Theorem 3.1]{zcy16} is redundant.
\end{rem}

\subsection{Measure-theoretic metric mean dimension of invariant measures} 
In this subsection, we introduce the notion of  measure-theoretic metric mean dimension for invariant measures.

We roughly explain our strategy to define the metric mean dimension of invariant measures.

(1) A nice and reasonable definition for the  measure-theoretic metric mean dimension of invariant measures should be independent of the choice of candidates in $\mathcal{E}$. Otherwise, it may lead to different types of definitions of  measure-theoretic metric mean dimension.

(2) By the proof of Theorem \ref{thm 2.15}, if $\mu$ is an ergodic measure, then for every measure-theoretic $\epsilon$-entropy $h_{\mu}(G,\{F_n\},\epsilon) \in \mathcal{E}$, there exist $C_1>0$ and $C_2>0$ such that for all $\epsilon >0$,
\begin{align}\label{equu 3.15}
	f(\mu,C_1\epsilon,\{F_n\})\leq h_{\mu}(G,\{F_n\},\epsilon)\leq f(\mu,C_2\epsilon,\{F_n\}),
\end{align}
where $f(\mu,\epsilon,\{F_n\}):=\lim_{\delta \to 0}\inf\{h_{{top}}(G, Z,d,\{F_n\},\epsilon):\mu(Z)\geq 1-\delta\}$.

(3) To realize (1), we extend the inequality \eqref{equu 3.15} to all non-ergodic measures. We complete this by the well-known Krein-Milman theorem, which asserts that $M(X,G)=\overline{\mathrm{co}(E(X,G))}$. Hence, we first define the measure-theoretic $\epsilon$-entropy for the convex hull $\mathrm{co}(E(X,G))$ of ergodic measures, and then define the  measure-theoretic metric mean dimension of invariant measures using suitable approximations from the elements of $\mathrm{co}(E(X,G))$.

More precisely, this idea proceeds via the following steps.

\begin{df}\label{df 3.26}
$(a)$ Let $\{F_n\}$ be a (tempered) F\o lner sequence of $G$ and $\epsilon  >0$. For $\mu \in  \mathrm{co}(E(X,G))$ and $h_{\mu}(G,\{F_n\},\epsilon) \in \mathcal{E}$, we write\footnote[3]{In light of the well-known Ergodic Decomposition Theorem, each element in  $\mathrm{co}(E(X,G))$ can be uniquely  expressed as a finite convex combination of the ergodic measures of $X$.} $\mu=\sum_{j=1}^k\lambda_j \mu_j,~\mu_j \in E(X,G), \sum_{j=1}^{k}\lambda_j=1,~0\leq \lambda_j\leq 1,~j=1,...,k$, and  define the measure-theoretic $\epsilon$-entropy of $\mu$ as
\begin{align*}
F(\mu, \{F_n\}, \epsilon):=\sum_{j=1}^{k}\lambda_jh_{\mu_j}(G,\{F_n\},\epsilon).
\end{align*}

$(b)$ For $\mu \in  M(X,G)$, let $M_{G}(\mu)$  denote the collection of all families $\{\mu_{\epsilon}\}_{\epsilon >0} \subset \mathrm{co}(E(X,G))$  that converge to $\mu$  as $\epsilon \to 0$ in the weak$^{*}$-topology. We define  the  measure-theoretic upper and lower  metric mean dimensions of $\mu$ as
\begin{align*}
\begin{split}
\overline{\rm {mdim}}_M(\mu,\{F_n\},d)&=\sup_{(\mu_{\epsilon})_{\epsilon}\in M_{G}(\mu) }\{\limsup_{\epsilon \to 0}\frac{F(\mu_{\epsilon},\{F_n\}, \epsilon)}{\log \frac{1}{\epsilon}}\},\\
\underline{\rm {mdim}}_M(\mu,\{F_n\},d)&=\sup_{(\mu_{\epsilon})_{\epsilon}\in M_{G}(\mu) }\{\liminf_{\epsilon \to 0}\frac{F(\mu_{\epsilon},\{F_n\}, \epsilon)}{\log \frac{1}{\epsilon}}\},
\end{split}
\end{align*}
respectively.
\end{df}


\begin{rem}\label{rem 3.24}
$(i)$. By the inequality (\ref{equu 3.15}) and the above steps (a) and (b), any candidate  $h_{\mu}(G,\{F_n\},\epsilon)$ from $\mathcal{E}$ yields the same   measure-theoretic metric mean dimension. Therefore,  Definition \ref{df 3.26}  is independent of the choice of  candidates in $\mathcal{E}$\footnote[4]{In the rest of this paper, before calculating the measure-theoretic metric mean dimension of invariant measures, we will first clarify which candidate \( h_{\mu}(G, \{F_n\}, \epsilon) \in \mathcal{E} \) to use.}, while inherently relying on the compatible metrics on $X$. 

$(ii)$. The readers may pose another definition of measure-theoretic (upper) metric mean dimension using Theorem \ref{thm 2.15} and Ergodic Decomposition Theorem.  More precisely, by Theorem \ref{thm 2.15}, for every ergodic measure  it is reasonable to define:
\begin{align*}
\overline{\rm {mdim}}_M^{*}(\mu,\{F_n\},d):=\limsup_{\epsilon \to 0}\frac{h_{\mu}(G,\{F_n\},\epsilon) }{\logf},
\end{align*}
where $h_{\mu}(G,\{F_n\},\epsilon) \in \mathcal{E}$ and $\mu \in E(X,G)$. 

Now let $\mu \in M(X,G)$ and $\mu=\int_{E(X,G)} m~ d\tau (m)$ be the ergodic decomposition of $\mu$, where $\tau$ is a  Borel probability measure on $M(X,G)$ and $\tau(E(X,G))=1$. We define the measure-theoretic upper metric mean dimension of $\mu$ as
\begin{align*}
\overline{\rm {mdim}}_M^{*}(\mu,\{F_n\},d):=\int_{E(X,G)} \overline{\rm {mdim}}_M^{*}(m,\{F_n\},d) d\tau(m).
\end{align*}

As is shown in forthcoming Example \ref{ex 4.4}, there exists a $G$-system  such that the variational principle  of upper metric mean dimension \emph{fails}  in terms of $\overline{\rm {mdim}}_M^{*}(\mu,\{F_n\},d)$. To establish the variational principle of metric mean dimension for all $G$-systems, we define the measure-theoretic metric mean dimension  of invariant measures as $\overline{\text{mdim}}_M(\mu, \{F_n\}, d)$, rather than using $\overline{\text{mdim}}_M^*(\mu, \{F_n\}, d)$.

$(iii)$. Notice that  $\inf_{\diam (\alpha)\leq \epsilon} h_{\mu}(G,\alpha)$  does not depend on  the  choice of  tempered F\o lner sequences of $G$.  Any two different  tempered F\o lner sequences of $G$ lead to the same  measure-theoretic metric mean dimension.  

$(iv)$. For every  $\epsilon >0$, the measure-theoretic $\epsilon$-entropy $F(\mu, \{F_n\}, \epsilon)$ is ``well-defined" in the sense that it is finite for all  $\mu \in  \mathrm{co}(E(X,G))$. To see this, fix  $\epsilon >0$ and consider the Katok's $\epsilon$-entropy $h_{\mu}(G,\{F_n\},\epsilon)= \overline{h}_{\mu}^K(\{F_n\},\epsilon)$. By definition, one has $$\overline{h}_{\mu}^K(\{F_n\},\epsilon)\leq \hat{h}_{top}(G, X,d,\{F_n\},\epsilon)<\infty$$ for all $\mu \in E(X,G)$. Together with (\ref{equu 3.15}), we have $F(\mu, \{F_n\}, \epsilon)<\infty$ for all  $\mu \in  co(E(X,G))$ and $h_{\mu}(G,\{F_n\},\epsilon) \in \mathcal{E}$.

$(v)$.  If  $(X,d,G)$ is a $G$-system with finite topological entropy, then   $h_{\mu}(G,X)<\infty$ for all $\mu \in M(X,G)$, and hence $\overline{\rm {mdim}}_M(\mu,\{F_n\},d)=0$ for all $\mu \in M(X,G)$. Therefore, the  measure-theoretic metric mean dimension is a metric-dependent  quantity to characterize the dynamics of $G$-systems with infinite topological entropy.

\end{rem}

In particular, if  $F(\mu_{\epsilon},\{F_n\}, \epsilon)=R_{\mu_{\epsilon},L^{\infty}}(\{F_n\},\epsilon)$ or $R_{\mu_{\epsilon},L^{p}}(\{F_n\},\epsilon)$, we apply the above idea to define the modified rate-distortion dimensions of invariant measures.

\begin{df}\label{df 3.27}
Let $(X,d,G)$ be a  $G$-system, and let $\{F_n\}$ be a  (tempered) F\o lner sequence of $G$. For every $\mu \in M(X,G)$, we  respectively
define  the modified  upper $L^{\infty}$ and $L^p$ ($1\leq p< \infty$) rate-distortion dimensions of $\mu$ as
\begin{align*}
\overline{\rm {mdim}}_{M,L^{\infty}}(\mu,\{F_n\},d)&=\sup_{(\mu_{\epsilon})_{\epsilon}\in M_{G}(\mu) }\{\limsup_{\epsilon \to 0}\frac{R_{\mu_\epsilon,L^{\infty}}(\{F_n\},\epsilon)}{\log \frac{1}{\epsilon}}\},\\
\overline{\rm {mdim}}_{M,L^p}(\mu,\{F_n\},d)&=\sup_{(\mu_{\epsilon})_{\epsilon}\in M_{G}(\mu) }\{\limsup_{\epsilon \to 0}\frac{R_{\mu_\epsilon,L^{p}}(\{F_n\},\epsilon)}{\log \frac{1}{\epsilon}}\}.
\end{align*}
\end{df}

Replacing $\limsup_{\epsilon \to 0}$ by $\liminf_{\epsilon \to 0}$ in Definition \ref{df 3.27}, one can   define  $\underline{\rm {mdim}}_{M,L^{\infty}}(\mu,\{F_n\},d)$ and $\underline{\rm {mdim}}_{M,L^{p}}(\mu,\{F_n\},d)$  as the corresponding  modified  lower $L^{\infty}$ and $L^p$ rate-distortion dimensions of $\mu$, respectively.

An elementary but  useful formulation for Definition \ref{df 3.26} is  the following proposition.

\begin{prop} \label{prop 3.27}
Let $(X, d, G)$ be a $G$-system, and let $\{F_n\}$ be a tempered F\o lner sequence of $G$.  Then for every $\lambda >0$, we have  
\begin{align*}
\begin{split}
\overline{\rm {mdim}}_M(\mu,\{F_n\},d)&=\sup_{(\mu_{\epsilon})_{\epsilon}\in M_{G}(\mu) }\{\limsup_{\epsilon \to 0}\frac{F(\mu_{\epsilon},\{F_n\}, \lambda\epsilon)}{\log \frac{1}{\epsilon}}\},\\
\underline{\rm {mdim}}_M(\mu,\{F_n\},d)&=\sup_{(\mu_{\epsilon})_{\epsilon}\in M_{G}(\mu) }\{\liminf_{\epsilon \to 0}\frac{F(\mu_{\epsilon},\{F_n\}, \lambda\epsilon)}{\log \frac{1}{\epsilon}}\},
\end{split}
\end{align*}
for all $\mu \in M(X,G)$.
\end{prop}

\begin{proof}
It can be proved using the fact: if $f:(0,\infty)\rightarrow \mathbb{R}_{\geq 0}\cup\{\infty\}$, then $ \limsup_{\epsilon \to 0}f( \epsilon)=\limsup_{\epsilon \to 0}f( \lambda \epsilon)$ and $ \liminf_{\epsilon \to 0}f( \epsilon)=\liminf_{\epsilon \to 0}f( \lambda \epsilon)$ for every $0<\lambda  <\infty$.
\end{proof}

We calculate the  measure-theoretic metric mean dimension of a product  measure of the full shift $(([0,1]^m)^G, G)$.
\begin{ex}\label{ex 3.28}
Let $m$ be a positive integer. Endow $[0,1]^{m}$ with a metric $$||x-y||_{\mathbb{R}^m}=\max_{1\leq n\leq m}\limits |x_n-y_n|.$$ Consider the full shift  $(([0,1]^m)^G,d, G)$ with the right shift defined by   $\sigma_h(x)=(x_{gh})_g$ for every $x=(x_g)_g \in ([0,1]^m)^G$ as in Example \ref{ex 3.6}.  Then
$$d((x_g),(y_g))=\sum_{g\in G}\alpha_g ||x_g-y_g||_{\mathbb{R}^m},$$
is a compatible product metric on $([0,1]^m)^{G}$, where $(\alpha_g)_g$ is a sequence of positive real  numbers such that $\sum_{g\in G}\alpha_g <\infty$ and $\alpha_{e_G}=1$.

Denote by $\text{Leb}([0,1]^{m})$  the  Lebesgue measure on $[0,1]^m$. Assume that $\mu=(\text{Leb}([0,1]^{m}))^{\otimes G}$ is a product measure  generated by the  Lebesgue measure on $[0,1]^m$.  Then for every  tempered   F\o lner sequence $\{F_n\}$ of $G$,
$$\underline{\rm {mdim}}_M(\mu,\{F_n\},d)=\overline{\rm {mdim}}_M(\mu,\{F_n\},d)=m.$$
\end{ex}

\begin{proof}
Fix a tempered   F\o lner sequence $\{F_n\}$ of $G$. By Example  \ref{ex 3.6}, we  have already known that $\overline{{\rm {mdim}}}_M(G, ([0,1]^m)^G,d)=m.$
As we have mentioned in Remark  \ref{rem 3.24}, for  every ergodic measure $\mu \in E(X,G)$, one has  $$F(\mu, \{F_n\}, \epsilon)=\overline{h}_{\mu}^K(\{F_n\},\epsilon) \leq  \hat{h}_{top}(G, ([0,1]^m)^G,d,\{F_n\},\epsilon).$$ Thus, for any $(\mu_{\epsilon})_{\epsilon}\in M_{G}(\mu)$ it holds that  for every $\epsilon >0$, $$F(\mu_\epsilon, \{F_n\}, \epsilon)\leq \hat{h}_{top}(G, ([0,1]^m)^G,d,\{F_n\},\epsilon).$$ This implies that $\overline{\rm {mdim}}_M(\mu,\{F_n\},d)\leq m.$

Now we let   $h_{\mu}(G,\{F_n\},\epsilon)= \inf_{ 
	\diam (\alpha)\leq \epsilon}h_{\mu}(G,\alpha)$.
	Given  $\epsilon>0$, we define the Brin-Katok's $\epsilon$-entropy of  $\mu$  at $x\in X$ w.r.t.  $\{F_n\}$ as
\begin{align*}
\underline{h}_{\mu}^{BK}(x,\{F_n\},\epsilon)=\liminf_{n \to \infty}-\frac{1}{|F_n|}\log \mu(B_{F_n}(x,\epsilon)).
\end{align*}
Noticing that for all $x\in ([0,1]^m)^G$ and $n\geq 1$,  one has $B_{F_n}(x,\epsilon)\subset \{y\in ([0,1]^m)^G:  ||x_g-y_g||_{\mathbb{R}^m}<\epsilon~~ \forall g \in F_n\},$  and hence
$$\underline{h}_{\mu}^{BK}(x,\{F_n\},\epsilon)\geq   \liminf_{n \to \infty}-\frac{\log (2\epsilon)^{m |F_n|}}{|F_n|} = m \cdot  \log \frac{1}{2\epsilon}.$$
Let $\alpha$ be a finite Borel   partition  of $([0,1]^m)^G$ with  $\diam(\alpha)\leq \frac{\epsilon}{2}$.  By Shannon-McMillan-Breiman Theorem of amenable group \cite[Theorem 1.3]{lin01}, for any tempered F\o lner sequence $\{L_n\}$ of $G$ with $\lim_{n \to \infty}\frac{|L_n|}{\log n}=\infty$, one has
$$\lim_{n \to\infty}-\frac{\log \mu (\alpha_{L_n}(x))}{|L_n|}=h_{\mu}(G,\alpha)$$
for $\mu$-a.e. $x\in X$,  where $\alpha_{L_n}(x)$ is the atom of the partition $\alpha_{L_n}$ to which  $x$ belongs. 

We  may require  that 
the F\o lner sequence $\{F_n\}$ satisfies $|F_n|>n$ for every $n$; otherwise,  one can consider its subsequence $\{F_{n_k}\}_k$  by  the fact $\lim_{n \to \infty} |F_n|=\infty$. Since $\alpha_{F_n}(x) \subset  B_{F_n}(x,\epsilon)$ for every  $n\geq 1$ and $x\in ([0,1]^m)^G$, we get 
$m \cdot  \log \frac{1}{2\epsilon}\leq h_{\mu}(G,\alpha)$, and hence  $$m \cdot  \log \frac{1}{2\epsilon}\leq \inf_{\diam(\alpha)\leq \frac{\epsilon}{2}}h_{\mu}(G,\alpha).$$
Let $\mu_{\epsilon}:=\mu$ for all $\epsilon >0$.  Since $\mu$ is ergodic, we have   $F(\mu_{\epsilon}, \{F_n\}, \frac{\epsilon}{2})= \inf_{\diam(\alpha)\leq \frac{\epsilon}{2}}\limits h_{\mu}(G,\alpha)$. 
By Proposition \ref{prop 3.27}, this yields that 
$$\underline{\rm {mdim}}_M(\mu,\{F_n\},d)\geq \liminf_{\epsilon \to 0}\frac{1}{\log \frac{1}{\epsilon}}\inf_{\diam(\alpha)\leq \frac{\epsilon}{2}}\limits h_{\mu_\epsilon}(G,\alpha)\geq m.$$ 
\end{proof}

The following proposition shows that the upper and lower measure-theoretic metric mean dimensions are  upper semi-continuous  functions on $M(X,G)$. Recall that an extended-valued function $f$ on a compact metric space $X$ is said to be \emph{upper semi-continuous} (u.s.c.) if  $\limsup_{y\to x}f(y)\leq f(x)$ for all  $x\in X$.
\begin{prop}\label{prop 3.26}
Let $(X, d, G)$ be a $G$-system, and let $\{F_n\}$ be a tempered F\o lner sequence of $G$.  Then
\begin{align*}
&\overline{\rm {mdim}}_M(\cdot,\{F_n\},d): M(X,G)\rightarrow [0,+\infty],\\
&\underline{\rm {mdim}}_M(\cdot,\{F_n\},d): M(X,G)\rightarrow [0,+\infty]
\end{align*}  
are upper semi-continuous.
\end{prop}

\begin{proof}
(1). We first show $\overline{\rm {mdim}}_M(\cdot,\{F_n\},d)$ is an u.s.c. function. Suppose that $D$ is a  metric compatible with the weak$^{*}$-topology of $M(X)$.
Fix $\mu \in M(X,G)$. If $\overline{\rm {mdim}}_M(\mu,\{F_n\},d)=+\infty$, it is clear that  the function is  u.s.c. at $\mu$. Now we  consider $\overline{\rm {mdim}}_M(\mu,\{F_n\},d)<+\infty$, and let $\mu_n\in M(X,G)$ such that $\mu_n \to \mu$ as $n \to \infty$. Without loss of generality, we assume that $$\lim_{n \to \infty}\overline{\rm {mdim}}_M(\mu_n,\{F_n\},d)=\limsup_{n \to\infty}\overline{\rm {mdim}}_M(\mu_n,\{F_n\},d);$$ otherwise we  replace $\lim_{n \to \infty}$ by  the $\lim_{k \to \infty}$ for  some  subsequence  $\{{n_k}\}$. Choose a subsequence $\{\mu_{n_j}\}_j$  with
$D(\mu_{n_j},\mu)<\frac{1}{2j}$ for every $j \in \mathbb{N}$. Let $\gamma >0$. For every $j$, we take $(v_{\epsilon}^j)_{\epsilon}\in M_G(\mu_{n_j})$ satisfying
$\overline{\rm {mdim}}_M(\mu_{n_j},\{F_n\},d)-\gamma<\limsup_{\epsilon \to 0} \frac{F(v_{\epsilon}^j,\{F_n\}, \epsilon) }{\logf}.$
Then there is a strictly decreasing  sequence $\{\epsilon_{n_j}\}_j$   converging to $0$ as  $j \to \infty$ such that
\begin{align*}
D(\nu_{\epsilon_{n_j}}^j,\mu_{n_j})<\frac{1}{2j},~\text{and}~
\overline{\rm {mdim}}_M(\mu_{n_j},\{F_n\},d)-\gamma<\ \frac{F(\nu_{\epsilon_{n_j}}^j,\{F_n\}, \epsilon_{n_j})}{\log\frac{1}{\epsilon_{n_j}}}.
\end{align*}
For every $\epsilon >0$, we let $\lambda_{\epsilon}:=\nu_{\epsilon_{n_j}}^j$ if $\epsilon \in (\epsilon_{n_{j+1}}, \epsilon_{n_j}]$, and $\lambda_{\epsilon}:=\mu$ if $\epsilon >\epsilon_{n_1}$. Then $\lambda_{\epsilon} \to \mu$ as $\epsilon  \to 0$ since $\nu_{\epsilon_{n_j}}^j \to \mu$ as $j \to \infty$. Hence, we have
$$\limsup_{n \to\infty}\overline{\rm {mdim}}_M(\mu_n,\{F_n\},d)-\gamma \leq \overline{\rm {mdim}}_M(\mu,\{F_n\},d).$$
This implies the upper semi-continuity  of $\overline{\rm {mdim}}_M(\cdot,\{F_n\},d)$ at $\mu$ by letting $\gamma \to 0$.

(2). We continue to show  $\underline{\rm {mdim}}_M(\cdot,\{F_n\},d)$ is also an u.s.c. function. It suffices to show the function is u.s.c. at these points  $\mu$ of $M(X,G)$ satisfying $\underline{\rm {mdim}}_M(\mu,\{F_n\},d)<+\infty$. Let  $\mu$ be such a measure.  If $\underline{\rm {mdim}}_M(\cdot,\{F_n\},d)$ is not u.s.c. at $\mu$, there exists a sequence $\{\mu_n\}$ of invariant measures converging to $\mu$ such that for some $\delta >0$,
 $$\limsup_{n \to\infty}\underline{\rm {mdim}}_M(\mu_n,\{F_n\},d) > \underline{\rm {mdim}}_M(\mu,\{F_n\},d) +\delta.$$ 
 Without loss of generality, for every $n \geq 1$ we assume that $D(\mu_n,\mu)<\frac{1}{2n}$ and
$$\underline{\rm {mdim}}_M(\mu_n,\{F_n\},d) > \underline{\rm {mdim}}_M(\mu,\{F_n\},d) +\delta.$$ Now for every $n \geq 1$, choose  $(\mu_n)_{\epsilon} \in M_G(\mu_n)$ such that $$\underline{\rm {mdim}}_M(\mu_n,\{F_n\},d)-\frac{\delta}{2}< \liminf_{\epsilon \to 0}\frac{F((\mu_n)_{\epsilon},\{F_n\},\epsilon)}{\log \frac{1}{\epsilon}}.$$
 Then there exists  $\epsilon_n \in (0,\frac{1}{n})$ such that  for all $\epsilon \in (0,\epsilon_n)$, one has $D((\mu_n)_\epsilon, \mu_n)<\frac{1}{2n}$ and 
$$\underline{\rm {mdim}}_M(\mu,\{F_n\},d) +\frac{\delta}{2}<\underline{\rm {mdim}}_M(\mu_n,\{F_n\},d)-\frac{\delta}{2}< \frac{F((\mu_n)_{\epsilon},\{F_n\},\epsilon)}{\log \frac{1}{\epsilon}}.$$
We may require that the resulting sequence $\{\epsilon_n\}$ satisfies $0<\epsilon_{n+1}<\epsilon_n$. If $\epsilon_{n+1}<\epsilon \leq \epsilon_n$, we define $\nu_\epsilon:= (\mu_n)_\epsilon$.  Then $\nu_\epsilon \to \mu$ as $\epsilon \to 0$ since $D((\mu_n)_\epsilon, \mu)<\frac{1}{n}$. Therefore,  for every  $0<\epsilon<\epsilon_1$ with  $\epsilon_{n+1}<\epsilon \leq \epsilon_n$ one has 
$$\underline{\rm {mdim}}_M(\mu,\{F_n\},d) +\frac{\delta}{2}< \frac{F(\nu_{\epsilon},\{F_n\},\epsilon)}{\log \frac{1}{\epsilon}}.$$
This implies that $ \underline{\rm {mdim}}_M(\mu,\{F_n\},d) +\frac{\delta}{2}\leq  \underline{\rm {mdim}}_M(\mu,\{F_n\},d)$, 
a contradiction!
\end{proof}

\section{Variational principles for metric mean dimension}\label{sec 4}
In this section, we prove  Theorems \ref{thm 1.1} and \ref{thm 1.2}.

\subsection{Lindenstrauss-Tsukamoto's variational principle  revisited}

We link the two  different types of metric mean dimensions  by the following variational principles:

\begin{thm}[=Theorem \ref{thm 1.1}]\label{thm 3.30}
Let $(X, d, G)$ be a $G$-system  and $\{F_n\}$ be a tempered F\o lner sequence of $G$.  Then
\begin{align*}
{\overline{\rm  mdim}}_M(G,X,d)&=\max_{\mu \in M(X,G)}\overline{\rm {mdim}}_M(\mu,\{F_n\},d).
\end{align*}
\end{thm}

We require an auxiliary lemma for our proof.

\begin{lem}\label{lem 4.2}
Let $(X, d, G)$ be a $G$-system  and  $\{F_n\}$ be a F\o lner sequence of $G$. Then for every  $\epsilon  >0$, there exist $\mu_{\epsilon} \in E(X,G)$ and $ \nu_{\epsilon}\in M(X,G)$ such that
\begin{align*}
h_{top}(G,X,d, \{F_n\},\epsilon)
\leq \inf_{\diam (\alpha)\leq \frac{\epsilon}{8}}h_{\nu_{\epsilon}}(G,\alpha)
\leq \inf_{\diam (\alpha)\leq \frac{\epsilon}{64}}h_{\mu_{\epsilon}}(G,\alpha).
\end{align*}
\end{lem}

\begin{proof}
Following the proof  of \cite[Theorem 9.48]{kl16}, whose proof yields a half of variational principle for  topological  entropy,  for every $\epsilon >0$, there exists $\nu_{\epsilon} \in M(X,G)$ such that 
\begin{align*}
h_{top}(G,X,d, \{F_n\}, \epsilon)\leq  \inf_{\diam (\alpha)\leq \epsilon,  \nu_{\epsilon}(\partial \alpha)=0}h_{\nu_{\epsilon}}(G,\alpha) 
\leq  \inf_{\diam (\alpha)\leq \frac{\epsilon}{8}}h_{\nu_{\epsilon}}(G,\alpha),
\end{align*}
where  we used Lemma \ref{lem 3.25}, (1) for the second inequality.

Choose  $\VV\in \mathcal{C}_X^o$ with  $\diam (\mathcal{V})\leq 
\frac{\epsilon}{8}$ and $\Leb (\mathcal{V}) \geq \frac{\epsilon}{32}$. Then  
 $ \inf_{
\diam (\alpha)\leq \frac{\epsilon}{8}} \limits h_{\nu_{\epsilon}}(G,\alpha)\leq h_{\nu_\epsilon}(G,\mathcal{V}).$ Applying the ergodic decomposition theorem  \cite[Theorem 3.13]{hyz11} to  $h_{\nu_{\epsilon}}(G,\mathcal{V})$, there exists $\mu_{\epsilon} \in E(X,G)$ such that
$ \inf_{
\diam (\alpha)\leq \frac{\epsilon}{8}}\limits h_{\nu_{\epsilon}}(G,\alpha)\leq h_{\mu_\epsilon}(G,\mathcal{V})$. Together with the fact
 $h_{\mu_\epsilon}(G,\mathcal{V}) \leq  \inf_{ 
\diam (\alpha)\leq \frac{\epsilon}{64}}\limits h_{\mu_{\epsilon}}(G,\alpha)$, this completes the proof.
\end{proof}

\begin{proof}[Proof of Theorem \ref{thm 3.30}]
(1). We first establish the  variational principle:
$${\overline{\rm  mdim}}_M(G,X,d)=\max_{\mu \in M(X,G)}\overline{\rm {mdim}}_M(\mu,\{F_n\},d).$$

We let $h_{\mu}(G,\{F_n\},\epsilon)= \overline{h}_{\mu}^K(\{F_n\},\epsilon)$. Since for every $\epsilon >0$,  we have $$F(\mu, \{F_n\}, \epsilon) \leq \hat{h}_{top}(G, X,d,\{F_n\},\epsilon) $$ for every $\mu \in \mathrm{co}(E(X,G))$. Let $\mu \in M(X,G)$. It holds that for every $(\mu_{\epsilon})_{\epsilon} \in M_G(\mu)$, 
\begin{align*}
 \limsup_{\epsilon \to 0}\frac{F(\mu_{\epsilon}, \{F_n\}, \epsilon)}{\logf} \leq {\overline{\rm  mdim}}_M(G,X,d).
\end{align*}
Taking the  supremum over all families $(\mu_{\epsilon})  \in M_G(\mu)$ and noting that $\mu \in M(X,G)$ is arbitrary, we conclude that
\begin{align}\label{inequ 2.14}
\sup_{\mu \in M(X,G)}\overline{\rm {mdim}}_M(\mu,\{F_n\},d) \leq {\overline{\rm  mdim}}_M(G,X,d).
\end{align}

Now let  $h_{\mu}(G,\{F_n\},\epsilon)= \inf_{ 
\diam (\alpha)\leq \epsilon}h_{\mu}(G,\alpha)$. Choose a strictly monotonically decreasing sequence $\{\epsilon_n\}$ with $\epsilon_n \to 0$ as $n \to \infty$ satisfying
$${\overline{\rm  mdim}}_M(G,X,d)=\lim_{n \to \infty}\frac{1}{\log \frac{1}{\epsilon_n}} h_{top}(G,X,d, \{F_n\},\epsilon_n).$$
By Lemma \ref{lem 4.2}, for every $\epsilon >0$, there exists $\mu_{\epsilon} \in E(X,G)$ such that 
\begin{align}\label{equu 3.17}
h_{top}(G,X,d, \{F_n\},\epsilon)
\leq&  \inf_{ 
\diam (\alpha)\leq \frac{\epsilon}{64}}h_{\mu_{\epsilon}}(G,\alpha)=F(\mu_{\epsilon},\{F_n\}, \frac{\epsilon}{64}).
\end{align}
Since $M(X,G)$ is compact,  the sequence $\{\mu_{\epsilon_n}\}_n$ has an invariant measure as its accumulation point. We may assume that for  some  $\mu \in M(X,G)$,
$\mu_{\epsilon_n} \rightarrow\mu $ as $n \to \infty$ (by passing to a subsequence if necessary). For every $\epsilon>0$, we define $\nu_\epsilon:=\mu_{
\epsilon_n}$ if $\epsilon \in (\epsilon_{n+1},\epsilon_n]$. Then  $(\nu_{\epsilon})_{\epsilon} \in M_G(\mu)$.
Using the inequality (\ref{equu 3.17}),  we get
\begin{align}\label{equ 3.6}
{\overline{\rm  mdim}}_M(G,X,d)
\leq&  \limsup_{n \to \infty} \frac{1}{\log \frac{1}{\epsilon_n}} F(\mu_{\epsilon_n},\{F_n\}, \frac{\epsilon_n}{64}) \nonumber \\
=& \limsup_{n \to \infty}   \frac{1}{\log \frac{1}{\epsilon_n}} F(\nu_{\epsilon_n},\{F_n\}, \frac{\epsilon_n}{64}) \\
\leq&  \overline{\rm {mdim}}_M(\mu,\{F_n\},d), \nonumber
\end{align}
and $\mu$ maximizes the supremum of the desired variational principle.
\end{proof}

\begin{rem}
 Under the setting of Theorem \ref{thm 1.1},  we have  $$\max_{\mu \in M(X,G)}\limits \underline{\rm {mdim}}_M(\mu,\{F_n\},d)\leq {\underline{\rm  mdim}}_M(G,X,d).$$ However,  it is unclear whether the strict inequality of the variational inequality is possible for some $G$-systems
\end{rem}

\begin{proof}[Proof of Theorem \ref{thm 1.2}]
By  Theorem \ref{thm 1.1}, we have already proved  $${\overline{\rm  mdim}}_M(G,X,d)=\max_{\mu \in M(X,G)}\overline{\rm {mdim}}_{M,L^{\infty}}(\mu,\{F_n\},d)$$
since $R_{\mu,L^{\infty}}(\{F_n\},\epsilon)$ is a candidate considered in $\mathcal{E}$.

Next  we establish the variational principle:
 $${\overline{\rm  mdim}}_M(G,X,d)=\max_{\mu \in M(X,G)}\overline{\rm {mdim}}_{M,L^{p}}(\mu,\{F_n\},d),$$
 where $1\leq p<\infty$.
 
Assume that the compact metric space $(X,d)$ has the tame growth of covering numbers.  Then, by (\ref{equu 3.12}) and (\ref{equu 3.13}), for every $\epsilon >0$, $\theta \in (0,1)$ and $L>2$, we have 
\begin{align}\label{equu 4.4}
\begin{split}
 \overline{h}_{\mu}^K(\{F_n\},2\epsilon^{1-\theta}) &\leq \log 2 + \epsilon^{\theta}\log r(X,d,\epsilon) +\overline{h}_{\mu,L^1}^K(\{F_n\},\epsilon)\\
&\leq \log 2 + \epsilon^{\theta}\log r(X,d,\epsilon)+ \frac{L}{L-1}R_{\mu,L^{p}}(\{F_n\},\frac{\epsilon}{8L+2}). 
\end{split}
\end{align}
for all $\mu \in E(X,G)$. Let   $h_{\mu}(G,\{F_n\},\epsilon)=  \overline{h}_{\mu}^K(\{F_n\},\epsilon)$.  Letting  $\mu \in \mathrm{co}(E(X,G))$, we rewrite $\mu=\sum_{j=1}^m a_j \mu_j,~\mu_j \in E(X,G),~ \sum_{j=1}^{m}a_j=1,~0\leq a_j\leq 1,~j=1,...,m$. Then, by Definition \ref{df 3.27} we have
\begin{align}\label{equu 4.5}
\begin{split}
F(\mu,\{F_n\}, 2\epsilon^{1-\theta})\leq& \log 2 + \epsilon^{\theta}\log r(X,d,\epsilon)+ \frac{L}{L-1} \sum_{j=1}^m a_j R_{\mu_j,L^{p}}(\{F_n\},\frac{\epsilon}{8L+2}),\\
=& \log 2 + \epsilon^{\theta}\log r(X,d,\epsilon)+ \frac{L}{L-1}  R_{\mu,L^{p}}(\{F_n\},\frac{\epsilon}{8L+2}).
\end{split}
\end{align}
Hence, this inequality is true  for all  $\mu \in \mathrm{co}(E(X,G))$.  Now choose $\nu \in M(X,G)$ such that  ${\overline{\rm  mdim}}_M(G,X,d)=\overline{\rm {mdim}}_M(\nu,\{F_n\},d).$ For every $(\nu_{\epsilon})_{\epsilon} \in M_G(\nu)$,  by (\ref{equu 4.5})  we have
\begin{align*}
(1-\theta)\cdot \limsup_{\epsilon \to 0}\frac{F(\nu_{\epsilon},\{F_n\}, 2\epsilon^{1-\theta})}{\log \frac{1}{2\epsilon^{1-\theta}}}\leq  \frac{L}{L-1}\limsup_{\epsilon \to 0}\frac{R_{\nu_{\epsilon},L^{p}}(\{F_n\},\frac{\epsilon}{8L+2})}{\log \frac{1}{\epsilon}}.
\end{align*}
By letting $L\to \infty$ and $\theta \to 1$, this  yields that  $${\overline{\rm  mdim}}_M(G,X,d)\leq  \overline{\rm {mdim}}_{M,L^{p}}(\nu,\{F_n\},d).$$
Recall that in (\ref{equu 3.14}), for  every  $1\leq p<\infty$  we have  $	R_{\mu,L^{p}}(\{F_n\},2\epsilon)\leq R_{\mu,L^{\infty}}(\{F_n\},\epsilon)$ for all $\mu \in M(X,G)$ and $\epsilon>0$. Hence,  one has
 $$ \overline{\rm {mdim}}_{M,L^{p}}(\nu,\{F_n\},d)\leq \overline{\rm {mdim}}_{M,L^{\infty}}(\nu,\{F_n\},d),$$ 
which implies that $$ \max_{\mu \in M(X,G)}\overline{\rm {mdim}}_{M,L^{p}}(\mu,\{F_n\},d) = {\overline{\rm  mdim}}_M(G,X,d).$$
This completes the proof.
\end{proof}

\begin{ex}\label{ex 4.3}
In \cite[Section VIII]{lt18}, Lindenstrauss and Tsukamoto showed that for a certain dynamical system of $\mathbb{Z}$-action, the variational principles for metric mean dimension fail for the $L^p$ ($p\in \{1,\infty\}$) rate-distortion functions. More precisely, let $G = \mathbb{Z}$, $X = \{0\} \cup \left\{\frac{1}{n} : n \geq 1\right\}$, and the metric $D(x, y) = \sum_{n \in \mathbb{Z}} \frac{|x_n - y_n|}{2^{|n|}}$ as in Example \ref{ex 3.6}. For such a full shift $(X^{\mathbb{Z}}, D, G)$, it holds that $\text{mdim}_M(G, X^{\mathbb{Z}}, D) = \frac{1}{2},$
and for all $\mu \in {M}(X, G)$ and $p \in \{1, \infty\}$,
\[
\lim_{\epsilon \to 0} \frac{1}{\logf} R_{\mu, L^p}(\{F_n\}, \epsilon) = 0 \footnote[5]{For the detailed calculation, we refer the reader to \cite[Lemma 3.1]{kd94}.},
\]
where $F_n := \{0, 1, \dots, n-1\}$.

Using the modified $L^p$ ($1 \leq p \leq \infty$) rate-distortion dimension,
Theorem \ref{thm 1.2} suggests that the variational principles of metric mean dimension are valid for all $G$-systems. It is easy to verify that $(X, |\cdot|)$ has the tame growth of covering numbers, and so does $(X^{\mathbb{Z}}, D)$. By Theorem \ref{thm 1.2}, there  exists $\mu \in {M}(X, G)$ such that for $p \in \{1, \infty\}$,
\[
\overline{\text{mdim}}_{M, L^p}(\mu, \{F_n\}, D) = \text{mdim}_M(G, X^{\mathbb{Z}}, D) = \frac{1}{2}.
\]

Furthermore,  Theorem \ref{thm 2.15} implies that
\[
\sup_{\mu \in {M}(X, G)} \overline{\text{mdim}}_M^*(\mu, \{F_n\}, D) = 0 < \text{mdim}_M(G, X^{\mathbb{Z}}, D) = \frac{1}{2},
\]
and hence this clarifies the aforementioned Remark \ref{rem 3.24} (ii).
\end{ex}

Finally, we further present the variants of the variational principles for  upper and lower metric mean dimensions in alternative formulations. 

\begin{thm}\label{thm 3.33}
Let $(X, d, G)$ be a $G$-system,  and  let $\{F_n\}$ be a  tempered F\o lner sequence of $G$. Then 
\begin{align*}
{\overline{\rm  mdim}}_M(G,X,d)&=\limsup_{\epsilon \to 0}\frac{1}{\logf}\sup_{\mu \in E(X,G)} {h}_{\mu}(G,\{F_n\},\epsilon)\\
&=\limsup_{\epsilon \to 0}\frac{1}{\logf}\sup_{\mu \in M(X,G)}{h}_{\mu}(G,\{F_n\},\epsilon) ,
\end{align*}
where ${h}_{\mu}(G,\{F_n\},\epsilon)$ is chosen from  the candidate set
$$\mathcal{E}_1 = \{ \inf_{
\diam (\alpha)\leq \epsilon}\limits h_{\mu}(G,\alpha), \inf_{\diam (\UU) \leq \epsilon}\limits h_{\mu}(G,\UU), \sup_{\diam (\UU) \leq \epsilon, \atop \Leb (\UU) \geq \frac{\epsilon}{4} } \limits h_{\mu}(G,\UU),\sup_{\Leb (\UU) \geq \epsilon} \limits h_{\mu}(G,\UU)\},$$
and 
\begin{align*}
{\overline{\rm  mdim}}_M(G,X,d)&=\limsup_{\epsilon \to 0}\frac{1}{\logf}\sup_{\mu \in E(X,G)} {h}_{\mu}(G,\{F_n\},\epsilon)\\
&=\limsup_{\epsilon \to 0}\frac{1}{\logf}\sup_{\mu \in M(X,G)}{h}_{\mu}(G,\{F_n\},\epsilon)\\
&=\limsup_{\epsilon \to 0}\frac{1}{\logf}\sup_{\mu \in M(X)}{h}_{\mu}(G,\{F_n\},\epsilon) ,
\end{align*}
where ${h}_{\mu}(G,\{F_n\},\epsilon)$ is chosen from  the candidate set
$$\mathcal{E}_2=\{  PS_\mu(\{F_n\},\epsilon), \underline{h}_{\mu}^K(\{F_n\}, \epsilon), \overline{h}_{\mu}^K(\{F_n\},\epsilon)\}.$$

These variational principles also hold for  ${\underline{\rm  mdim}}_M(G,X,d)$ by changing $\limsup_{\epsilon \to 0}$ into $\liminf_{\epsilon \to 0}$.
\end{thm}

\begin{proof}
Using Lemma \ref{lem 4.2} again,  for every $\epsilon >0$,  we have
\begin{align*}
h_{top}(G,X,d,\{F_n\},\epsilon)\leq \inf_{\diam (\alpha)\leq \frac{\epsilon}{8}}h_{\nu_{\epsilon}}(G,\alpha)
\leq \inf_{\diam (\alpha)\leq \frac{\epsilon}{64}}h_{\mu_{\epsilon}}(G,\alpha)
\end{align*}
for some $\mu_{\epsilon} \in E(X,G)$ and $\nu_{\epsilon} \in M(X,G)$. 
By definition, for every $\epsilon >0$ one has $$h_{top}(G,X,d,\{F_n\},\epsilon) \geq \overline{h}_{\mu}^{K}(\{F_n\},\epsilon)$$ for all $\mu \in E(X,G)$. Together the above two facts with the inequality (\ref{equu 3.15}),   the  following variational principles hold for all ${h}_{\mu}(G,\{F_n\},\epsilon)\in \mathcal{E}$ in terms of ergodic measures, i.e.,
\begin{align}\label{equ 3.14}
{\overline{\rm  mdim}}_M(G,X,d)&=\limsup_{\epsilon \to 0}\frac{1}{\logf}\sup_{\mu \in E(X,G)}{h}_{\mu}(G,\{F_n\},\epsilon)\\
&=\limsup_{\epsilon \to 0}\frac{1}{\logf}\sup_{\mu \in M(X,G)}\inf_{\diam (\alpha)\leq \epsilon}h_{\mu}(G,\alpha) \nonumber.
\end{align}
Then,  by Lemma \ref{lem 3.25},(1), we know that  the  non-ergodic variational principles (\ref{equ 3.14}) hold  for every  candidate ${h}_{\mu}(G,\{F_n\},\epsilon)$ from  $\mathcal{E}_1$.

To establish the remaining variational principles,  for every  ${h}_{\mu}(G,\{F_n\},\epsilon) \in \mathcal{E}_2$,  by (\ref{equ 3.14}) it suffices  to show  for every $\mu \in M(X)$ and $\epsilon >0$,
$${h}_{\mu}(G,\{F_n\},\epsilon) \leq h_{top}(G,X,d,\{F_n\},\epsilon).$$
This follows directly by comparing their definitions with that of
$h_{top}(G,X,d,\{F_n\},\epsilon)$.
\end{proof}

\begin{rem}
$(i)$.  For $\mathbb{Z}$-actions,  the variational principles\footnote[6]{A summary of  variational principles that relate mean dimensions to different types of measure-theoretic $\epsilon$-entropies of $\mathbb{Z}$-actions can be found in \cite[Theorem 1.3]{ycz25} and \cite[Theorem 1.3]{y25}.} between metric mean dimension and the following 
\begin{itemize}
\item [(a)]  Kolmogorov-Sinai $\epsilon$-entropy $\inf_{
	\diam (\alpha)\leq \epsilon}\limits h_{\mu}(G,\alpha)$ \cite[Theorem 3.1]{gs20},
\item [(b)] Shapira's $\epsilon$-entropy $\inf_{\diam (\UU) \leq \epsilon}\limits h_{\mu}(G,\UU)$ \cite[Theorem 7]{shi},\\
\item [(c)] Katok's $\epsilon$-entropy $\underline{h}_{\mu}^K(\{F_n\}, \epsilon), \overline{h}_{\mu}^K(\{F_n\},\epsilon) $ \cite[Theorem 9, Proposition 19]{shi},\\
\item [(d)] Pfister and Sullivan's  $\epsilon$-entropy  $PS_\mu(\{F_n\},\epsilon)$ \cite[Theorem 1.3]{ycz25},
\end{itemize} 
where $G=\mathbb{Z}$ and $F_n=\{0,...,n-1\}$, have been already established. Here, Theorem \ref{thm 3.33} extends   the ones given in $\mathbb{Z}$-actions \cite{gs20, shi, ycz25} to actions of amenable groups.

$(ii)$. Within the framework of the action of amenable groups, the  variational principle for metric mean dimension  \cite[Theorem A]{l21} in terms of Kolmogorov-Sinai $\epsilon$-entropy  was established  under the assumptions of the F\o lner sequence $\{F_n\}$  with  $\lim_{n \to \infty}\frac{|F_n|}{\log n}=\infty$ and a certain condition imposed on the boundary of finite partitions.  By Theorem \ref{thm 3.33}, these additional conditions are redundant, and hence  can be removed for $G$-systems.
\end{rem}

\subsection{Thermodynamic formalism of $G$-systems with infinite topological entropy}

Motivated by  the classical thermodynamic formalism \cite{wal75,r04}, in this subsection we extend Theorem \ref{thm 1.1} to metric mean dimension with potential, and  characterize the  equilibrium states attaining the supremum of the variational principle.

\subsubsection{Variational principles for metric mean dimension with potential}
Let $(X, d, G)$ be a $G$-system. Let $C(X,\mathbb{R})$ denote  the space of  all continuous real-valued functions on $X$,  endowed with the supremum norm  $\|\cdot\|$.  Given a non-empty subset $Z\subset X$, $f\in C(X,\mathbb{R})$  and  $F\in \FF(G)$, we write $S_{F}f(x)=\sum_{g\in F}f(gx)$ as the Birkhoff sum of $f$ along the orbit-segment $\{gx: g\in F\}$ of $x$, and put
 $$P(Z,f,d,\{F_n\},\epsilon)=\limsup_{n \to \infty} \frac{1}{|F_n|}\log P_{F_n}(Z,f,d,\epsilon),$$
 where $P_F(Z,f,d,\epsilon)=\sup\{\sum_{x\in E}e^{S_Ff(x)}: E~\mbox{is a}~ (d_F,\epsilon)\mbox{-separated set of}~Z\}.$
 
 Inspired by the definition of metric mean dimension with potential for $\mathbb{Z}$-actions \cite{t20},  for actions of amenable groups the corresponding \emph{upper and lower metric mean dimensions of $X$ with potential $f$} \cite{twl20}   are respectively given by
\begin{align*}
{\overline{\rm  mdim}}_M(G,X,f,d)&:=\limsup_{\epsilon \to 0}\frac{1}{\logf}P(X,f\logf,d,\{F_n\},\epsilon),\\
{\underline{\rm  mdim}}_M(G,X,f,d)&:=\liminf_{\epsilon \to 0}\frac{1}{\logf}P(X,f\logf,d,\{F_n\},\epsilon).
\end{align*}

A standard approach shows that the above definitions are independent of the choice of  F\o lner sequences of $G$.
If we define
\begin{align*}
f\in C(X,\mathbb{R})\mapsto  \overline{\Gamma}(f):={\overline{\rm  mdim}}_M(G,X, f,d),\\
f\in C(X,\mathbb{R})\mapsto  \underline{\Gamma}(f):={\underline{\rm  mdim}}_M(G,X, f,d),
\end{align*}
then the two functions possess the following properties:

\begin{prop}\label{prop 3.8}
Let $(X, d, G)$ be a $G$-system and $f,h\in C(X,\mathbb{R})$. Then
\begin{enumerate}
\item  if  $f\leq h$, then $ \overline\Gamma (f)\leq \overline\Gamma (h)$;
\item for all $c\in \mathbb{R}$, $\overline\Gamma(f+c)=\overline\Gamma(f)+c$;
\item  $ {\overline{\rm  mdim}}_M(G,X,d) -\|f\| \leq  \overline\Gamma(f) \leq  {\overline{\rm  mdim}}_M(G,X, d)+\|f\|$;
\item the function  $\overline{\Gamma}(f) $ is either finite or  always infinite;
\item if $\overline\Gamma(f)$ is finite, then $|\overline\Gamma(f)-\overline\Gamma(h)|\leq \|f-h\|$;
\item for any $g\in G$,
$\overline\Gamma(f)=\overline\Gamma(f+h\circ T_g-h)$;
\item for any $p\in [0,1]$, $\overline\Gamma(pf+(1-p)h)\leq p \overline\Gamma(f)+(1-p)\overline\Gamma (h)$;
\item $\overline\Gamma(f+h)\leq \overline\Gamma(f)+\overline\Gamma(h)$;
\item  if  $c\geq 1$, then $\overline\Gamma(cf)\leq c\overline\Gamma(f)$; if  $c\leq 1$, then $\overline\Gamma(cf)\geq c\overline\Gamma(f)$, where we obey the convention $0\cdot \infty:=0$;
\item $|\overline\Gamma(f)|\leq\overline\Gamma(|f|)$.
\end{enumerate}
\end{prop}
\begin{proof}
These statements can be verified by following the arguments in  \cite[\S 9, Theorem 9.7]{w82} and using the  inequality:
$$\limsup_{x\to x_0}(f(x)+g(x))\leq \limsup_{x\to x_0} f(x)+\limsup_{x\to x_0}g(x)$$
for any two real-valued functions $f,g$ on $X$.  
\end{proof}

\begin{rem}
All properties mentioned in Proposition  \ref{prop 3.8} are valid for $\underline \Gamma(f)$ except the properties $(7)$ and $(8)$.
\end{rem}

The following Theorem is an extension of the variational principle for metric mean dimension  to the continuous potentials.

\begin{thm}\label{thm 4.6}
Let $(X, d, G)$ be a $G$-system, and   let $\{F_n\}$ be a tempered F\o lner sequence of $G$.  Then for every $f\in C(X,\mathbb{R})$,
\begin{align*}
{\overline{\rm  mdim}}_M(G,X,f,d)&=\max_{\mu \in M(X,G)}\left\{\overline{\rm {mdim}}_M(\mu,\{F_n\},d)+\int f d\mu\right\}.
\end{align*}
\end{thm}

The proof of Theorem \ref{thm 4.6} is similar to  Theorem \ref{thm 1.1}. Hence, we only give a sketch of two auxiliary  lemmas and omit the precise proof.

\begin{lem}
Let $(X, d, G)$ be a $G$-system  and  $\{F_n\}$ be a F\o lner sequence of $G$, and let $f\in C(X,\mathbb{R})$. Then for every  $\epsilon  >0$, there exist $\mu_{\epsilon} \in E(X,G)$ and $ \nu_{\epsilon}\in M(X,G)$ such that
\begin{align*}
P(X,f,d,\{F_n\},\epsilon)&
\leq \inf_{\diam (\alpha)\leq \frac{\epsilon}{8}}h_{\nu_{\epsilon}}(G,\alpha)+ \int fd\nu_{\epsilon}\\
&\leq \inf_{\diam (\alpha)\leq \frac{\epsilon}{64}}h_{\mu_{\epsilon}}(G,\alpha)+ \int fd\mu_{\epsilon}.
\end{align*}
\end{lem}

\begin{lem}
Let $(X, d, G)$ be a $G$-system  and  $\{F_n\}$ be a  tempered F\o lner sequence of $G$, and let $f\in C(X,\mathbb{R})$. Then for all  $\epsilon >0$ and  $\mu \in E(X,G)$,
\begin{align*}
&\lim_{\delta \to 0} \inf\{P(Z,f,d,\{F_n\},\epsilon):\mu(Z)\geq 1-\delta\}\\
=& \lim_{\delta \to 0} \inf\{h_{top}(G, Z,d,\{F_n\},\epsilon):\mu(Z)\geq 1-\delta\} +\int fd\mu.
\end{align*}
\end{lem}

\begin{proof}
It suffices to show 
\begin{align*}
&\lim_{\delta \to 0} \inf\{h_{top}(G, Z,d,\{F_n\},\epsilon):\mu(Z)\geq 1-\delta\} +\int fd\mu\\
\leq & \lim_{\delta \to 0} \inf\{P(Z,f,d,\{F_n\},\epsilon):\mu(Z)\geq 1-\delta\},
\end{align*}
and  the reverse inequality can be proved in a similar manner.

Let  $\delta \in (0,1)$ and $E$ be a Borel set of $X$  with $\mu(E)\geq  1-\frac{\delta}{2}$. By Lindenstrauss's pointwise ergodic theorem  \cite[Theorem  1.2]{lin01},  for $\mu$-a.e. $x\in X$ one has
$\lim_{n\to \infty}\frac{1}{|F_n|}S_{F_n}f(x)=\int fd\mu.$
Fix $\gamma >0$. For every $N\geq 1$, we define $$X_N:=\{x\in X: |\frac{1}{|F_n|}S_{F_n}f(x)-\int fd\mu|<\gamma~ \forall  n\geq N\}.$$
Then $X_N$ is non-decreasing in $N$, and the union  $\cup_{N\geq 1}X_N$  has  full $\mu$-measure. Then there exists $N_0$ such that $\mu(X_{N_0})\geq 1-\frac{\delta}{2}$. So $\mu(X_{N_0}\cap E)\geq 1-\delta$. For every $n\geq N_0$, if $R$ is a $(d_{F_n},\epsilon)$-separated set of $X_{N_0}\cap E$ with the largest cardinality, then
\begin{align*}
P_{F_n}(E,f,d,\epsilon)\geq  \sum_{x\in R}e^{S_{F_n}f(x)} > e^{|F_n|\cdot (\int fd\mu -\gamma)}\cdot s(X_{N_0}\cap E,d_{F_n},\epsilon).
\end{align*}
This yields that
\begin{align}\label{equ 2.12}
\begin{split}
&\inf\{h_{top}(G, Z,d,\{F_n\},\epsilon):\mu(Z)\geq 1-\delta\} +\int fd\mu-\gamma\\
\leq&  h_{top}(G, X_{N_0}\cap E,d,\{F_n\},\epsilon) +\int fd\mu -\gamma\\
\leq & P(E,f,d,\{F_n\},\epsilon).
\end{split}
\end{align}
Letting $\gamma \to 0$, taking infimum over Borel sets $E$ with $\mu(E)\geq 1-\frac{\delta}{2}$ and then letting $\delta \to 0$  on  both sides of the above inequality (\ref{equ 2.12}), we get  the desired inequality.
\end{proof}


\begin{cor}\label{thm 3.39}
Let $(X, d, G)$ be a $G$-system such that ${\overline{\rm  mdim}}_M(G,X,d)=0$,  and let $\{F_n\}$ be a tempered F\o lner sequence of $G$. Then for any  $f\in C(X,\mathbb{R})$,
\begin{itemize}
\item [(1)] ${\overline{\rm  mdim}}_M(G,X,f,d)=\max_{\mu \in E(X,G)}\limits \int f d\mu$;
\item [(2)] there is a  Borel probability measure $\tau$ on $M(X,G)$ with $\tau (E(X,G))=1$ such that for $\tau$-a.e. $m\in E(X,G)$,
$$\lim_{n \to \infty} \frac{1}{|F_n|}\sum_{g\in F_n}f(gx)=\overline{\rm mdim}_M(G,X,f,d)$$
for $m$-a.e. $x\in X$.
\end{itemize}
\end{cor}
\begin{proof}
By Theorem \ref{thm 4.6}, we  assume that $\mu$ maximizes the supremum of the variational principle. Let $\mu =\int_{E(X,G)} m d\tau(m)$ be the  ergodic decomposition of $\mu$, where $\tau$ is a Borel probability measure on $M(X,G)$ and $\tau(E(X,G))=1$. Then   $$\overline{\rm mdim}_M(G,X,f,d)=\int f d\mu=\int_{E(X,G)} \int f dm~ d\tau(m).$$
Hence, we  have $\overline{\rm mdim}_M(G,X,f,d)=\int fdm$ for   $\tau$-a.e. $m\in E(X,G)$. This shows (1).  Together with the pointwise ergodic theorem \cite{lin01},  the remaining statement  is verified.
\end{proof}

\subsubsection{Metric mean dimension with potential determines the invariant measures}

It follows from Theorem \ref{thm 1.1} that the metric mean dimension is determined  by the invariant measures of $G$-systems and the  measure-theoretic metric mean dimension of these measures. Conversely, using Theorem \ref{thm 4.6}, we prove that the metric mean dimension with potential also determines the invariant measures and their  measure-theoretic metric mean dimension.

Recall that a finite signed  measure  $\mu$ on $X$ is a  real-valued function defined on $\mathcal{B}(X)$ satisfying the countably additive property.

\begin{thm}\label{thm 4.13}
Let $(X, d, G)$ be a $G$-system  such that $\overline{\rm {mdim}}_M(G,X,d)<\infty$, and let $\mu:\mathcal{B}(X)\rightarrow \mathbb{R}$ be a finite signed measure. Then 
\begin{align*}
\mu\in M(X,G)  &\Leftrightarrow \int f d\mu\leq {\underline{\rm  mdim}}_M(G,X,f,d) ~\forall f\in C(X,\mathbb{R})\\
&\Leftrightarrow \int f d\mu\leq {\overline{\rm  mdim}}_M(G,X,f,d) ~\forall f\in C(X,\mathbb{R}).
\end{align*}
\end{thm}

\begin{proof}
Notice that for every $f\in C(X,\mathbb{R})$,
\begin{align}\label{inequ 4.8}
\max_{\mu \in M(X,G)}\left\{\underline{\rm {mdim}}_M(\mu,\{F_n\},d)+\int f d\mu\right\}\leq {\underline{\rm  mdim}}_M(G,X,f,d).
\end{align} 
 If $\mu\in M(X,G)$, we have  $\int f d\mu\leq {\underline{\rm  mdim}}_M(G,X,f,d)$ for all $f\in C(X,\mathbb{R})$.  

Now assume that $\int f d\mu\leq {\overline{\rm  mdim}}_M(G,X,f,d)$ for all $f\in C(X,\mathbb{R})$.
We  show  that $\mu$ only takes non-negative values on $\mathcal{B}(X)$.  Let  $f\geq 0$ and $\epsilon>0$.  Let $n\in \mathbb{N}$  be sufficiently large such that  $-\overline{\rm {mdim}}_M(G,X,d)+ n\epsilon>0$. Then
\begin{align*}
\int n(f+\epsilon)d\mu&=-\int -n(f+\epsilon)d\mu\geq -{\overline{\rm  mdim}}_M(G,X,-n(f+\epsilon),d)\\
&= -{\overline{\rm  mdim}}_M(G,X, -nf,d)+n\epsilon,~ \text{by Proposition \ref{prop 3.8}, (2)}\\
&\geq-\overline{\rm {mdim}}_M(G,X,d)+ n\epsilon>0, ~ \text{by Proposition \ref{prop 3.8}, (1)}.
\end{align*}
Letting  $\epsilon \to 0$, we have $\int f d\mu\geq 0$. Hence, $\mu$  only takes  non-negative  values. 

We need to show  that $\mu \in M(X,G)$.  If $n\in \mathbb{Z}$,  then $$n\cdot \mu(X)=\int n d\mu\leq {\underline{\rm  mdim}}_M(G,X,n,d)={\underline{\rm  mdim}}_M(G,X,d)+ n.$$ 
This implies that  $\mu(X)\leq 1$ by  letting $n \to \infty$ and  $\mu(X)\geq 1$ by letting $n \to -\infty $.  Hence, $\mu \in M(X)$.

If $n\in \mathbb{Z}$ and $f\in C(X,\mathbb{R})$, then,  by Proposition \ref{prop 3.8},(6) for all $g\in G$ $$ n\int (f\circ T_g-f)d\mu\leq   \underline{\rm {mdim}}_M(G,X,n(f\circ T_g-f),d)= \underline{\rm {mdim}}_M(G,X,d).$$  Letting $n \to \infty$,  we have $\int (f\circ T_g-f)d\mu\leq0$ for all $g\in G$,  which implies   $\mu\in M(X,G)$. 
\end{proof}

Next, we give a characterization  of the lower measure-theoretic metric mean dimension using  the invariant measures of $G$-systems and  the lower metric mean dimension with potential. We invoke  a \emph{separation theorem for convex sets}  presented in \cite[p. 417]{ds88}.

\begin{lem}\label{lem 4.3}
Let $V$ be  a locally convex  real linear Hausdorff
topological space. If $K_1,K_2$ are  two disjoint closed convex subsets of $V$ and  $K_1$ is compact, then there exists a continuous real-valued linear functional $L$ on $V$ such that $L(x)<L(y)$ for all $x\in K_1$ and $y\in K_2$.
\end{lem}

\begin{thm}\label{thm  3.35}
Let $(X, d, G)$ be a $G$-system  such that $\underline{\rm {mdim}}_M(G,X,d)<\infty$, and  let $\{F_n\}$ be a  tempered F\o lner sequence of $G$. If $${\underline{\rm  mdim}}_M(G,X,f,d)=\max_{\mu \in M(X,G)}\left\{\underline{\rm {mdim}}_M(\mu,\{F_n\},d)+\int f d\mu\right\}$$ for all $f\in C(X,\mathbb{R})$, then  for every  $\mu_0\in M(X,G)$,
$$\underline{\rm {mdim}}_M(\mu_0,\{F_n\},d) =\inf_{f\in C(X,\mathbb{R})}\lbrace {\underline{\rm  mdim}}_M(G,X,f,d)-\int f d\mu_0\rbrace.$$ 
\end{thm}

\begin{proof}
 By (\ref{inequ 4.8}), it holds that
$$\underline{\rm {mdim}}_M(\mu_0,\{F_n\},d) \leq\inf_{f\in C(X,\mathbb{R})}\lbrace {\underline{\rm  mdim}}_M(G,X,f,d)-\int f d\mu_0\rbrace.$$ 

Let $b>\underline{\rm {mdim}}_M(\mu_0,\{F_n\},d)$. We  define  $$C=\lbrace (\mu,t)\in M(X,G)\times\mathbb{R}: 0\leq t\leq \underline{\rm {mdim}}_M(\mu,\{F_n\},d)\rbrace.$$
One can think of  $C$ as a subset of $C(X,\mathbb{R})^\ast\times\mathbb{R}$, where  the dual space $C(X,\mathbb{R})^\ast$ is endowed with  the weak$^*$-topology. By Proposition \ref{prop 3.26}, the upper semi-continuity of $\underline{\rm {mdim}}_M(\cdot,\{F_n\},d)$ at $\mu_0$ implies that $(\mu_0,b)\notin \overline{C}$. By Definition  \ref{df 3.26},  $ \underline{\rm {mdim}}_M(\mu,\{F_n\},d)$  is  concave on $M(X,G)$. Then $C$ is a convex set.   By  Lemma \ref{lem 4.3},  there exists a continuous linear functional $F:C(X,\mathbb{R})^\ast\times\mathbb{R}\rightarrow\mathbb{R}$ such that $F((\mu,t))<F((\mu_0,b))$ for all $(\mu,t)\in \overline{C}$. In this case,  $F$ must have  the form
$$F(\mu,t)=\int f d\mu +tm$$ for some $f\in C(X,\mathbb{R})$ and  $m\in \mathbb{R}$. Notice that $(\mu_0,\underline{\rm {mdim}}_M(\mu_0,\{F_n\},d)) \in C$.  We have  $$\int f d\mu_0+m \cdot \underline{\rm {mdim}}_M(\mu_0,\{F_n\},d)<\int f d\mu_0+mb.$$ This shows that  $m>0$, and hence
\begin{align*}
\underline{\rm {mdim}}_M(\mu,\{F_n\},d)+\int\frac{f}{m} d\mu <b +\int\frac{f}{m} d\mu_0
\end{align*}
for all  $\mu\in M(X,G)$. This yields that $\underline{\rm  mdim}_M(G,X,\frac{f}{m},d) \leq b+ \int\frac{f}{m} d\mu_0.$
Thus, we have
\begin{align*}
b\geq &\underline{\rm  mdim}_M(G,X,\frac{f}{m},d)-\int\frac{f}{m} d\mu_0\\
\geq &\inf_{g\in C(X,\mathbb{R})}\lbrace \underline{\rm  mdim}_M(G,X,g,d)-\int g d\mu_0\rbrace.
\end{align*}
Letting $b \to \underline{\rm {mdim}}_M(\mu_0,\{F_n\},d)$, this completes the proof.
\end{proof}

Using the fact that for all $c\in \mathbb{R}$, $\underline\Gamma(f+c)=\underline\Gamma(f)+c$, we have
$$\inf_{g\in C(X,\mathbb{R})}\lbrace \underline{\rm  mdim}_M(G,X,g,d)-\int g d\mu_0\rbrace=\mathcal{H}(\mu):=\inf_{f\in \mathcal{A}}\int f d\mu,$$
where $\mathcal{A}:=\{f\in C(X,\mathbb{R}): \underline{\rm  mdim}_M(G,X,-f,d)=0\}$. Therefore, Theorem \ref{thm  3.35} provides a  criterion whether the variational principle holds  for the under metric mean dimension with potential.
\begin{cor}
Let $(X, d, G)$ be a $G$-system  such that $\underline{\rm {mdim}}_M(G,X,d)<\infty$, and  let $\{F_n\}$ be a  tempered F\o lner sequence of $G$. If there exists   $\mu_0\in M(X,G)$ such that 
$\underline{\rm {mdim}}_M(\mu_0,\{F_n\},d) \not=\mathcal{H}(\mu_0),$ 
then  $$\max_{\mu \in M(X,G)}\left\{\underline{\rm {mdim}}_M(\mu,\{F_n\},d)+\int f d\mu\right\}< {\underline{\rm  mdim}}_M(G,X,f,d)$$ for some $f\in C(X,\mathbb{R})$.
\end{cor}

\subsubsection{A characterization of measures attaining equilibrium states}
In this subsection, we give the  characterizations of the invariant measures that attain the supremum  of  the variational principle  stated in Theorem \ref{thm 4.6}.

The following notion  is an  analogue of the concept of \emph{equilibrium state}  of  topological pressure in  classical thermodynamic formalism \cite{wal75, w82}.

\begin{df}
Let $(X, d, G)$ be a $G$-system  and $\{F_n\}$ be a F\o lner sequence of $G$, and let $f\in C(X,\mathbb{R})$.  A measure $\mu \in M(X,G)$ is called  an equilibrium state for $f$  if $${\overline{\rm  mdim}}_M(G,X, f,d)=\overline{\rm {mdim}}_M(\mu,\{F_n\},d)+\int f d\mu.$$ 

Specially, if $f=0$ is a zero potential, we call $\mu$ a \emph{maximal metric mean dimension measure} if ${\overline{\rm  mdim}}_M(G,X,d)=\overline{\rm {mdim}}_M(\mu,\{F_n\},d)$. 
\end{df}

Denote by  $M_{f}(G,X,d)$ the set of  all equilibrium states for $f$. For the case of zero potential, we let  $M_{max}(G,X,d):=M_{0}(G,X,d)$. Clearly, if ${\overline{\rm  mdim}}_M(G,X,d)=\infty$, then  $$M_{f}(G,X,d)=\{\mu \in M(X,G):\overline{\rm {mdim}}_M(\mu,\{F_n\},d)=\infty\}={\overline{\rm  mdim}}_M(G,X,d)$$ for all $f\in C(X,\mathbb{R})$. 
\begin{rem}
The non-emptiness of $M_{f}(G,X,d)$  is guaranteed by Theorem \ref{thm 4.6}. Besides, the proof of Theorem \ref{thm 1.1} tells us a way to  find such equilibrium state measures.
\end{rem}

Actually,  the equilibrium state is  closely tied with the notion  of \emph{tangent functional} to  the convex function  ${\overline{\rm  mdim}}_M(G,X, \cdot,d)$.

\begin{df}
Let $(X, d, G)$ be a $G$-system such that ${\overline{\rm  mdim}}_M(G,X,d)<\infty$,  and let  $f\in C(X,\mathbb{R})$. A tangent functional to ${\overline{\rm  mdim}}_M(G,X,\cdot,d)$ at $f$  is a  finite signed measure $\mu$ on $X$ such that for all $g\in C(X,\mathbb{R})$,
$${\overline{\rm  mdim}}_M(G,X, f+g,d)-{\overline{\rm  mdim}}_M(G,X, f,d)\geq \int gd\mu.$$
\end{df}
Denote  by $t_{f}(G,X,d)$ the set of  all tangent functionals of $f$. 
\begin{rem}
As  is  shown in  \cite[p.225, Remark]{w82},  by  Riesz representation  theorem,  one can think of  the  tangent functional $\mu$ to ${\overline{\rm  mdim}}_M(G,X, \cdot,d)$ at $f$  as  an element $L$  of the dual space $C(X,\mathbb{R})^{*}$ of the continuous function space such that for all $g\in C(X,\mathbb{R})$,
$$L(g)\leq {\overline{\rm  mdim}}_M(G,X, f+g,d)-{\overline{\rm  mdim}}_M(G,X, f,d).$$  
Using this relation,  $t_{f}(G,X,d)$ is non-empty by  applying the Hahn-Banach theorem (cf. \cite [Appendix, A.3.2]{r04}).
\end{rem}

\begin{thm}
Let $(X, d, G)$ be a $G$-system such that ${\overline{\rm  mdim}}_M(G,X,d)<\infty$,  and let $\{F_n\}$ be a tempered F\o lner sequence of $G$. Let  $f\in C(X,\mathbb{R})$. Then the following statements hold:
\begin{itemize}
\item  [(1)] The set  $M_{f}(G,X,d)$ is a  non-empty closed subset of $M(X,G)$, and $M_{f}(G,X,d)= \cap_{n\geq 1}L_n$, where 
\begin{align*}
L_n:= \overline{\{\mu \in M(X,G):  \overline{\rm {mdim}}_M(\mu,\{F_n\},d)+\int fd\mu>{\overline{\rm  mdim}}_M(G,X, f,d)-\frac{1}{n}\}}.
\end{align*}

\item [(2)] $M_{f}(G,X,d)\subset t_{f}(G,X,d)\subset M(X,G).$

\item [(3)] There exists a dense  subset $\mathcal{D}$ of $C(X,\mathbb{R})$  such that for any $f\in \mathcal{D}$, $M_{f}(G,X,d)$  has  a unique equilibrium  state.
\end{itemize}
\end{thm}

\begin{proof}
$(1)$. It follows from  Theorem \ref{thm 4.6} and the upper semi-continuity of $\overline{\rm {mdim}}_M(\cdot,\{F_n\},d)$ on $M(X,G)$ presented in  Proposition  \ref{prop 3.26}. 

$(2)$. Let $\mu  \in M_{f}(G,X,d)$. Then for every $g\in C(X,\mathbb{R})$, \begin{align*}
 &{\overline{\rm  mdim}}_M(G,X, f+g,d)-{\overline{\rm  mdim}}_M(G,X, f,d)\\
  \geq& (\overline{\rm {mdim}}_M(\mu,\{F_n\},d)+\int f+g d\mu)-(\overline{\rm {mdim}}_M(\mu,\{F_n\},d)+\int f d\mu)\\
  =&\int g d\mu.
\end{align*}
This shows that $\mu \in t_{f}(G,X,d) $, and hence  $M_{f}(G,X,d)\subset t_{f}(G,X,d)$.

 Now assume that  $\mu \in  t_{f}(G,X,d)$. Then 
for all $g\in C(X,\mathbb{R})$,
\begin{align*}
\int gd\mu&\leq {\overline{\rm  mdim}}_M(G,X, f+g,d)-{\overline{\rm  mdim}}_M(G,X, f,d)\\
&\leq   {\overline{\rm  mdim}}_M(G,X, g,d), ~ \text{by Proposition \ref{prop 3.8}, (8)}.
\end{align*}
Then $\mu \in M(X,G)$ by Theorem \ref{thm 4.13}, and hence $t_{f}(G,X,d)\subset M(X,G)$.

$(3)$. Recall a theorem in \cite[p.450]{ds88} that  a convex function on a separable Banach space has  a  unique  tangent  functional at a  dense set  of points. Applying this  fact to  the convex function ${\overline{\rm  mdim}}_M(G,X,\cdot,d)$ on $C(X,\mathbb{R})$,  and noticing that $M_{f}(G,X,d)$ is not empty and $ M_{f}(G,X,d)\subset t_{f}(G,X,d)$ by (2),  this  implies   (3).
\end{proof}

\section{Two applications of the variational principle}\label{sec 5}

In this section, we apply the variational principle to  study $G$-systems with zero metric mean dimension and those with positive metric mean dimension, and prove Theorems \ref{thm 1.3} and  \ref{thm 1.4}.

\subsection{Infinite entropy systems with zero metric mean dimension}\label{subsec 5.1}
In this subsection, for $G$-systems with zero metric mean dimension, we introduce the concept of infinite  entropy dimensions in both topological and measure-theoretic situations, and prove Theorem \ref{thm 1.3}.

\subsubsection{Infinite  entropy dimension of subsets}


To distinguish the zero entropy systems, Carvalho \cite{c97}  first introduced the concept of \emph{entropy dimension} by employing  a sub-exponentially growing function in the definition of topological entropy for $\mathbb{Z}$-actions. This definition has a flavour of Hausdorff dimension  in fractal geometry. More precisely,
given  a non-empty subset $Z\subset X$ and  a F\o lner sequence $\{F_n\}$ of $G$, the $s$-topological entropy of $Z$ w.r.t. $\{F_n\}$ is defined by
$$h_{top}(G,Z,s,\{F_n\})=\lim_{\epsilon \to 0}\limsup_{n \to \infty}\frac{\log s(Z,d_{F_n},\epsilon)}{|F_n|^s}.$$ Then  the  \emph{entropy dimension of $Z$} is  defined by
\begin{align*}
D(G,Z,\{F_n\}):&=\inf\{s>0: h_{top}(G,Z,s,\{F_n\})=0\},\\
&=\sup\{s>0: h_{top}(G,Z,s,\{F_n\}) =\infty\}.
\end{align*}

One may think of the parameter \( s \) as a number which will be determined later. For instance, for zero entropy systems, we decrease \( s \) until we find the largest lower bound of the set \( \{s>0 \mid h_{\text{top}}(G,Z,s,\{F_n\})=0\} \). However, for infinite entropy systems, we increase \( s \) until we find the smallest upper bound of the set \( \{s>0 \mid h_{\text{top}}(G,Z,s,\{F_n\})=\infty\} \). Then we can use  the entropy dimension   to classify the systems with zero  and infinite entropy systems.

Now we inject this idea into \( G \)-systems that are infinite entropy systems with zero metric mean dimension by introducing the so-called \emph{infinite entropy dimension}. The definition is a fusion of the    Hausdorff dimension  and the box dimension  in fractal geometry. It allows us to  understand the dynamics  hidden in  infinite entropy systems with zero metric mean dimension. 

Let $Z$ be a non-empty subset of $X$,  and let $\{F_n\}$ be a  F\o lner sequence of $G$. 
For $s>0$, we  define \emph{$s$-upper metric mean dimension of $Z$ w.r.t. $\{F_n\}$} as 
\begin{align*}
{\overline{\rm  mdim}}_M(G,Z, \{F_n\},s,d)&=\limsup_{\epsilon \to 0}\frac{h_{top}(G,Z,d,\{F_n\},\epsilon)}{(\logf)^s},\\
&=\limsup_{\epsilon \to 0}\frac{\hat{h}_{top}(G,Z,d,\{F_n\},\epsilon)}{(\logf)^s}.
\end{align*}

In particular, if $s=1$,  the above notion reduces to the upper metric mean dimension of $Z$. Observe that $ {\overline{\rm  mdim}}_M(G,Z, \{F_n\},s,d)$ is  non-negative and non-increasing in $s$. There exists $s_0\in [0,\infty]$ such that
\begin{align*}
{\overline{\rm  mdim}}_M(G,Z, \{F_n\},s,d)=
\begin{cases}
\infty,  &\mbox{if}~0<s< s_0\\
0,&\mbox{if}~s>s_0
\end{cases}.
\end{align*} 
This fact allows us to introduce the following concept:

\begin{df}
The infinite  upper entropy dimension  of $Z$ w.r.t. $\{F_n\}$ is defined by
\begin{align*}
\overline{D}_M(G,Z,\{F_n\},d)&=\inf\{s>0:  {\overline{\rm  mdim}}_M(G,Z, \{F_n\},s,d)=0\}\\
&=\sup\{s>0:  {\overline{\rm  mdim}}_M(G,Z, \{F_n\},s,d)=\infty\}.
\end{align*}
\end{df}
Using the $s$-lower metric mean dimension ${\underline{\rm  mdim}}_M(G,Z, \{F_n\},s,d)$ of $Z$,  the similar approach is  applied to  define  the \emph{infinite lower entropy dimension}  $\underline{D}_M(G,Z,\{F_n\},d)$ of $Z$. If $Z\subset X$ is a $G$-invariant subset,   the infinite upper and lower entropy dimensions are independent of the choice of  F\o lner sequences of $G$. For this case, we sometimes write $\overline{D}_M(G,Z,d)$ and $\underline{D}_M(G,Z,d)$ by dropping  the  dependence $\{F_n\}$  in  $\overline{D}_M(G,Z,\{F_n\},d)$ and $\underline{D}_M(G,Z,\{F_n\},d)$.

One immediately relates the infinite entropy dimension with the  metric mean dimension of a subset by comparing their definitions.

\begin{thm}\label{thm 3.41}
Let $(X, d, G)$ be a $G$-system,  and let $\{F_n\}$ be a F\o lner sequence of $G$.  Let $Z$ be  a non-empty subset of $X$. Then 
\begin{itemize}
\item [(1)]  if ${\overline{\rm  mdim}}_M(G,Z, \{F_n\},d)=\infty$, then  $ \overline{D}_M(G,Z,\{F_n\},d)\geq 1$;
\item [(2)]  if ${\overline{\rm  mdim}}_M(G,Z, \{F_n\},d)<\infty$, then  $0\leq \overline{D}_M(G,Z,\{F_n\},d) \leq 1$;
\item [(3)]if $0<{\overline{\rm  mdim}}_M(G,Z, \{F_n\},d)<\infty$, then  $ \overline{D}_M(G,Z,\{F_n\},d)= 1$.
\end{itemize}
\end{thm}

We verify the definition of infinite entropy dimension by the following examples. 

\begin{ex}\label{ex 4.4}
\item [(1)]  The infinite entropy dimension of any \( G \)-system with finite topological entropy is zero. For example, the \( G \)-full shift over finite symbols and expansive \( G \)-systems are such cases.
\item [(2)] Let $X=[0,1]$, $d=|\cdot|$ and $D$ as in Example \ref{ex 3.6}.  We  have $$\overline{D}_M(G,[0,1]^G,D)=  1.$$ 
\item [(3)]  For any $0<s < 1$, there is a  $G$-system $(X,d,G)$ with infinite entropy and $\overline{{\rm {mdim}}}_M(G, X,d)=0$, while $\overline{D}_M(G,X,d)=s$. 
\end{ex}

\begin{proof}

$(1)$. It  holds by the fact that  the $\epsilon$-topological entropy of $X$  is  uniformly bounded by a constant for all $\epsilon>0$.

$(2)$. Since $\overline{{\rm {mdim}}}_M(G, [0,1]^G,D)=\limsup_{\epsilon \to 0}\frac{1}{\logf}h_{top}(G,[0,1]^G,D,\{F_n\},\epsilon)=1,$ we may think of $h_{top}(G,[0,1]^G,D,\{F_n\},\epsilon)\approx \logf$.
Then for any $s>1$, we  have  $$\overline{{\rm {mdim}}}_M(G, [0,1]^G, s, D)=0,$$ and hence  $\overline{D}_M(G,[0,1]^G,\{F_n\},D)\leq 1$. For any $s<1$, we  have  $\overline{{\rm {mdim}}}_M(G, [0,1]^G, s, D)=\infty$, and hence $\overline{D}_M(G,[0,1]^G,\{F_n\},D)\geq 1$. So   $\overline{D}_M(G,[0,1]^G,D)=1$.

$(3)$. Let $(X,d)$ be a compact metric space. 
Similar to the proof of Example \ref{ex 3.6},  for sufficiently small  $\epsilon >0$, one  has $h_{top}(G,X^G,D,\{F_n\},\epsilon)\approx\log N_{\epsilon}(X,d),$
where $N_{\epsilon}(X,d)$ is  the smallest cardinality of the  open balls $B_d(x,\epsilon)$ needed to cover $X$. Therefore,   for any $s>0$ we have
\begin{align*}
\overline{{\rm {mdim}}}_M(G, X^G,s,D)&=\limsup_{\epsilon \to 0}\frac{\log N_{\epsilon}(X,d)}{(\logf)^s}.
\end{align*}
Using this  fact, for any fixed $s\in (0,1)$ and $0<\alpha <\infty$ we shall construct a compact subset $X \subset [0,1]$ such that $\overline{\rm dim}_B(X,d)=0$ and  $\limsup_{\epsilon \to 0}\frac{\log N_{\epsilon}(X,d)}{(\logf)^s}=\alpha$, and  hence  $\overline{{\rm {mdim}}}_M(G, X^G,D)=0$,  $\overline{D}_M(G,X^G,D)=s.$

Now fix such  $0<s<1$ and $0<\alpha<\infty$. For any $k\geq 1$, let $\epsilon_k= e^{-k^{\frac{1}{s}}}<1$ and  $M_k=\lfloor e^{\alpha k} \rfloor$, where $\lfloor x \rfloor:=\max\{n\in \mathbb{Z}: n\leq x\}$. Fix a $k_0$ sufficiently large such that $f(k):=\epsilon_k\cdot e^{\alpha k}<1$ and is decreasing for all $k\geq k_0$. Consider the set
$E_k:=\{j\epsilon_k: j=1,...,M_k\}$, which is   uniformly distributed in $[0, 1]$  closing to $0$ with  gap $\epsilon_k$. We define  
$$X=\{0\}\bigcup (\bigcup_{k\geq k_0} E_k),$$
which is  a closed subset of $[0,1]$. Given  a sufficiently small $\epsilon>0$, there exists  some $k\geq k_0$ such that $\epsilon_{k+1}<\epsilon\leq \epsilon_k$.   Then   $N_{\epsilon}(X,d)\geq N_{\epsilon_{k}}(E_{k},d)$. Thus, we have 
\begin{align*}
\limsup_{\epsilon \to 0}\frac{\log N_{\epsilon}(X,d)}{(\logf)^s}\geq \limsup_{\epsilon \to 0}\frac{ \log \lfloor e^{\alpha k} \rfloor}{k+1} \frac{k+1}{(\logf)^s}=\alpha.
\end{align*}

On the other hand, if we choose the open balls  centered at the points of $\cup_{j=k_0}^{k+1}E_{j}$ with radius $\epsilon_{k+1}$ to cover $X$,  then $N_{\epsilon}(X,d)$ is bounded above  by  $$N_{\epsilon}(X,d)\leq \sum_{j=k_0}^{k+1}e^{\alpha k}+1 <C\cdot e^{\alpha(k+1)}$$ for some $C>1$ independent of  $k$. Similarly, we obtain  $$\limsup_{\epsilon \to 0}\frac{\log N_{\epsilon}(X,d)}{(\logf)^s}\leq \limsup_{\epsilon \to 0}\frac{ \log (C\cdot e^{\alpha(k+1)})}{k} \frac{k}{(\logf)^s}=\alpha.$$
This shows $\limsup_{\epsilon \to 0}\frac{\log N_{\epsilon}(X,d)}{(\logf)^s}=\alpha$.
\end{proof}

Observe that if $h_{top}(G,Z,\{F_n\})=\infty$ then $D(G,Z,\{F_n\})\geq 1$. It is interesting to compare  the two types of  entropy dimensions  in the framework of infinite entropy systems.

\begin{thm}
Let $(X, d, G)$ be a $G$-system,  and let $\{F_n\}$ be a F\o lner sequence of $G$.  Let $Z$ be a non-empty subset of $X$. Then  
\begin{itemize}
\item [(1)] if  $h_{top}(G,Z,\{F_n\})=\infty$  and ${\overline{\rm  mdim}}_M(G,Z, \{F_n\},d)<\infty$, then
$$\overline{D}_M(G,Z,\{F_n\},d)\leq 1 \leq  D(G,Z,\{F_n\});$$
\item [(2)]  if $0<{\overline{\rm  mdim}}_M(G,Z, \{F_n\},d)<\infty$, then $$\overline{D}_M(G,Z,\{F_n\},d)= 1=D(G,Z,\{F_n\}).$$
\end{itemize}
\end{thm}

\begin{proof}
$(1)$. This is due to the definitions.

$(2)$.
The  fact $0<{\overline{\rm  mdim}}_M(G,Z, \{F_n\},d)<\infty$ implies that $h_{top}(G,Z,\{F_n\})=\infty$.
Hence, $D(G,Z,\{F_n\}) \geq 1=\overline{D}_M(G,Z,\{F_n\},d)$. Now it suffices to show $D(G,Z,\{F_n\})\leq 1$.

Fix $\gamma >0$ and let $a:={\overline{\rm  mdim}}_M(G,Z, \{F_n\},d)+1$. There exists $\epsilon_0>0$ such that for every $0<\epsilon<\epsilon_0$,  there is  a sufficiently  large $N_0$ (depending on $\epsilon$) such that
$\log s(Z,d_{F_n},\epsilon) < (a\logf)|F_n|$
for all $n\geq N_0$.  We have
$$\frac{\log s(Z,d_{F_n},\epsilon)}{|F_n|^{1+\gamma}}< \frac{a\cdot \logf}{|F_n|^{\gamma}}.$$
Notice that $\lim_{n \to \infty}|F_n|=\infty$ since $\{F_n\}$ is a F\o lner sequence. This implies that $h_{top}(G,Z,1+\gamma,\{F_n\})=0$, and hence $D(G,Z,\{F_n\})\leq 1+\gamma $. Letting $\gamma\to 0$, we have  $D(G,Z,\{F_n\}) \leq 1$.
\end{proof}

\begin{rem}
$(i)$. By  Example \ref{ex 4.4}, there exists  an infinite-entropy $G$-system  $(X,d,G)$ with zero  metric mean dimension  such that   $\overline{D}_M(G,X,\{F_n\},d)<1 \leq D(G,X,\{F_n\}).$ 

$(ii)$. The relation $\overline{D}_M(G,Z,\{F_n\},d)\leq  D(G,Z,\{F_n\})$ is not true for some systems with infinite  metric mean dimension.  For $\mathbb{Z}$-actions, Burguet and Shi  \cite[Corollary 19]{bs25} showed that if $(X,T)$ is a  positive entropy system with a Lipschitz map $T:X\rightarrow X$ on a  compact metric space $(X,d)$ with finite upper box dimension (e.g. the symbolic systems over finite symbols, the doubling map on unit interval $\mathbb{R}/\mathbb{Z}$), then there exists $\alpha >1$ such that for every  $0<\epsilon <1$, 
$$(\frac{1}{\epsilon})^{\frac{1}{\alpha}}\leq h_{top}(T_{*},M(X),W,\epsilon)\leq (\frac{1}{\epsilon})^{\alpha},$$
where $T_{*}:M(X)\rightarrow M(X)$ is the induced map and $W$ is the   Wasserstein distance on $M(X)$.

Clearly, such  induced systems have infinite metric mean dimension, and  the entropy dimensions of induced systems are $1$. However,  the infinite  upper entropy dimension are $\infty$ since $$\limsup_{\epsilon \to 0}\frac{h_{top}(T_{*},M(X),W,\epsilon)}{(\logf)^s} \geq \limsup_{\epsilon \to 0}\frac{(\frac{1}{\epsilon})^{\frac{1}{\alpha}}}{(\logf)^s}=\infty$$ 
for any $s>1$. 
\end{rem}

The following elementary properties exhibit how the infinite upper entropy dimension of subsets behaves with respect to Lipschitz factor maps, power systems, and product systems.

\begin{prop}
The following statements hold:

$(1)$ If $\pi: (X,d_X)\rightarrow (Y,d_Y)$ is a Lipschitz factor map between the $G$-systems $(X,d_X,G)$ and $(Y,d_Y,G)$, then
$$\overline{D}_M(G,X,d_X)\geq \overline{D}_M(G,Y,d_Y).$$
Additionally, if $\pi$ is a bi-Lipschitz conjugate map, then 
$$\overline{D}_M(G,X,d_X)= \overline{D}_M(G,Y,d_Y).$$

$(2)$ Let $(X, d, G)$ be a $G$-system. If $X$ is a finite union of   closed subsets $K_j$, $j=1,...,m$, then for any  F\o lner sequence $\{F_n\}$ of $G$,
$$\overline{D}_M(G,X,d)=\max_{1\leq j\leq m} \overline{D}_M(G,K_j,\{F_n\},d).$$

$(3)$ Let $(X, d, G)$ be a $G$-system. Suppose that  $H$ is a subgroup of $G$ with finite index in $G$. If for any finite complete set $R\subset G$ of  representatives   of left-cosets $\{gH:g\in G\}$ of $H$, $T|_{R\backslash\{e_G\}}$ are  Lipschitz maps, then
$$ \overline{D}_M(H,X,d)=\overline{D}_M(G,X,d).$$

$(4)$ Let $(X,d_X, G)$  and $(Y,d_Y,G)$ be two  $G$-systems. Then
\begin{align*}
\overline{D}_M(G,X\times Y,d_{X}\times d_Y)
=\max\{\overline{D}_M(G,X,d_{X}),\overline{D}_M(G,Y,d_{Y})\}.
\end{align*}
Consequently,  $\overline{D}_M(G,X\times X,d_{X}\times d_X)=\overline{D}_M(G,X,d_{X}).$
\end{prop}

\begin{proof}
We prove these statements one by one.

$(1)$.  Assume that $L>0$ is the Lipschitz constant of $\pi$ such that $d_Y(\pi(x),\pi(y))\leq L\cdot  d_X(x,y)$ for all $x,y  \in X$. Then  one has $$h_{top}(G,Y,d_Y,\{F_n\},\epsilon)\leq h_{top}(G,X,d_X,\{F_n\},\frac{\epsilon}{L}),$$ and hence ${\overline{\rm  mdim}}_M(G,Y,s,d_Y)\leq {\overline{\rm  mdim}}_M(G,X, s,d_X)$ for all $s> 0$. This implies  (1).

$(2)$. It is clear that $$\overline{D}_M(G,X,d)\geq \max_{1\leq j\leq m}\limits \overline{D}_M(G,K_j,\{F_n\},d).$$  On the other hand, one can show  that for every $s>0$,
$${\overline{\rm  mdim}}_M(G,X,s,d)= \max_{1\leq j \leq m}{\overline{\rm  mdim}}_M(G,K_j, \{F_n\},s,d).$$
Letting $s> \max_{1\leq j\leq m}\limits \overline{D}_M(G,K_j,\{F_n\},d)$, we have $${\overline{\rm  mdim}}_M(G,K_j, \{F_n\},s,d)=0$$ for every $1\leq j \leq m$, and hence ${\overline{\rm  mdim}}_M(G,X,s,d)= 0$. Thus, we have $\overline{D}_M(G,X,d) \leq s$. This shows the inequality
$$\overline{D}_M(G,X,d)\leq \max_{1\leq j\leq m}\limits \overline{D}_M(G,K_j,\{F_n\},d).$$

$(3)$. Let $\{L_n\}$ and $\{F_n\}:=\{RL_n\}$ be  the F\o lner sequences  of $H$  and  $G$, respectively. By the proof of Proposition  \ref{prop 3.9}, for any $s>0$ one has
$${\overline{\rm  mdim}}_M(H,X,\{L_n\},s,d)={\overline{\rm  mdim}}_M(G,X,\{F_n\},s,d)\cdot[G:H],$$
which  yields  that  $\overline{D}_M(G,X,\{F_n\},d)= \overline{D}_M(H,X,\{L_n\},d)$.

$(4)$.  By the proof of Proposition \ref{prop 3.14}, for any $s>0$ we have the following inequality: 
\begin{align*}
{\overline{\rm  mdim}}_M(G,X\times Y,s,d_X\times d_Y)
\leq {\overline{\rm  mdim}}_M(G,X,s,d_X)+{\overline{\rm  mdim}}_M(G,Y,s,d_Y).
\end{align*}
Let $s > \max\{\overline{D}_M(G,X,d_{X}),\overline{D}_M(G,Y,d_{Y})\}$. Then  ${\overline{\rm  mdim}}_M(G,X\times Y,s,d_X\times d_Y) = 0$. Thus, we  deduce that
$$\overline{D}_M(G,X\times Y,d_{X} \times d_Y)\leq \max\{\overline{D}_M(G,X,d_{X}),\overline{D}_M(G,Y,d_{Y})\}.$$

Now consider the coordinate projections \( \pi_X \) and \( \pi_Y \) from \( X \times Y \) to \( X \) and to \( Y \), respectively. Then both $\pi_X$ and $ \pi_Y$ are  one-Lipschitz factor maps. By means of (1), we get the reverse inequality:
$$\overline{D}_M(G,X\times Y,d_X\times d_Y)\geq \max\{\overline{D}_M(G,X,d_{X}),\overline{D}_M(G,Y,d_{Y})\}.$$
\end{proof}

\subsubsection{Infinite  entropy dimension of invariant measures}
Inspired  by the definition of infinite entropy dimension, we  use a similar approach  to  define a measure-theoretic version of infinite entropy dimension. 

\begin{df}\label{df 3.44}  In the setting of  Definition \ref{df 3.26}, recall that for every  $\mu \in  \mathrm{co}(E(X,G))$, we  have defined the measure-theoretic $\epsilon$-entropy:
\begin{align*}
F(\mu, \{F_n\}, \epsilon):=\sum_{j=1}^{k}\lambda_jh_{\mu_j}(G,\{F_n\},\epsilon),
\end{align*}
where $h_{\mu}(G,\{F_n\},\epsilon)$ is chosen from the candidate set $ \mathcal{E}$ in Theorem \ref{thm 2.15}.

(a) Fix $\mu \in  M(X,G)$ and $s>0$. We  define  the $s$-upper and lower  measure-theoretic metric mean dimensions of $\mu$ as
\begin{align*}
\overline{\rm {mdim}}_M(\mu,s,\{F_n\},d)&=\sup_{(\mu_{\epsilon})_{\epsilon}\in M_{G}(\mu) }\{\limsup_{\epsilon \to 0}\frac{F(\mu_{\epsilon},\{F_n\}, \epsilon)}{(\logf)^s}\}.\\
\underline{\rm {mdim}}_M(\mu,s,\{F_n\},d)&=\sup_{(\mu_{\epsilon})_{\epsilon}\in M_{G}(\mu) }\{\liminf_{\epsilon \to 0}\frac{F(\mu_{\epsilon},\{F_n\}, \epsilon)}{(\logf)^s}\},
\end{align*}
respectively.

$(b)$ There exist  crucial values  in $s$ such that $\overline{\rm {mdim}}_M(\mu,s,\{F_n\},d),    \underline{\rm {mdim}}_M(\mu,s,\{F_n\},d)$ jumps from $\infty$ to $0$. We respectively  define  the crucial values as the infinite upper and lower  entropy dimensions of $\mu$:
\begin{align*}
\overline{D}_M(\mu,\{F_n\},d)&=\inf\{s>0:  \overline{\rm {mdim}}_M(\mu,s,\{F_n\},d)=0\},\\
&=\sup\{s>0: \overline{\rm {mdim}}_M(\mu,s,\{F_n\},d)=\infty\}.\\
\underline{D}_M(\mu,\{F_n\},d)&=\inf\{s>0:  \underline{\rm {mdim}}_M(\mu,s,\{F_n\},d)=0\},\\
&=\sup\{s>0: \underline{\rm {mdim}}_M(\mu,s,\{F_n\},d)=\infty\}.
\end{align*}
\end{df}

We remark that the infinite entropy dimension of invariant measures is independent of the choice of the (tempered) Følner sequences of \( G \) and the candidates \( h_{\mu}(G, \{F_n\}, \epsilon) \in \mathcal{E} \), but depends on the compatible metrics on \( X \). 

Analogous to the measure-theoretic metric mean dimension, the following proposition shows that the infinite  entropy dimensions are also upper semi-continuous functions on $M(X,G)$.

\begin{prop}\label{prop 5.8}
	Let $(X, d, G)$ be a $G$-system, and let $\{F_n\}$ be a tempered F\o lner sequence of $G$.  Then
	\begin{align*}
		&\overline{D}_M(\cdot,\{F_n\},d): M(X,G)\rightarrow [0,+\infty],\\
		&\underline{D}_M(\cdot,\{F_n\},d): M(X,G)\rightarrow [0,+\infty]
	\end{align*}  
	are upper semi-continuous.
\end{prop}

\begin{proof}
 Fix $\mu \in M(X,G)$, and let $\{\mu_n\}$ be a sequence of invariant measures converging to $\mu$ in the weak$^{*}$-topology.  Without loss of generality, we may assume that  $\overline{D}_M(\mu,\{F_n\},d)<\infty$. Let $\alpha> \overline{D}_M(\mu,\{F_n\},d)$. Then $\overline{\rm {mdim}}_M(\mu,\alpha,\{F_n\},d)=0$. Similar to the proof of Proposition \ref{prop 3.26}, for every $s >0$, the function  $\overline{\rm {mdim}}_M(\mu,s,\{F_n\},d): M(X,G)\rightarrow [0,+\infty]$  is  upper semi-continuous. Thus,  $\lim_{n \to \infty}\overline{\rm {mdim}}_M(\mu_n,\alpha,\{F_n\},d)=0$. For every  $\gamma >0$, we have  $\overline{\rm {mdim}}_M(\mu_n,\alpha,\{F_n\},d)<\gamma$ for sufficiently large $n$, and hence for such  $n$,  we have  $\overline{D}_M(\mu_n,\{F_n\},d)\leq \alpha$. This shows that  $\limsup_{n \to \infty}\overline{D}_M(\mu_n,\{F_n\},d)\leq \overline{D}_M(\mu,\{F_n\},d) $ by letting $\alpha \to  \overline{D}_M(\mu,\{F_n\},d)$,  i.e.,   $\overline{D}_M(\cdot,\{F_n\},d)$ is u.s.c. at $\mu$.

The same proof works for the function $\underline{D}_M(\cdot,\{F_n\},d): M(X,G)\rightarrow [0,+\infty]$.
\end{proof}
 
\subsubsection{Variational principle for infinite entropy dimensions}

For $\mathbb{Z}$-actions with zero entropy, the known variational principles for zero entropy-like quantities are as follows:
\begin{itemize}
	\item [(a)] Goodman proved that the supremum of measure-theoretic sequence entropy over the set of invariant measures is less than or equal to the topological sequence entropy of the phase space \cite[Theorem 3.1]{g0074}; he also constructed a counterexample to show that the strict inequality is possible for certain dynamical systems \cite[Section 5]{g0074}.
	\item [(b)] Analogous to topological sequence entropy, Carvalho \cite{c97} proved that the supremum of measure-theoretic entropy dimension over the set of invariant measures is less than or equal to the topological entropy dimension of the phase space. Later, the authors in \cite{adp10} revealed  that  the strict inequality is also possible for certain dynamical systems.
\end{itemize}

Thus, the variational principles between the above two pairs of topological and measure-theoretic zero entropy-like quantities are generally not valid. However, for every infinite-entropy $G$-system with zero metric mean dimension, the topological and measure-theoretic infinite entropy dimensions are related by the following variational principles.

\begin{thm}[=Theorem \ref{thm 1.3}]
Let $(X, d, G)$ be a $G$-system, and let  $\{F_n\}$ be a tempered  F\o lner sequence  of $G$. Then
\begin{align*}
\overline{D}_M(G,X,d)&=\max_{\mu \in M(X,G)}\overline{D}_M(\mu,\{F_n\},d).
\end{align*}
\end{thm}

\begin{proof}
Similar to the proof of Theorem \ref{thm 3.30}, for any $s>0$ we have
$${\overline{\rm  mdim}}_M(G,X,s,d)=\max_{\mu \in M(X,G)}\overline{\rm {mdim}}_M(\mu,s,\{F_n\},d).$$
Let $0<s < \overline{D}_M(G,X,d)$. Then there exists $\mu \in M(X,G)$ such that  $\overline{\rm {mdim}}_M(\mu,s,\{F_n\},d)={\overline{\rm  mdim}}_M(G,X,s,d)=\infty$. This yields that $$s \leq \overline{D}_M(\mu,\{F_n\},d)\leq \sup_{\mu \in M(X,G)}\overline{D}_M(\mu,\{F_n\},d).$$
Letting $s \to  \overline{D}_M(G,X,d)$, we have
$$\overline{D}_M(G,X,d)\leq \sup_{\mu \in M(X,G)}\overline{D}_M(\mu,\{F_n\},d).$$
The reverse inequality $\overline{D}_M(G,X,d)\geq \sup_{\mu \in M(X,G)}\overline{D}_M(\mu,\{F_n\},d)$  is obtained similarly. 

It is well-known that for an extended-valued upper semi-continuous function on a compact metric space, there exists a point attaining the supremum of the function over the space. Then,  by Proposition \ref{prop 5.8}  the supremum in the variational principle for infinite entropy dimension can be attained by some invariant measure of $X$.
\end{proof}

\subsection{Infinite entropy systems with positive metric mean dimension} \label{subsec 5.2}
In this subsection,  we further investigate the topological structures of $G$-systems with positive metric mean dimension from a local viewpoint, and prove Theorem \ref{thm 1.4}.

\subsubsection{Linking   local metric mean dimension and  metric mean dimension}

We start by  introducing the following  definition.

\begin{df}\label{df 4.9}
Let $(X, d, G)$ be a $G$-system and  $\{F_n\}$ be a F\o lner sequence of $G$.  For $x\in X$, we respectively define the   local  upper and lower metric mean dimensions  of $x$ as
\begin{align*}
\overline{\rm mdim}_M(x,\{F_n\},d)=\inf_K\{\overline{\rm mdim}_M(G,K,\{F_n\},d)\},\\
\underline{\rm mdim}_M(x,\{F_n\},d)=\inf_K\{\underline{\rm mdim}_M(G,K,\{F_n\},d)\},
\end{align*}
where the infimum  ranges over all closed neighborhoods $K$ of $x$.
\end{df}

The following   variational principles  establish the  precise  relations between the  metric mean dimension  and the local metric mean dimension in a topological and measure-theoretic manners.

\begin{thm}[=Theorem \ref{thm 1.4}]\label{thm 3.49}
Let $(X, d, G)$ be a $G$-system.  Then for every tempered  two-sided F\o lner sequence  $\{F_n\}$ of $G$,
\begin{align*}
\overline{\rm mdim}_M(G,X,d)
&=\max_{\mu \in M(X,G)}\int \overline{\rm mdim}_M(x,\{F_n\},d) d\mu,\\
&=\max_{x\in X}  \overline{\rm mdim}_M(x,\{F_n\},d).
\end{align*}
\end{thm}

We need the following lemma for our proof.

\begin{lem}\label{lem 3.12}
Let $(X, d, G)$ be a $G$-system, and  let  $\{F_n\}$ be a  two-sided F\o lner sequence of $G$. If $\mu \in E(X,G)$ and $A\subset X$ is a  Borel  set with $\mu(A)>0$, then for any $\epsilon >0$,
$$\overline{h}_{\mu}^K(\{F_n\},\epsilon)\leq h_{top}(G,A,d,\{F_n\},\frac{\epsilon}{8}).$$
\end{lem}

\begin{proof}

Fix $\delta \in (0,1)$. Since $\mu$ is ergodic and $G$ is countable, there exists a finite subset $S\subset G$ containing $e_{G}$ satisfying
$\mu(\cup_{g\in S}~gA)>1-\delta.$ For sufficiently large $n$, we assume that $E$ is a $(d_{F_n},\epsilon)$-separated set of $\cup_{g\in S}~ gA$ with the largest cardinality. Then, by Pigeonhole Principle, there is a $g_0\in S$ such that   $\#E_0 \geq \frac{\#E}{|S|}$ for some subset $E_0\subset  g_0 A\cap E.$
Therefore, $g_0^{-1}E_0 \subset A$ and  is a $(d_{F_nS},\epsilon)$-separated set of $A$. So
\begin{align*}
\limsup_{n \to \infty}\frac{\log  \frac{1}{|S|} s(\cup_{g\in S}~gA,d_{F_n},\epsilon)}{|F_n|}&\leq  \limsup_{n \to \infty}\frac{\log  s(A,d_{F_n S},\epsilon)}{|F_n|}\\
&\leq  \limsup_{n \to \infty}\frac{\log  r(A,d_{F_n S},\frac{\epsilon}{2})}{|F_n|}\\
&\leq  \limsup_{n \to \infty}\frac{\log  \left(r(A,d_{F_n},\frac{\epsilon}{4})\cdot r(X,d,\frac{\epsilon}{4})^{|F_nS\backslash F_n|}\right)}{|F_n|}.
\end{align*}
Notice that $\{F_n\}$ is a  two-sided F\o lner sequence of $G$. This means that $\lim_{n \to \infty}\frac{|F_nS \backslash F_n|}{|F_n|}=\lim_{n \to \infty}\frac{|SF_n \backslash F_n|}{|F_n|}=0$. Therefore, we have 
$$h_{top}(G,\cup_{g\in S}~gA,d,\{F_n\},\epsilon)\leq {h}_{top}(G,A,d,\{F_n\},\frac{\epsilon}{4}).$$
Let $F$ be a $(d_{F_n},\epsilon)$-separated set of $ \cup_{g\in S}~gA$ with the largest cardinality. Then $\mu(\cup_{x\in F} B_{F_n}(x,2\epsilon))>1-\delta$.   This implies that
$\overline{h}_{\mu}^K(\{F_n\},2\epsilon) \leq  h_{top}(G,A,d,\{F_n\},\frac{\epsilon}{4}).$ This completes the proof by the arbitrariness of $\epsilon$.
\end{proof}

\begin{proof}[Proof of Theorem \ref{thm 3.49}]

For every $0<q\leq\infty$,   the set  $$X_q:=\{x\in X: 0\leq \overline{\rm mdim}_M(x,\{F_n\},d)<q\}$$ is open.  
Thus, the function $\overline{\rm mdim}_M(x,\{F_n\},d): x\in X\mapsto \mathbb{R}_{\geq 0} \cup \{\infty\}$ is  Borel measurable. By definitions, it is  obvious that $$\overline{\rm mdim}_M(G,X,d) \geq \sup_{\mu \in M(X,G)}\int \overline{\rm mdim}_M(x,\{F_n\},d) d\mu.$$ 

By the inequality (\ref{equ 3.6}) presented in  the proof Theorem \ref{thm 1.1}, there exist $\mu_0 \in M(X,G)$ and  $\nu_\epsilon \in E(X,G)$  for every sufficiently small $\epsilon >0$ such that $\nu_\epsilon \to \mu_0$ as $\epsilon \to 0$  and  $${\overline{\rm  mdim}}_M(G,X,d)=\overline{\rm {mdim}}_M(\mu_0,\{F_n\},d)=\limsup_{\epsilon \to 0}   \frac{1}{\log \frac{1}{\epsilon}} \inf_{ 
\diam (\alpha)\leq \frac{\epsilon}{64}}h_{\nu_{\epsilon}}(G,\alpha).$$
	
By (\ref{equu 3.15}), we know that
\begin{align}\label{equu 5.1}
{\overline{\rm  mdim}}_M(G,X,d)=\limsup_{\epsilon \to 0}   \frac{1}{\log \frac{1}{\epsilon}} \overline{h}_{\nu_\epsilon}^K(\{F_n\},\frac{\epsilon}{64}).
\end{align}	 Let $x_0$ be a point of  the support  $\text{supp}(\mu_0)$ of $\mu_0$. Fix an open neighborhood $O$    of $x_0$.  Then   $\nu_{\epsilon}(O)>0$ for every sufficiently small $\epsilon >0$.   Letting  $h_{\mu}(G,\{F_n\},\epsilon)= \overline{h}_{\mu}^K(\{F_n\},\epsilon)$, by Lemma  \ref{lem 3.12}, for some $C>0$ one has
\begin{align}\label{inequ 4.2}
F(\nu_{\epsilon},\{F_n\}, \frac{\epsilon}{64})\leq h_{top}(G,O,d,\{F_n\},C\cdot \epsilon),
\end{align}
for all sufficiently large $\epsilon >0$. Together with  (\ref{equu 5.1}), 
this implies that $$
{\overline{\rm  mdim}}_M(G,X,d)\leq   \overline{\rm mdim}_M(x_0,\{F_n\},d),$$ and hence
\begin{align}\label{inequ 4.1}
{\overline{\rm  mdim}}_M(G,X,d)=\overline{\rm {mdim}}_M(\mu_0,\{F_n\},d)\leq  \overline{\rm mdim}_M(x,\{F_n\},d)
\end{align} 
for all  $x\in \text{supp}(\mu_0)$.  Therefore, we get $${\overline{\rm  mdim}}_M(G,X,d)\leq \int  \overline{\rm mdim}_M(x,\{F_n\},d) d\mu_0,$$
and $\mu_0$  realizes the supremum of the variational principle.

$(2)$. It is clear that $$\sup_{x\in X}  \overline{\rm mdim}_M(x,\{F_n\},d) \leq  \overline{\rm mdim}_M(G,X,d).$$ Assume the above inequality is strict. Then for every $x\in X$, there is an open neighborhood $O_x$ of $x$ such that $\overline{\rm mdim}_M(G,O_x,\{F_n\},d)< \overline{\rm mdim}_M(G,X,d).$
Now cover $X$ with a  finite subcover $\{O_{x_j}\}_{j=1}^m$ of $\{O_x\}_{x\in X}$. Then
\begin{align}\label{ineq 3.3}
\overline{\rm mdim}_M(G,X,d)=\max_{1\leq j\leq m} \overline{\rm mdim}_M(G,O_{x_j},\{F_n\},d) < \overline{\rm mdim}_M(G,X,d),
\end{align}
a contradiction. This shows that
$$\sup_{x\in X} \overline{\rm mdim}_M(x,\{F_n\},d) =  \overline{\rm mdim}_M(G,X,d).$$
Furthermore, the same reasoning yields that the supremum can be attained for some $x\in X$.

\end{proof}


\subsubsection{Application to (full) metric mean dimension point}

Next, we characterize the topological structure of the set of points that attain the maxima in Theorem \ref{thm 1.4}.

\begin{df}\label{df 3.51}
Let $(X, d, G)$ be a $G$-system and $\{F_n\}$  be  a F\o lner sequence  of $G$. One says that
\begin{itemize}
\item [(1)]  $x\in X$  is a  metric mean dimension point if $\overline{\rm mdim}_M(x,\{F_n\},d)>0$;
\item [(2)]   $x\in X$  is a  full metric mean dimension point if  $${\overline{\rm  mdim}}_M(G,X,d)=\overline{\rm mdim}_M(x,\{F_n\},d)>0.$$
\end{itemize}
\end{df}

Denote by $ME_p(X,G)$ and $ME_p^f(X,G)$ the sets of metric mean dimension points and of  full metric mean dimension points, respectively. It follows from the Definitions \ref{df 3.51} that a  metric mean dimension point  $x$ means that all closed neighborhoods $K$ of $x$ have positive metric  mean dimension, and  a full metric mean dimension point  $x$ means that all  closed neighborhoods $K$ of $x$ not only have positive metric  mean dimension, but also  have full metric mean dimension. The Definition \ref{df 3.51} is highly inspired by the concepts of \emph{entropy point and full entropy point} of $\mathbb{Z}$-actions introduced by Ye and Zhang \cite{yz07}.  

Set $$h_{top}(x,\{F_n\})=\inf_{K}h_{top}(G,K,\{F_n\}),$$
where $K$ ranges over all  closed neighborhoods $K$ of $x$.

A point $x$ is an \emph{entropy point} if  $h_{top}(x,\{F_n\})>0$, and  is a \emph{full entropy point} if  $h_{top}(G,X)=h_{top}(x,\{F_n\})>0$.  Denote the sets of entropy points and full entropy points by \( E_p(X,G) \) and \( E_p^f(X,G) \), respectively.

For  $G$-systems with positive metric mean dimension, the following theorem  characterizes the structures of  the aforementioned  four types of sets, and gives a quantitative relationship  involving  the metric mean dimensions of these sets and the whole phase space.

\begin{thm}\label{thm 3.52}
 Let $(X, d, G)$ be a $G$-system such that ${\overline{\rm  mdim}}_M(G,X,d)>0$, and let $\{F_n\}$ be a tempered  two-sided F\o lner sequence   of $G$.
\begin{itemize}
\item [(1)]  The sets $ME_p(X,G)$ and $ME_p^f(X,G)$ are non-empty $G$-invariant closed subsets of $X$. In particular, $ ME_p^f(X,G)= ME_p(X,G)=X$ if $(X,G)$ is minimal.
\item [(2)] There exists $\mu \in  M_{max}(G,X,d)$ such that $$\text{supp}(\mu)\subset ME_p^f(X,G) \subset  ME_p(X,G).$$ 
\item [(3)] It holds that $ME_p^f(X,G) \subset ME_p(X,G)\subset E_p^f(X,G) \subset E_p(X,G) $.
\item [(4)] If there exists $\mu \in E(X,G)$ such that ${\overline{\rm  mdim}}_M(G,X,d)=\limsup_{\epsilon \to 0}\frac{1}{\logf}h_{\mu}(G,\{F_n\},\epsilon)$ for some  measure-theoretic $\epsilon$-entropy $h_{\mu}(G,\{F_n\},\epsilon)$ in  Theorem  \ref{thm 2.15},  then
\begin{align*}
{\overline{\rm  mdim}}_M(G,X,d)&=\max_{\mu \in M_{max}(G,X,d)} {\overline{\rm  mdim}}_M(G,  \text{supp}(\mu),d)\\
&={\overline{\rm  mdim}}_M(G, A,d)
\end{align*}
for every 
$A\in \{ME_p^f(X,G),ME_p(X,G),E_p^f(X,G),E_p(X,G)\}$.
\end{itemize}
\end{thm}

\begin{proof}
$(1)$. By Theorem \ref{thm 1.4}, $ME_p(X,G)$ and $ME_p^f(X,G)$ are non-empty closed sets of $X$, and are $G$-invariant  since  $\overline{\rm mdim}_M(gx,\{F_n\},d)=\overline{\rm mdim}_M(x,\{F_n\},d)$ for all $g\in G$ and $x\in X$.

$(2)$.  It is due to the inequality (\ref{inequ 4.1}).

$(3)$. It directly follows from
the fact that if $A$ is a non-empty subset of $X$ such that ${\overline{\rm  mdim}}_M(G,A,\{F_n\},d)>0$, then $h_{top}(G,A,\{F_n\})=\infty$.

$(4)$. Assume that $\mu \in E(X,G)$ such that ${\overline{\rm  mdim}}_M(G,X,d)=\limsup_{\epsilon \to 0}\frac{1}{\logf}h_{\mu}(G,\{F_n\},\epsilon)$. Then, by Theorem \ref{thm 2.15} and Definition \ref{df 3.26} we have
$${\overline{\rm  mdim}}_M(G,X,d)=\limsup_{\epsilon \to 0}\frac{1}{\logf}\overline{h}_{\mu}^K(\{F_n\},\epsilon)=\overline{\rm {mdim}}_M(\mu,\{F_n\},d).$$ 
For such $\mu$, the inequality (\ref{inequ 4.1})  tells us $\text{supp}(\mu)\subset ME_p^f(X,G)$.
Together with Lemma  \ref{lem 3.12},  we have 
\begin{align*}
{\overline{\rm  mdim}}_M(G,X,d)&\leq {\overline{\rm  mdim}}_M(G,  \text{supp}(\mu),d)\\
&\leq {\overline{\rm  mdim}}_M(G, ME_p^f(X,G),d)\\
&\leq {\overline{\rm  mdim}}_M(G, X,d).
\end{align*} 
By the relations presented in (3), we get desired equalities.
\end{proof}

\begin{rem}
$(i)$. For every infinite entropy system $(X,d,G)$ with zero metric mean dimension (e.g. taking $X=\{0,\frac{1}{2},\frac{1}{2^2},...\}$, $d=|\cdot|$ and $D$ as in   Example \ref{ex 3.6}), we have $ME_p^f(X,G)=\emptyset$, but $ E_p^f(X,G)\not=\emptyset$.

$(ii)$. There are  certain  $G$-systems satisfying  the assumption in  Theorem \ref{thm 3.52}, (4).  By the proof of Example  \ref{ex 3.28}, for the full shift over $[0,1]^m$ and   $\mu=(\text{Leb}([0,1]^{m}))^{\otimes G}$, we have
\begin{align*}
{\overline{\rm  mdim}}_M(G,([0,1]^m)^G,d)=&\overline{\rm {mdim}}_M(\mu,\{F_n\},d)\\
=&\limsup_{\epsilon \to 0}\frac{1}{\logf}\inf_{\diam \alpha \leq \epsilon}h_{\mu}(G,\alpha) = m.
\end{align*}
However,  it is unclear  whether these additional conditions can be removed. 
\end{rem}


\section*{Acknowledgement} 
 
The joint work with Prof. Xiaoyao Zhou  was  completed  when Rui Yang was staying at  the Mathematical Center of Chongqing University. The authors would like to thank Prof. Hanfeng Li for a careful reading and many  valuable comments and suggestions that   lead to a  great improvement of this manuscript.  The first author was  supported by  the China Postdoctoral Science Foundation (No. 2024M763856) and  the Postdoctoral Fellowship Program of CPSF  (No. GZC20252040).  
The  second author   was  supported by the National Natural Science Foundation of China (No. 11971236) and Qinglan project of Jiangsu Province.  






\begin{thebibliography}{HD82}

\normalsize
\baselineskip=15pt
\bibitem[AKM65]{akm65} R. Adler, A. Konheim and M. McAndrew,   Topological entropy,  \emph {Trans. Amer. Math. Soc.} \textbf{114} (1965), 309-319.


\bibitem[ADP10]{adp10} Y. Ahn, D. Dou and K. Park, Entropy dimension and variational principle, \emph {Studia Math.} \textbf{199} (2010), 295-309.

\bibitem[Bow71]{bow71}R. Bowen, Entropy for group endomorphisms and homogeneous spaces,  \emph {Trans. Amer. Math. Soc.}  \textbf{153} (1971), 401-414.

\bibitem[BS25]{bs25}D. Burguet and R. Shi, Topological mean dimension of induced systems,  \emph {Trans. Amer. Math. Soc.}  \textbf{378} (2025), 3085-3103.


\bibitem[Ca97]{c97} M. Carvalho, Entropy dimension of dynamical systems, \emph {Portugal. Math.} \textbf{54} (1997), 19-40.

\bibitem[CPV24]{cpv24} M. Carvalho, G.  Pessil and P. Varandas, A convex analysis approach to the metric mean dimension: limits of scaled pressures and variational principles, \emph{Adv. Math.} \textbf{436} (2024), Paper No. 109407, 54 pp.



\bibitem[CDZ22]{cdz22} E. Chen, D. Dou and D. Zheng,  Variational principles for amenable metric mean dimensions, \emph {J. Diff.  Equ.} \textbf{319} (2022), 41-79.












\bibitem[Coo15]{coo15}  M. Coornaert, Topological dimension and dynamical systems, Universitext
Springer, Cham, 2015.

\bibitem[CT06]{ct06} T. Cover and J. Thomas, Elements of information theory, second edition, Wiley, New York, 2006.





\bibitem[DZ15]{dz15} A. Dooley and G. Zhang, Local entropy theory of a random dynamical system, \emph{Mem. Amer. Math. Soc.} \textbf{233} (2015), no. 1099, vi+106 pp.






\bibitem[DS88]{ds88} N.  Dunford and J. T. Schwartz,  Linear operators, Part 1,  John Wiley \& Sons, 1988.

\bibitem[GG25]{gg25}  F. Garc\'{i}a-Ramos and  Y. Gutman, Local mean dimension theory for sofic group actions, \emph {Groups Geom. Dyn.}, 2025. DOI: 10.4171/GGD/862. 

\bibitem[Goo74]{g0074} T. Goodman, Topological sequence entropy,  \emph {Proc. London Math. Soc.} \textbf{29} (1974), 331-350.

\bibitem[Gra11]{gra 11}R. Gray, Entropy and information theory, Springer, New York, Second edition, 2011.

\bibitem[Gro99]{gromov} M. Gromov,  Topological invariants of dynamical systems and spaces of holomorphic maps: I, \emph{Math. Phys, Anal. Geom.} \textbf{4} (1999), 323-415.

\bibitem[Gut15]{g15} Y. Gutman, Mean dimension and Jaworski-type theorems, \emph{Proc. Lond. Math. Soc.} \textbf{111} (2015),  831-850.

\bibitem[Gut17]{g17} Y. Gutman, Embedding topological dynamical systems with periodic points in cubical shifts, \emph{Ergodic Theory Dynam. Syst.} \textbf{37} (2017), 512-538.

\bibitem[GLT16]{glt16} Y. Gutman, E. Lindenstrauss and M. Tsukamoto, Mean dimension of $\mathbb{Z}^{k}$ actions, \emph{Geom. Funct. Anal.} \textbf{26} (2016),  778-817.

\bibitem[GS20]{gs} Y. Gutman and A. \'Spiewak,  Metric mean dimension and analog compression, \emph{IEEE Trans. Inform. Theory} \textbf{66} (2020), 6977-6998.

\bibitem[GS21]{gs20} Y. Gutman and A. \'Spiewak,  Around the variational principle for metric mean dimension,  \emph {Studia Math.} \textbf{261} (2021), 345-360.

\bibitem[GT20]{gt20}  Y. Gutman and M. Tsukamoto, Embedding minimal dynamical systems into Hilbert cubes, \emph {Invent. Math.} \textbf{221} (2020), 113-166.






\bibitem[H08]{h08} W. Huang, Stable sets and   $\epsilon$-stable sets in positive-entropy systems, \emph {Comm. Math. Phys.} \textbf{279} (2008), 535-557.

\bibitem[HYZ11]{hyz11} W. Huang, X. Ye and G. Zhang,  Local entropy theory for a countable discrete amenable
group action, \emph{J. Funct. Anal.} \textbf{261} (2011), 1028-1082.

\bibitem[JQ24]{jq24} L. Jin and Y. Qiao, Mean dimension of product spaces: a fundamental formula, \emph{Math. Ann.} \textbf{388} (2024), 249-259.

\bibitem[Kat80]{k80} A. Katok,  Lyapunov exponents, entropy and periodic orbits for diffeomorphisms, \emph{Publ. Math. Inst. Hautes \'{E}tudes Sci.} \textbf{51} (1980), 137-173.

\bibitem[KD94]{kd94} T. Kawabata  and A. Dembo,  The rate-distortion dimension of sets and measures, \emph{IEEE Trans. Inform. Theory} \textbf{40} (1994), 1564-1572.

\bibitem[KL11]{kl11} D. Kerr  and H. Li,  Entropy and the variational principle for actions of sofic groups, \emph{Invent. Math.} \textbf{186} (2011), 501-558.

\bibitem[KL16]{kl16} D. Kerr and H. Li, Ergodic theory, Independence and dichotomies,  \emph{Springer Monographs in Mathematics. Springer, Cham.},  2016.

\bibitem[Kol58]{kol58} A.  Kolmogorov, A new metric invariant of transient dynamical systems and automorphisms
of Lebesgue spaces,  \emph{Dokl. Akad. Sci. SSSR} \textbf{119} (1958), 861-864.


\bibitem[Li13]{l13}  H. Li, Sofic mean dimension, \emph{Adv. Math. } \textbf{244} (2013), 570-604. 

\bibitem[Li21]{l21} Z. Li,   Amenable upper mean dimensions, \emph{Anal. Math. Phys.} \textbf{11} (2021), Paper No. 99, 12 pp. 

\bibitem[LL19]{ll19} H. Li and B. Liang, Sofic mean length, \emph{Adv. Math.} \textbf{353} (2019), 802-858.

\bibitem[LY12]{ly12} B. Liang and K. Yan,  Topological pressure for sub-additive potentials of amenable group actions, \emph{J. Funct. Anal.} \textbf{262} (2012), 584-601.


\bibitem[LJZZ25]{ljzz25} J. Li, Y. Ji, T. Zhan and Y. Zhang, Variational principle of metric mean dimension and rate distortion dimension for amenable group actions, \emph{J. Math. Anal. Appl.} \textbf{546} (2025), no. 1, Paper No. 129282, 16 pp.

\bibitem[Lin01]{lin01} E.  Lindenstrauss, Pointwise theorems for amenable groups, \emph{Invent Math.} \textbf{146} (2001),  259-295.

\bibitem[Lin99]{l99} E. Lindenstrauss,  Mean dimension, small entropy factors and an embedding theorem, \emph{Publ. Math. Inst. Hautes \'Etudes Sci.} \textbf{89} (1999), 227-262.

\bibitem[LT18]{lt18} E. Lindenstrauss  and M. Tsukamoto, From rate distortion theory to metric mean dimension: variational principle, \emph{IEEE Trans. Inform. Theory} \textbf{64} (2018), 3590-3609.

\bibitem[LW00]{lw00} E. Lindenstrauss and B. Weiss,  Mean topological dimension, \emph{Israel J. Math.} \textbf{115} (2000), 1-24. 



\bibitem[Mis76]{mis75} M. Misiurewicz,  A short proof of the variational principle for a $\mathbb {Z} _ {+}^{N} $ action on a compact space, \emph{Ast\'{e}risque} \textbf{40} (1976), 227-262.

\bibitem[OP82]{op82} J. Ollagnier and D. Pinchon,  The variational principle, \emph{Studia Math.} \textbf{72} (1982), 151-159.

\bibitem[OW87]{ow87}D. Ornstein and B. Weiss,  Entropy and isomorphism theorems for actions of amenable groups, \emph{J. Anal. Math.} \textbf{48} (1987), 1-141.


\bibitem[Pes97]{p97} Y.B. Pesin, Dimension theory in dynamical systems, University of Chicago Press, 1997.


\bibitem[RTZ23]{rtz23} X. Ren, X.Tian and Y. Zhou, On the topological entropy of saturated sets for amenable group actions, \emph{J. Dynam. Diff.  Equ.} {\bf 35} (2023), 2873-2904.

\bibitem[RS12]{rs12} J. Robinson and N. Sharples, Strict inequality in the box-counting dimension product formulas, \emph{Real Anal. Exchange} {\bf 38} (2012), 95-119.

\bibitem[Rue04]{r04} D. Ruelle, Thermodynamic formalism: the mathematical structure of equilibrium statistical mechanics, Cambridge University Press, 2004.


\bibitem[S22]{shi} R. Shi, On variational principles for metric mean dimension, \emph{IEEE Trans. Inform. Theory} \textbf{68} (2022), 4282-4288.

\bibitem[S59]{s59} Y. Sinai,  On the concept of entropy for a dynamical system, \emph{Dokl. Akad. Nauk SSSR} \textbf{124} (1959), 768-771.



\bibitem[ST80]{st80} A. Stepin, and A.Tagi-Zade, Variational characterization of topological pressure of the amenable groups of transformations, \emph{Dokl. Akad. Nauk SSSR} \textbf{254} (1980), 545-549.


\bibitem[TWL20]{twl20} D. Tang, H. Wu and Z. Li, Weighted upper metric mean dimension for amenable group actions, \emph{Dyn. Syst.} \textbf{35} (2020), 382-397.


\bibitem[Tsu19]{t19} M. Tsukamoto, Mean dimension of full shifts, \emph{Israel J. Math.} \textbf{230} (2019), 183-193.

\bibitem[Tsu20]{t20} M. Tsukamoto, Double variational principle for mean dimension with potential, \emph{Adv. Math.} \textbf{361} (2020), 106935, 53 pp. 



\bibitem[VV17]{vv17}  A. Velozo and R. Velozo, Rate distortion theory, metric mean dimension and measure theoretic entropy, arXiv:1707.05762.



\bibitem[Wal75]{wal75} P. Walters, A variational principle for the pressure of continuous transformations, \emph{Amer. J. Math.} \textbf{97} (1975),  937-971. 

\bibitem[Wal82]{w82} P. Walters, An introduction to ergodic theory,  Springer-Verlag, New York, 1982.

\bibitem[WWW16]{www16} C. Wei, S. Wen and  Z. Wen, Remarks on dimensions of Cartesian product sets, \emph{Fractals} \textbf{24} (2016), no. 3,  8 pp.


\bibitem[YCZ22]{ycz22b} R. Yang, E. Chen and X. Zhou, \emph{On variational principle for upper metric mean dimension with potential},  arXiv:2207.01901. 


\bibitem[YCZ25]{ycz25} R. Yang, E. Chen and X. Zhou,  
Measure-theoretic metric mean dimension, \emph{Studia Math.} \textbf{40} (2025),  1-25.

\bibitem[Y25]{y25} R. Yang,  
Mean dimension and rate-distortion function revisited (preprint), arXiv:2510.08051.

\bibitem[YZ07]{yz07}X. Ye and G. Zhang,  Entropy points and applications,  \emph{Trans. Amer. Math. Soc.} \textbf{359}(2007), 6167-6186.



\bibitem[ZCY16]{zcy16} D. Zheng,  E. Chen and J. Yang,  On large deviations for amenable group actions, \emph{Discrete Contin. Dyn. Syst.} \textbf{36} (2016),  7191-7206.

\end{thebibliography}

\end{document}